\documentclass[11pt]{amsart}
\usepackage[latin1]{inputenc}
\usepackage{amsmath}
\usepackage{amssymb}
\usepackage[all]{xy}
\usepackage{mathrsfs}
\usepackage{amsthm}
\usepackage{graphicx}
\usepackage{stmaryrd}

\theoremstyle{plain}
\newtheorem{thm}[equation]{Theorem}
\newtheorem{lem}[equation]{Lemma}
\newtheorem{prop}[equation]{Proposition}
\newtheorem{cor}[equation]{Corollary}

\theoremstyle{definition}

\theoremstyle{remark}
\newtheorem{remark}[equation]{Remark}

\numberwithin{equation}{subsection}

\def\sheafHom{\mathcal{H} \hspace{-1pt} \mathit{om}}

\newcommand{\nc}{\newcommand}
\newcommand{\rc}{\renewcommand}

\nc{\Aut}{{	\operatorname{Aut}	}}
\nc{\codim}{{	\operatorname{codim}	}}
\nc{\Ob}{{	\operatorname{Ob}	}}
\nc{\PGL}{{	\operatorname{PGL}	}}
\nc{\supp}{{	\operatorname{supp}	}}
\nc{\tr}{{	\operatorname{tr}	}}

\nc{\Rep}{{	{\cal{R}}ep		}}
\nc{\one}{{	\mbox{\bf{1}}		}}
\nc{\iso}{	\overset{\sim}{\lra}	}

\nc{\nen}{\newenvironment}

\nc{\pr}{\protect}

\nc{\nn}{{\newline}}
\nc{\np}{{\newpage}}	

\nc{\lab}{	\label}
\nc{\npp}{{	\newpage\setcounter{page}{0}	}}
\nc{\setpa}{		\setcounter{part}		}
\nc{\setse}{		\setcounter{section}	}
\nc{\setsus}{		\setcounter{subsection}		}
\nc{\setsss}{		\setcounter{subsubsection}	}
\nc{\setpage}{		\setcounter{page}	}

\nc{\nfd}{ $$\text{ This version is preliminary and approximate, 
		             it is not for distribution. }$$	}

\nc{\noi}{{\noindent}}
\nc{\pf}{{	\noindent {\em Proof.}		}}
					
\nc{\epf}{ \fbox{\bf QED}	}

\nc{\heart}{{\tiny \cen{\tiny $\heartsuit $ }	}} 
\nc{\cont}{\tableofcontents}

\nc{\sbr}{{	\smallpagebreak	}}
\nc{\mbr}{{	\medpagebreak	}}
\nc{\bbr}{{	\bigpagebreak	}}
\nc{\bbb}{ 	\boldsymbol 	}  

\nc{\bib}{		}

\rc{\b}{ 	\big         			}  
\nc{\lam}[1]{{ 	\text{\large $#1$	}	}}  
\nc{\smm}[1]{{ 	\text{\small $#1$	}	}}  
\nc{\fom}[1]{{ 	\text{\footnotesize $#1$	}	}}  
\nc{\tinm}[1]{{ \text{\tiny $#1$	}	}}  

\nc{\bu}{ \bullet         }  			
\nc{\bbu}{ \aa{\bbb \bullet}         }  	
\nc{\bus}{{	^\bullet	}}	 	
\nc{\bui}{{	_\bullet	}}	 	

\nc{\bem}{{	\begin{em}	}}
\nc{\eem}{{	\end{em} 	}}

\nc{\bbox}{{	\blackbox	}}	
\nc{\bx}{	\boxed	}		
\nc{\tbx}[1]{{\boxed{\tx{#1}}}}		

\nc{\mmbox}[1]{{	\mbox{$#1$}	}}	
\nc{\tbox}[1]{{		\mbox{\tx{#1}}	}}
 	
\nc{\ot}{		\leftarrow			}
\nc{\tto}{		\longrightarrow			}
\nc{\ott}{		\longleftarrow			}
\nc{\too}[1]{{		\aa{#1}\rightarrow			}}
\nc{\oot}[1]{{		\aa{#1}\leftarrow			}}
\nc{\ttoo}[1]{{		\aa{#1}\longrightarrow			}}
\nc{\oott}[1]{{		\aa{#1}\longleftarrow			}}

\nc{\Too}[2]{{		\aa{#1}{\bb{#2}\rightarrow}		}}
\nc{\ooT}[2]{{		\aa{#1}{\bbb{#2}\leftarrow}		}}
\nc{\TToo}[2]{{		\aa{#1}{\bb{#2}\longrightarrow}		}}
\nc{\ooTT}[2]{{		\aa{#1}{\bbb{#2}\longleftarrow}		}}
\nc{\toot}[2]{{		\aa{#1}{\bb{#2}\rightleftarrows}	}}
\nc{\ttoot}[2]{{	\aa{#1}{\bb{#2}\rightleftrightarrows}	}}

\nc{\ra}{{	\rightarrow		}}
\nc{\laa}{{	\leftarrow	}}	
\nc{\lra}{{\longrightarrow}}

\nc{\lr}{{\leftrightarrow}}     	
\nc{\lrs}{{\rightleftarrows}}     	

\nc{\imp}{{\Rightarrow}}        	
\nc{\impp}{{\Leftarrow}}        	
\nc{\eq}{{\Leftrightarrow}}        	
\nc{\impl}{{\Longrightarrow}}        	
\nc{\imppl}{{\Longleftarrow}}        	
\nc{\eql}{{\Longleftrightarrow}}        	
	\nc{\Ra}{{\Rightarrow}}         	
	\nc{\LRa}{{\Leftrightarrow}}        	

\nc{\inj}{{\pr	\hookrightarrow	}}    		
\nc{\injj}{{\pr	\hookleftarrow	}}    		

\nc{\sur}{{	\twoheadrightarrow	}}	
\nc{\surr}{{	\twoheadleftarrow	}}	
			
\nc{\mm}{{	\mapsto		}}     		
\nc{\mmm}{{	\leftarrow\shortmid }}		
\nc{\ainj}[1]{{\aa{#1}{\pr\hookrightarrow}	}}    	
\nc{\ainjj}[1]{{\aa{#1}{\pr\hookleftarrow}	}}    	

\nc{\asur}[1]{{	\aa{#1}\twoheadrightarrow	}}	
\nc{\asurr}[1]{{\aa{#1}\twoheadleftarrow	}}	
			
\nc{\amm}[1]{{	\aa{#1}\mapsto		}}     	
\nc{\ammm}[1]{{	\aa{#1}\leftarrow\shortmid }}	

\nc{\va}{{\uparrow}}              		

\nc{\syp}[1]{	^{ (#1) }		} 	
\nc{\up}[1]{	^{ (#1) }		} 	
\nc{\lp}[1]{	_{ (#1) }		}	
\nc{\hp}[1]{	^{ [#1] }		}	
\nc{\cle}{\preceq}		
\nc{\cl}{\prec}			
\nc{\cge}{\succeq}		
\nc{\cg}{\succ}			

\nc{\bb}{	\pr\underset 	}           
\rc{\aa}{ 	\pr\overset 	}            

\nc{\indd}{{ ${} \ \ \ \ \  \ \        {} $	}}	
\nc{\inddd}{{ 	\indd\indd			}}	
\nc{\nnd}{{ 	\nn  \indd 			}}	
\nc{\nndb}{{ 	\nn  \indd $\bullet$		}}	
		
\nc{\bce}{	\begin{center}	}
\nc{\ece}{	\end  {center}	}
\nc{\cen}[1]{	\begin{center}	{\em  #1}	\end  {center}	}

\nc{\bss}{{\backslash}}           		
\nc{\barr}{ 	\overline 	}      		
\nc{\ud}{	\underline	}		

\nc{\ti}{\tilde}              
\nc{\tii}{\widetilde}         

\nc{\hatt}{\widehat}				
\nc{\hata}{{	\bbb{ \hat{} }		}}	

\nc{\ch}{\check}              			
\nc{\cha}{{ 	\bbb{ \check{} }	}}      

\nc{\sub}{{	\subseteq	}}         
\nc{\subb}{{	\supseteq	}}         
\nc{\nsub}{{	\nsubseteq	}}         
\nc{\nsubb}{{	\nsupseteq	}}         %

\nc{\nin}{{	\notin	}}

\nc{\lb}{\langle}             				
\nc{\rb}{\rangle}

\nc{\lB}{	\left(	}             			
\nc{\rB}{	\right)	}
\nc{\BBl}{{	\bbb{ \left( \right.}	}}             	
\nc{\BBr}{{	\bbb{ \left. \right)}	}}

\nc{\Pa}[2]{ {\lb} #1 {,} #2 {\rb} }				

\nc{\cD}[1]{ \tx{ $$\CD {#1} \endCD $$ }  }		

\nc{\mat} {		\left(		\matrix	}	
\nc{\emat}{		\endmatrix	\right)	}
\nc{\sm} {		\left(		\smallmatrix	}	
\nc{\esm}{		\endsmallmatrix	\right)	}
\nc{\smat} {		\left(		\smallmatrix	}	
\nc{\esmat}{		\endsmallmatrix	\right)	}

\nc{\matr} {		\left[		\matrix	}	
\nc{\ematr}{		\endmatrix	\right]	}
\nc{\smr} {		\left[		\smallmatrix	}	
\nc{\esmr}{		\endsmallmatrix	\right]	}
\nc{\smatr} {		\left[		\smallmatrix	}	
\nc{\esmatr}{		\endsmallmatrix	\right]	}

\nc{\imat} {		\left.		\matrix	}	
\nc{\eimat}{		\endmatrix	\right.	}
\nc{\ism} {		\left.		\smallmatrix	}	
\nc{\eism}{		\endsmallmatrix	\right.	}

\nc{\ca}{		\left\{		\smallmatrix	}	
\nc{\eca}{		\endsmallmatrix	\right\}	}
\nc{\Ca}{		\left\{		\matrix		}	
\nc{\Eca}{		\endmatrix	\right.		}	
\nc{\eCa}{		\endmatrix	\right\}	}	

\nc{\com}{	\begin{diagram}	}
\nc{\ecom}{	  \end{diagram}	}

\nc{\tab}{	\begin{tabular}		}
\nc{\etab}{	\end{tabular}		}	
\nc{\hl}{{	\hline			}}

\nc{\Eq}{	\begin{equation}	}
\nc{\Eeq}{	\end{equation}	}
\nc{\aln}{	\begin{align}	}
\nc{\ealn}{	\end{align}	}

\nc{\Rpart}{	\rc{\thepart}{\Roman{part}}	}
\nc{\Apart}{	\rc{\thepart}{\arabic{part}}	}
		
\nc{\rref}[2]{\ref{#1}.\ref{#2}}

\nc{\pa}[1]{ 	\part{#1}		}

\nc{\se}[1]{ 	\section{\bf#1}		}
\nc{\ses}[1]{ 	\section*{\bf#1}		}
\nc{\sus}{ 	\subsection		}
\nc{\sss}{ 	\subsubsection		}

\nc{\Lem}{ 	\subsection{Lemma}		}
\nc{\slem}{ 	\subsubsection*{Lemma}		}
\nc{\sublem}{ 	\subsubsection{ Sublemma}	}
\nc{\ssublem}{ \subsubsection*{ Sublemma}	}

\nc{\Lemm}{ 	\subsection{Lemma}		}
\nc{\lemm}{ 	\subsubsection{Lemma}		}
\nc{\slemm}{ 	\subsubsection*{Lemma}		}
\nc{\sublemm}{ 	\subsubsection{ Sublemma}	}
\nc{\ssublemm}{ \subsubsection*{ Sublemma}	}

\nc{\Pro}{ 	\subsection{Proposition}	}
\nc{\pro}{ 	\subsubsection{Proposition}	}
\nc{\spro}{ 	\subsubsection*{Proposition}	}

\nc{\Cor}{ 	\subsection{Corollary}		}
\nc{\scor}{ 	\subsubsection*{Corollary}	}

\nc{\Corr}{ 	\subsection{Corollary}		}
\nc{\corr}{ 	\subsubsection{Corollary}	}
\nc{\scorr}{ 	\subsubsection*{Corollary}	}

\nc{\Theo}{ 	\subsection{Theorem}		}		
\nc{\theo}{ 	\subsubsection{Theorem}		}
\nc{\stheo}{ 	\subsubsection*{Theorem}	}

\nc{\pretheo}{ 	\subsubsection{Pretheorem}	}

\nc{\rem}{ 	\subsubsection{Remark}		}
\nc{\srem}{ 	\subsubsection*{Remark}	}

\nc{\rems}{ 	\subsubsection{Remarks}		}

\nc{\srems}{ 	\subsubsection*{Remarks}	}

\nc{\Def}{ 	\subsection{Definition}		}
\nc{\ddef}{ 	\subsubsection{Definition}	}

\nc{\comm}{ 	\subsubsection{Comment}		}
\nc{\scomm}{ 	\subsubsection*{Comment}	}
\nc{\comms}{ 	\subsubsection{Comments}		}
\nc{\scomms}{ 	\subsubsection*{Comments}	}

\nc{\claim}{ 	\subsubsection{Claim}		}
\nc{\sclaim}{ 	\subsubsection*{Claim}	}

\nc{\nota}{ 	\subsubsection{Notation}	}

\nc{\sconj}{ 	\subsubsection*{Conjecture}	}

\nc{\ex}{ 	\subsubsection{Example}		}
\nc{\sex}{ 	\subsubsection*{Example}	}
\nc{\exs}{ 	\subsubsection{Examples}	}
\nc{\sexs}{ 	\subsubsection*{Examples}	}
\nc{\Ex}{ 	\subsection{Example}		}
\nc{\sEx}{ 	\subsection*{Example}	}
\nc{\Exs}{ 	\subsection{Examples}	}
\nc{\sExs}{ 	\subsection*{Examples}	}

\nc{\que}{ 	\subsubsection{Question}	}
\nc{\ques}{ 	\subsubsection{Questions}	}
\nc{\sque}{ 	\subsubsection*{Question}	}
\nc{\sques}{ 	\subsubsection*{Questions}	}

\nc{\bi}{	\begin{itemize}\item		}
\rc{\i}{	\item			}
\nc{\ei}{ \end{itemize}	} 
\nc{\ben}{	\begin{enumerate}\item		}
\nc{\een}{	\end{enumerate}			}
\nc{\ftt}[1]{{\footnote{#1}}}
\nc{\fttt}[1]{{$^($\footnote{#1}$^)$}}
\nc{\bftt}[1]{\footnote{#1}}
\nc{\f}[1]{ \fbox{$ $}\footnote{ \fbox{!}#1 }\fbox{$ $}		}

\nc{\Ao}{{	\A^1	}}
\nc{\Po}{{	\P^1	}}

\nc{\h}{{	\hslash	}}	

\nc{\All}{{	\forall		}}
\nc{\Exx}{{	\exists 	}}
\nc{\yy}{\infty}                       
\nc{\ys}{{  \frac{\infty}{2}  }}

\nc{\ii}{{i\in I}}
\nc{\ww}{{w\in W}}
\nc{\SES}[5]{{	0 @>>> {#1} @>{#2}>> {#3} @>{#4}>> {#5} @>>> 0	}}
\nc{\Ses}[3] {{	0 @>>> {#1} @>>>     {#2} @>>>     {#3} @>>> 0	}}
\nc{\pl}{{\oplus}}              		
\nc{\tim}{{\times}}             
\nc{\btim}{{\boxtimes}}
\nc{\ltim}{\ltimes}                  	%
\nc{\rtim}{\rtimes}			%
\nc{\ltr}{\triangleleft}        %
\nc{\rtr}{\triangleright}       %

\nc{\ten}{{	\otimes		}}            
\nc{\Lten}{{	\aa{L}\otimes	}}            
\nc{\Ltim}{{	\aa{L}\times	}}            
\nc{\Rcap}{{	\aa{R}\cap	}}            
\nc{\tenA}{	\bb{A}\ten	}
\nc{\tenB}{	\bb{B}\ten	}
\nc{\tenZ}{	\bb{\Z}\ten	}
\nc{\tenR}{	\bb{\R}\ten	}
\nc{\tenC}{	\bb{\C}\ten	}
\nc{\tenk}{	\bb{\k}\ten	}
\nc{\bten}{{\boxtimes}}         		

\nc{\con}{{ @>{\protect\cong}>> }}  	
\nc{\conl}{{ 	@>{\cong}>>	}}  	
\nc{\conn}{{    @<{\cong}<<  	}}  	
\nc{\Con}{{	\equiv		}}	
\nc{\appr}{{	\sim		}}	
\nc{\eqr}{{	\sim		}}	
\nc{\equi}{{	\sim		}}	

\nc{\fra}{ 	\frac	}     	
\nc{\ffr}[2]{{ 	\text{\footnotesize $\frac{#1}{#2}$	}	}}  

\nc{\ha}{{ \frac{1}{2} }}     		
	\nc{\half}{{ \frac{1}{2} }}    	

\nc{\ci}{{\circ}}               
\nc{\cd }{{\cdot}}            	
\nc{\cddd}{{\cdot\cdot\cdot}}	

\nc{\ox}{{	\OO_X		}}               
\nc{\omx}{{	\om_X		}}               
\nc{\Omx}{{	\Om_X^1		}}               
\nc{\qcoh}{{	q\CC oh		}}               %

\nc{\xt}{{	X_*(T)		}}
\nc{\Xt}{{	X^*(T)		}}

\nc{\cfm}{{	co\fm		}}	

\nc{\cupp}{\bigcup}             
\nc{\capp}{\bigcap}
\nc{\pll}{\bigoplus}

\nc{\pii}{\prod}                
\nc{\ppii}{\bigprod}            
\nc{\cci}{\sqcup}              
\nc{\ccii}{\bigsqcup}

\nc{\wwe}{\bigwedge}            
\nc{\cce}{\bigcoprod}           

\nc{\aaa}{	\stackerel	}	

\nc{\edd}{{ \end{document}	}} 

\nc{\tx}{	\text		}		
\nc{\df}{{ \protect\overset{ \text{def}}= 	}}		
\nc{\dff}{{ \ \df\				}}		
\nc{\inv}{{ {}^{-1}      }}			

\nc{\thh}{	^{\text{th}}	}                     	
\nc{\st}{	^{\text{st}}	}                     	
\nc{\nd}{	^{\text{nd}}	}                     	
\nc{\rd}{	^{\text{rd}}	}                     	

\nc{\pmo}{{ 	\pm 1		}}
\nc{\mpo}{{ 	\mp 1		}}

\nc{\htt}{  \text{ht}}				

\nc{\emp}{{   \emptyset}}      			

\nc{\cowe}{{	\vee	}}			
\nc{\we}{{\wedge}}				
\nc{\wee}{{	\aa{\bullet}\wedge	}}		
\nc{\wetwo}{{     \pr\overset{2}\wedge       }}	

\nc{\limp}{{	\pr\underset {\leftarrow} \lim		}}	
\nc{\Limp}{{	\pr\underset {\leftarrow} {\bbb\lim}	}}	
\nc{\limi}{{	\pr\underset {\rightarrow}\lim		}}      
\nc{\Limi}{{	\pr\underset {\rightarrow}{\bbb\lim}	}}	
\nc{\llim}[1]{	 \bb{#1}\lim        	}   
\nc{\llimp}[1]{ \bb{#1}{ \pr\underset {\leftarrow} \lim       } }
\nc{\LLimp}[1]{ \bb{#1}{ \pr\underset {\leftarrow} {\bbb\lim} } } 	
\nc{\llimi}[1]{ \bb{#1}{ \pr\underset {\rightarrow}\lim       } }
\nc{\LLimi}[1]{ \bb{#1}{ \pr\underset {\rightarrow}{\bbb\lim} } }	
						
\nc{\ppp}{{ {\Bbb P}^1 }}            		
\nc{\ppn}{{ {\Bbb P}^n }}            		
\nc{\pt}{	{ \text{pt} }	}		
				
\nc{\qlb}{{ \barr{{\Bbb Q}_l} }}      		
\nc{\ffq}{{  {\Bbb F}_q  }}           		
\nc{\ffp}{{  {\Bbb F}_p  }}           		
\nc{\tw}{   {}^{(1)}	}		

\nc{\Ab}{{ 	\AA b 		}}      		%
\nc{\Set}{{ 	\SS et 		}}      		%
\nc{\Top}{{ 	\TT op 		}}      		%
\nc{\Pic}{{ 	\tx{Pic}	}}      		%

\nc{\del}{{\partial }}
\nc{\delb}{{\partial }}
\nc{\dd}[2]{	\fra{d{#1}}{d{#2}}		}
\nc{\ddel}[2]{	\fra{\del{#1}}{\del {#2}}	}
                                        
\nc{\Spec}{{ 	\text{Spec}      		}} 
\nc{\Specf}{{ 	\text{Specf}      		}} 
\nc{\Spf}{{ 	\text{Spf}      		}} 
\nc{\hk}{{     \text{hyperk{\"a}hler} 	}}
\nc{\susy}{{\text{supersymmetry}}}

\nc{\ie}{{,\ \     \text{i.e.,}\ \ 	}}
\nc{\iif}{{\ \     \text{if}\ \ 	}}
\nc{\aand}{{\ \ \  \text{and}\ \ \ 	}}
\nc{\hence}{{\ \ \ \text{hence}\ \ \ 	}}
\nc{\while}{{\ \ \ \text{while}\ \ \ 	}}
\nc{\with}{{\ \ \  \text{with}\ \ \ 	}}
\nc{\oor}{{\ \     \text{or}\ \ 	}}
\nc{\foor}{{\ \     \text{for}\ \ 	}}
\nc{\suchthat}{{\ \     \text{such that}\ \ 	}}

\nc{\Coker}{{\operatorname{Coker}}}
\rc{\Im}{{ 	\text{Im} 	}}
\nc{\rank}{{	\ \text{rank} 	}}
\nc{\Res}{{	\  \text{Res}   }}
\nc{\End}{{	\text{End}	}}
\nc{\HHom}{{	\text{$\HH$om}	}}
\nc{\RHHom}{{	\text{R$\HH$om} }}
\nc{\RGa}{{	\text{R$\Ga$}	}}
\nc{\EEnd}{{	\text{$\EE nd$}	}}
\nc{\AAut}{{	\text{$\AA ut$}	}}
\nc{\Der}{{	\text{Der}	}}

\nc{\ord	}{{ \text{ord} }}			
\nc{\divv	}{{ \text{div} }}			
\nc{\Lie	}{{ \text{Lie} }}

\nc{\timA} {{   \pr\underset{A}\tim             }}
\nc{\timB} {{   \pr\underset{B}\tim             }}
\nc{\timC} {{   \pr\underset{C}\tim             }}
\nc{\timG} {{   \pr\underset{G}\tim             }}
\nc{\timH} {{   \pr\underset{H}\tim             }}
\nc{\timN} {{   \pr\underset{N}\tim             }}
\nc{\timP}{{    \pr\underset{P}\tim             }}
\nc{\timQ}{{    \pr\underset{Q}\tim             }}
\nc{\timS} {{   \pr\underset{S}\tim             }}
\nc{\timT} {{   \pr\underset{T}\tim             }}
\nc{\timU} {{   \pr\underset{U}\tim             }}
\nc{\timV} {{   \pr\underset{V}\tim             }}
\nc{\timX} {{   \pr\underset{X}\tim             }}
\nc{\timY} {{   \pr\underset{Y}\tim             }}
\nc{\timZ} {{   \pr\underset{Z}\tim             }}

\nc{\ab}{{       ^{\text{ab}}   		}}
\nc{\af}{{       ^{\text{aff}}  		}}

\nc{\cod}{\text{codim}}	

\rc{\AA}{{\mathcal A}}
\nc{\BB}{{\mathcal B}} 
\nc{\CC}{{\mathcal C}}
\nc{\DD}{{\mathcal D}}
\nc{\EE}{{\mathcal E}}
\nc{\FF}{{\mathcal F}}
\nc{\GG}{{\mathcal G}}
\nc{\HH}{{\mathcal H}}
\nc{\II}{{\mathcal I}}
\nc{\JJ}{{\mathcal J}}
\nc{\KK}{{\mathcal K}}
\nc{\LL}{{\mathcal L}}
\nc{\MM}{{\mathcal M}}
\nc{\NN}{{\mathcal N}}
\nc{\OO}{{\mathcal O}}
\nc{\PP}{{\mathcal P}}
\nc{\QQ}{{\mathcal Q}}
\nc{\RR}{{\mathcal R}}
\rc{\SS}{{\mathcal S}}
\nc{\TT}{{\mathcal T}}
\nc{\UU}{{\mathcal U}}
\nc{\VV}{{\mathcal V}}
\nc{\WW}{{\mathcal W}}
\nc{\ZZ}{{\mathcal Z}}
\nc{\XX}{{\mathcal X}}
\nc{\YY}{{\mathcal Y}}

\nc{\A}{{\Bbb A }}
\nc{\B}{{\Bbb B}}
\nc{\C}{{\mathbb C}}
		\nc{\cc}{{\Bbb C}}
\nc{\Cs}{{\Bbb C^*}}
		\nc{\cs}{{\Bbb C^*}}
		\nc{\ccs}{{\Bbb C^*}}
\nc{\D}{{\Bbb D}}
\nc{\E}{{\Bbb E}}
\nc{\F}{{\Bbb F}}
\nc{\G}{{\Bbb G}}
	\nc{\hH}{{\Bbb H}}
\nc{\I}{{\Bbb I}}
\nc{\J}{{\Bbb J}}
\nc{\K}{{\Bbb K}}
	\nc{\lL}{{\Bbb L}}
\nc{\M}{{\Bbb M}}
\nc{\N}{{\Bbb N}}
	\nc{\oO}{{\Bbb O}}
	\nc{\pP}{{\Bbb P}}      
\nc{\Q}{{\Bbb Q}}
\nc{\R}{{\Bbb R}}
	\nc{\sS}{{\Bbb S}}
\nc{\T}{{\Bbb T}}
\nc{\U}{{\Bbb U}}
\nc{\V}{{\Bbb V}}
\nc{\W}{{\Bbb W}}
\nc{\Z}{{\mathbb Z}}
\nc{\X}{{\Bbb X}}
\nc{\Y}{{\Bbb Y}}

\let\P\pP

\nc{\fA}{{\frak A}}
\nc{\fB}{{\frak B}}
\nc{\fC}{{\frak C}}
\nc{\fD}{{\frak D}}
\nc{\fE}{{\frak E}}
\nc{\fF}{{\frak F}}
\nc{\fG}{{\frak G}}
\nc{\fH}{{\frak H}}
\nc{\fI}{{\frak I}}
\nc{\fJ}{{\frak J}}
\nc{\fK}{{\frak K}}
\nc{\fL}{{\frak L}}
\nc{\fM}{{\frak M}}
\nc{\fN}{{\frak N}}
\nc{\fO}{{\frak O}}
\nc{\fP}{{\frak P}}
\nc{\fQ}{{\frak Q}}
\nc{\fR}{{\frak R}}
\nc{\fS}{{\frak S}}
\nc{\fT}{{\frak T}}
\nc{\fU}{{\frak U}}
\nc{\fV}{{\frak V}}
\nc{\fW}{{\frak W}}
\nc{\fZ}{{\frak Z}}
\nc{\fX}{{\frak X}}
\nc{\fY}{{\frak Y}}
\nc{\fa}{{\frak a}}
\nc{\fb}{{\frak b}}
\nc{\fc}{{\frak c}}
\nc{\fe}{{\frak e}}
\nc{\ff}{{\frak f}}
\nc{\fgg}{{\frak g}}
\nc{\fh}{{\frak h}}
\nc{\fiI}{{\frak i}}  
	\nc{\ffi}{{\frak i}}  
\nc{\fj}{{\frak j}}
\nc{\fk}{{\frak k}}
\nc{\fm}{{\frak m}}
\nc{\fn}{{\frak n}}
\nc{\fo}{{\frak o}}
\nc{\fp}{{\frak p}}
\nc{\fq}{{\frak q}}
\nc{\fr}{{\frak r}}
\nc{\fs}{{\frak s}}
\nc{\ft}{{\frak t}}
\nc{\fu}{{\frak u}}
\nc{\fv}{{\frak v}}
\nc{\fw}{{\frak w}}
\nc{\fz}{{\frak z}}
\nc{\fx}{{\frak x}}
\nc{\fy}{{\frak y}}

\nc{\al}{{\alpha }}
\nc{\be}{{\beta }}
\nc{\ga}{{\gamma }}
\nc{\de}{{\delta }}
\nc{\ep}{{\varepsilon }}
\nc{\vap}{{\epsilon }}

\nc{\ze}{{\zeta }}
\nc{\et}{{\eta }}
\rc{\th}{{\theta }}
\nc{\vth}{{\vartheta }}

\nc{\io}{{\iota }}
\nc{\ka}{{\kappa }}
\nc{\la}{{\lambda }}
\nc{\vpi}{{	\varpi		}}
\nc{\vrho}{{	\varrho		}}
\nc{\si}{{	\sigma 		}}
\nc{\ups}{{	\upsilon 	}}
\nc{\vphi}{{	\varphi 	}}
\nc{\om}{{	\omega 		}}

\nc{\Ga}{{\Gamma }}
\nc{\De}{{\Delta }}
\nc{\nab}{{\nabla}}
\nc{\Th}{{\Theta }}
\nc{\La}{{\Lambda }}
\nc{\Si}{{\Sigma }}
\nc{\Ups}{{\Upsilon }}
\nc{\Om}{{\Omega }}

\nc{\Aa}{{	\text{A}	}}
\nc{\Bb}{{	\text{B}	}}
\nc{\Cc}{{	\text{C}	}}
\nc{\Dd}{{	\text{D}	}}
\nc{\Ee}{{	\text{E}	}}
\nc{\Ff}{{	\text{F}	}}
\nc{\Gg}{{	\text{G}	}}
\nc{\Hh}{{	\text{H}	}}
\nc{\Ii}{{	\text{I}	}}
\nc{\Jj}{{	\text{J}	}}
\nc{\Kk}{{	\text{K}	}}
\nc{\Ll}{{	\text{L}	}}
\nc{\Mm}{{	\text{M}	}}
\nc{\Nn}{{	\text{N}	}}
\nc{\Oo}{{	\text{O}	}}
\nc{\Pp}{{	\text{P}	}}
\nc{\Qq}{{	\text{Q}	}}
\nc{\Rr}{{	\text{R}	}}
\nc{\Ss}{{	\text{S}	}}
\nc{\Tt}{{	\text{T}	}}
\nc{\Uu}{{	\text{U}	}}
\nc{\Vv}{{	\text{V}	}}
\nc{\Ww}{{	\text{W}	}}
\nc{\Zz}{{	\text{Z}	}}
\nc{\Xx}{{	\text{X}	}}
\nc{\Yy}{{	\text{Y}	}}

\nc{\bGa}{{	\bbb{\Ga}	}}
\nc{\bA}{{	\bbb{A}		}}
\nc{\bB}{{	\bbb{B}		}}
\nc{\bE}{{	\bbb{E}	}}
\nc{\bF}{{	\bbb{F}	}}
\nc{\bG}{{	\bbb{G}	}}
\nc{\bH}{{	\bbb{H}	}}
\nc{\bI}{{	\bbb{I}	}}
\nc{\bJ}{{	\bbb{J}	}}
\nc{\bK}{{	\bbb{K}	}}
\nc{\bL}{{	\bbb{L}	}}
\nc{\bM}{{	\bbb{M}	}}
\nc{\bN}{{	\bbb{N}	}}
\nc{\bO}{{	\bbb{O}	}}
\nc{\bP}{{	\bbb{P}	}}
\nc{\bQ}{{	\bbb{Q}	}}
\nc{\bR}{{	\bbb{R}	}}
\nc{\bS}{{	\bbb{S}	}}
\nc{\bT}{{	\bbb{T}	}}
\nc{\bU}{{	\bbb{U}	}}
\nc{\bV}{{	\bbb{V}	}}
\nc{\bW}{{	\bbb{W}	}}
\nc{\bX}{{	\bbb{X}	}}
\nc{\bY}{{	\bbb{Y}	}}
\nc{\bZ}{{	\bbb{Z}	}}
\nc{\ba}{{	\bbb{a}	}}
			\nc{\bbbb}{{	\bbb{b}	}}
\nc{\bc}{{	\bbb{c}	}}
\nc{\bd}{{	\bbb{d}	}}
			\nc{\bbe}{{	\bbb{e}	}}
			\nc{\bbf}{{	\bbb{f}	}}
\nc{\bg}{{	\bbb{g}	}}
\nc{\bh}{{	\bbb{h}	}}
			\nc{\bbi}{{	\bbb{i}	}}
\nc{\bj}{{	\bbb{j}	}}
			\nc{\bbk}{{	\bbb{k}	}}
\nc{\bl}{{	\bbb{l}	}}
\nc{\bm}{{	\bbb{m}	}}
\nc{\bn}{{	\bbb{n}	}}
\nc{\bo}{{	\bbb{o}	}}
\nc{\bp}{{	\bbb{p}	}}
\nc{\bq}{{	\bbb{q}	}}
\nc{\br}{{	\bbb{r}	}}
\nc{\bs}{{	\bbb{s}	}}
\nc{\bt}{{	\bbb{t}	}}
			\nc{\bbbu}{{	\bbb{u}	}}
\nc{\bv}{{	\bbb{v}	}}
\nc{\bw}{{	\bbb{w}	}}
\nc{\bxx}{{	\bbb{x}	}}
\nc{\by}{{	\bbb{y}	}}
\nc{\bz}{{	\bbb{z}	}}

\nc{\sA}{{\mathsf A}}
\nc{\sB}{{\mathsf B}}
\nc{\sC}{{\mathsf C}}
\nc{\sD}{{\mathsf D}}
\nc{\sE}{{\mathsf E}}
\nc{\sF}{{\mathsf F}}
\nc{\sG}{{\mathsf G}}
\nc{\sH}{{\mathsf H}}
\nc{\sI}{{\mathsf I}}
\nc{\sJ}{{\mathsf J}}
\nc{\sK}{{\mathsf K}}
\nc{\sL}{{\mathsf L}}
\nc{\sM}{{\mathsf M}}
\nc{\sN}{{\mathsf N}}
\nc{\sO}{{\mathsf O}}
\nc{\sP}{{\mathsf P}}
\nc{\sQ}{{\mathsf Q}}
\nc{\sR}{{\mathsf R}}
\rc{\sS}{{\mathsf S}}
\nc{\sT}{{\mathsf T}}
\nc{\sU}{{\mathsf U}}
\nc{\sV}{{\mathsf V}}
\nc{\sW}{{\mathsf W}}
\nc{\sX}{{\mathsf X}}
\nc{\sY}{{\mathsf Y}}
\nc{\sZ}{{\mathsf R}}
\nc{\sa}{{\mathsf a}}
\rc{\sb}{{\mathsf b}}
\rc{\sc}{{\mathsf c}}
\nc{\sd}{{\mathsf d}}
\nc{\sg}{{\mathsf g}}
\nc{\sh}{{\mathsf h}}
\nc{\sj}{{\mathsf j}}
\nc{\sk}{{\mathsf k}}
\nc{\sn}{{\mathsf n}}
\nc{\so}{{\mathsf o}}
\nc{\sq}{{\mathsf q}}
\nc{\sr}{{\mathsf r}}
\nc{\su}{{\mathsf u}}
\nc{\sv}{{\mathsf v}}
\nc{\sw}{{\mathsf w}}
\nc{\sx}{{\mathsf x}}
\nc{\sy}{{\mathsf y}}
\nc{\sz}{{\mathsf z}}

\nc{\toc}{{ 	\small{\tableofcontents} }}
\nc{\addl}{	\addcontentsline{toc}{subsection}	}

\def\sheafHom{\mathcal{H} \hspace{-1pt} \mathit{om}}

\newcommand{\bD}{\mathbb{D}}
\newcommand{\uC}{\underline{\mathbb{C}}}
\newcommand{\uD}{\underline{\mathbb{D}}}

\newcommand{\bbY}{\mathbb Y}
\newcommand{\bbX}{\mathbb X}
\newcommand{\calA}{\mathcal{A}}
\newcommand{\calB}{\mathcal{B}}

\newcommand{\calD}{\mathcal{D}}
\newcommand{\calE}{\mathcal{E}}
\newcommand{\calF}{\mathcal{F}}
\newcommand{\calG}{\mathcal{G}}
\newcommand{\calH}{\mathcal{H}}

\newcommand{\calM}{\mathcal{M}}

\newcommand{\calO}{\mathcal{O}}

\newcommand{\calZ}{\mathcal{Z}}
\newcommand{\wcalN}{\widetilde{\mathcal{N}}}
\newcommand{\wfrakg}{\widetilde{\mathfrak{g}}}

\newcommand{\frakg}{\mathfrak{g}}
\newcommand{\frakt}{\mathfrak{t}}
\newcommand{\frakb}{\mathfrak{b}}

\newcommand{\frakK}{\mathfrak{K}}

\newcommand{\frakp}{\mathfrak{p}}

\newcommand{\lotimes}{{\stackrel{_L}{\otimes}}}
\newcommand{\rcap}{{\stackrel{_R}{\cap}}}
\newcommand{\Gm}{{\mathbb{G}}_{\mathbf{m}}}
\newcommand{\Hom}{{\rm Hom}}
\newcommand{\Ext}{{\rm Ext}}

\newcommand{\Coh}{\mathsf{Coh}}

\newcommand{\rk}{{\rm rk}}

\newcommand{\op}{{\rm op}}

\newcommand{\aff}{{\rm aff}}

\newcommand{\coH}{\mathsf{H}}
\newcommand{\coHc}{\widehat{\mathsf{H}}}
\newcommand{\Kth}{\mathsf{K}}

\newcommand{\Co}{\mathrm{(Comp)}}
\newcommand{\BC}{\mathrm{(BC)}}

\newcommand{\Fourier}{\mathfrak{Fourier}}
\newcommand{\Koszul}{\mathbf{Koszul}}
\newcommand{\uRR}{\underline{\mathrm{RR}}}
\newcommand{\oRR}{\overline{\mathrm{RR}}}
\newcommand{\can}{\mathsf{can}}
\newcommand{\resK}{\mathbf{res}}
\newcommand{\resH}{\mathfrak{res}}
\newcommand{\pdiK}{\mathbf{pdi}}
\newcommand{\pdiH}{\mathfrak{pdi}}
\newcommand{\dual}{\mathbf{D}}
\newcommand{\bfi}{\mathbf{i}}

\newcommand{\simto}{\xrightarrow{\sim}}

\newcommand{\Haff}{\mathcal{H}_{\aff}}
\newcommand{\gHaff}{\overline{\mathcal{H}}_{\aff}}
\newcommand{\hgHaff}{\widehat{\overline{\mathcal{H}}}_{\aff}}

\begin{document}

\begin{abstract}
In this paper we prove that the \emph{linear Koszul duality} isomorphism for convolution algebras in $\Kth$-homology of \cite{MR3} and the \emph{Fourier transform} isomorphism for convolution algebras in Borel--Moore homology of \cite{EM} are related by the Chern character. So, Koszul duality appears as a categorical upgrade of Fourier transform of constructible sheaves. This result explains the connection between the categorification of the Iwahori--Matsumoto involution for graded affine Hecke algebras in \cite{EM} and for ordinary affine Hecke algebras in \cite{MR3}.
\end{abstract}

\title[Linear Koszul duality and Fourier transform]{Linear Koszul duality and Fourier transform \\ for convolution algebras}

\author{Ivan Mirkovi\'c}
\address{University of Massachusetts, Amherst, MA.}
\email{mirkovic@math.umass.edu}

\author{Simon Riche}
\address{Universit{\'e} Clermont Auvergne, Universit{\'e} Blaise Pascal, Laboratoire de Math{\'e}matiques, BP 10448, F-63000 Clermont-Ferrand, France -- CNRS, UMR 6620, LM, F-63178 Aubi{\`e}re, France.}
\email{simon.riche@math.univ-bpclermont.fr}

\maketitle

\section*{Introduction}

\subsection{}

This article is a sequel to \cite{MR, MR2, MR3}. It links 
two kinds of ``Fourier'' transforms
prominent in mathematics, the Fourier transform for constructible sheaves
and the Koszul duality.
This is  done in a particular situation which is of interest 
in representation theory, namely the context of convolution algebras.

\subsection{Chern character map}

Our geometric setting consists of two vector subbundles $F_1,F_2$ of a trivial vector bundle $X\times V$ over a (smooth and proper) complex algebraic variety $X$. We consider the fiber product $F_1\times_V F_2$ as well as the dual object -- the 
fiber product
$F_1^\perp \times_{V^*} F_2^\perp$ of 
orthogonal complements of $F_1$ and $F_2$
inside the dual vector bundle $X\times V^*$.
The \emph{linear Koszul duality} mechanism from \cite{MR, MR2, MR3}
is a geometric version of the standard Koszul duality between  
graded modules over the symmetric algebra of a vector space and 
graded modules over the exterior algebra of the dual vector space. 
Here, this  formalism provides an equivalence of categories of
equivariant coherent sheaves
on the derived fiber products $F_1\, \aa{R}\tim_V\, F_2$ 
and  $F_1^\perp \, \aa{R}\tim_{V^*}\, F_2^\perp$ (in the sense of dg-schemes).
In particular we get an isomorphism of 
equivariant $\Kth$-homology groups
of algebraic varieties $F_1\times_V F_2$ and $F_1^\perp \times_{V^*} F_2^\perp$.\footnote{
Note that $\Kth$-homology does not distinguish the derived fiber product
from the usual fiber product of varieties, see \cite{MR3}.
}
On the other hand, the Fourier transform for constructible sheaves provides
an isomorphism of equivariant Borel--Moore homologies of
fiber products
$F_1\times_V F_2$ and $F_1^\perp \times_{V^*} F_2^\perp$, see \cite{EM}.

Our main result shows that the maps in $\Kth$-homology and in
Borel--Moore homology are related by the Chern character
map (the ``Riemann--Roch map'')
from equivariant $\Kth$-homology to (completed) equivariant Borel--Moore homology.\footnote{
For simplicity we work under a technical assumption
on $F_i$'s which is satisfied in all known applications.
}
In this way, linear
Koszul duality appears as a categorical upgrade
of the 
topological Fourier transform.

\subsection{Convolution algebras}

In Representation Theory the above
setting provides a geometric construction of
algebras. Indeed,
when 
$F_1=F_2=:F$ 
then
the equivariant $\Kth$-homology and Borel--Moore homology of $F \times_V F$ 
have structures of convolution algebras; for simplicity in this introduction we denote these 
$A_{\Kth}(F)$ and $A_{\mathsf{BM}}(F)$.
The Chern character provides
a map of algebras
$A_{\Kth}(F) \to \widehat{A}_{\mathsf{BM}}(F)$
from the $\Kth$-homology algebra to a completion of the Borel--Moore homology
algebra \cite{CG, Ka}.
This gives a strong relation between their representation theories:
one obtains results
on the representation theory
of the (more interesting) 
algebra $A_{\Kth}(F)$ through the relation to
the representation theory of the  algebra 
$A_{\mathsf{BM}}(F)$ which is
more accessible.\footnote{ 
The reason is the powerful machinery of perverse sheaves that
one can use in the topological setting, see \cite{CG}. 
}
In this setting, the maps 
\[
\imath_{\Kth} \colon A_{\Kth}(F) \simto A_{\Kth}(F^\perp), \qquad
\imath_{\mathsf{BM}} \colon A_{\mathsf{BM}}(F) \simto A_{\mathsf{BM}}(F^\perp)
\]
induced respectively by linear Koszul duality and 
by Fourier transform
are isomorphisms of algebras. 

An important example of this mechanism
appears  in the study of affine Hecke algebras, see \cite{KL, CG}.
The Steinberg variety $Z$ of a complex
connected reductive algebraic group $G$ (with simply connected derived subgroup)
is of the above form $F \times_V F$ where 
the space $X$ is 
the flag variety $\BB$ of $G$, the vector space $V$ is the dual $\frakg^*$ of the Lie algebra $\frakg$ of $G$, and $F$ is the
cotangent bundle $T^*\BB$. 
The $G \times \Gm$-equivariant $\Kth$-homology and Borel--Moore homology of the Steinberg 
variety $Z$ are then known to be realizations of the
\emph{affine Hecke algebra} $\Haff$
of the dual reductive group $\check{G}$ (with equal parameters) and of 
the corresponding \emph{graded affine Hecke algebra} $\gHaff$.
In this case the dual version 
$F^\perp \times_{V^*} F^\perp$ turns out to be
another -- homotopically equivalent -- version of the Steinberg variety 
$Z$.
Therefore, $\imath_{\Kth}$ and $\imath_{\mathsf{BM}}$ are automorphisms of 
 $\Haff$ and  $\gHaff$, respectively. In fact these are
(up to minor ``correction factors'') 
geometric realizations of 
the Iwahori--Matsumoto involution
of $\gHaff$
(see \cite{EM})
and  $\Haff$
(see \cite{MR3}).
The Chern character map can also be identified, in this case, with (a variant of) a morphism constructed (by algebraic methods) by Lusztig~\cite{LuAff}.
So, in this situation, Theorem \ref{thm:LKDFourier} 
explains the relation between
results of
\cite{MR3} 
and
\cite{EM}.

\subsection{Character cycles and characteristic cycles}

In \cite{Kas}, Kashiwara introduced for a group $G$ acting on a space $X$
an invariant   of a $G$-equivariant
constructible sheaf $\FF$ on $X$.
This is an element
$\mathrm{ch}_G(\FF)$ 
of the Borel--Moore homology of the stabilizer space
$G_X := \{(g,x)\in G\tim X \mid g \cdot x=x\}$. 
He ``linearized'' this construction to an element
$\mathrm{ch}_\frakg(\FF)$ of the Borel--Moore homology of the analogous 
stabilizer space
$\frakg_X$ for the Lie algebra $\frakg$ of the group $G$.
Under some assumptions (that put one in the above geometric setting)
he proved that the characteristic cycle of
$\FF$ is the image of $\mathrm{ch}_\frakg(\FF)$ under a Fourier transform map
in Borel--Moore homology (see \cite[\S 1.9]{Kas}). This work is 
the origin of papers on Iwahori--Matsumoto involution \cite{EM}
and
linear Koszul duality
\cite{MR}. From this point of view, the present paper is a part of the effort to 
categorify Kashiwara's character cycles.

\subsection{Conventions and notation}

In the body of the paper we will consider many morphisms involving $\Kth$-homology and Borel--Moore homology. We use the general convention that morphisms involving only $\Kth$-homology are denoted using bold letters, those involving only Borel--Moore homology are denoted using fraktur letters, and the other ones are denoted using ``sans serif'' letters.

If $X$ is a complex algebraic variety endowed with an action of a reductive algebraic group $A$, we denote by $\Coh^A(X)$ the category of $A$-equivariant coherent sheaves on $X$. If $Y \subset X$ is an $A$-stable closed subvariety we denote by $\Coh_Y^A(X)$ the subcategory consisting of sheaves supported set-theoretically on $Y$; recall that $\calD^b \Coh_Y^A(X)$ identifies with a full subcategory in $\calD^b \Coh^A(X)$. When considering $\Gm$-equivariant coherent sheaves, we denote by $\langle 1 \rangle$ the functor of tensoring with the tautological $1$-dimensional $\Gm$-module.

\subsection{Organization of the paper}

In Section \ref{sec:definitions} we define all our morphisms, and state our main result (Theorem \ref{thm:LKDFourier}). In Section~\ref{sec:convolution-algebras} we study more closely the case of convolution algebras, and even more closely the geometric setting for affine Hecke algebras; in this case we make all the maps appearing in Theorem \ref{thm:LKDFourier} explicit. In Sections \ref{sec:compatibility-Fourier} and \ref{sec:compatibility-others} we prove some compatibility statements for our constructions, and we apply these results in Section \ref{sec:proof} to the proof of Theorem \ref{thm:LKDFourier}. Finally, Appendix \ref{sec:appendix} contains the proofs of some technical lemmas needed in other sections.

\subsection{Acknowledgements} 

We thank Roman Bezrukavnikov for useful conversations, and the referee for helpful suggestions and for insisting on making the application to Hecke algebras more explicit.

I.M.~was supported by NSF Grants.
S.R.~was supported by ANR Grants No.~ANR-09-JCJC-0102-01, ANR-2010-BLAN-110-02 and ANR-13-BS01-0001-01.

\section{Definitions and statement}
\label{sec:definitions}

\subsection{Equivariant homology and cohomology}
\label{ss:homology-cohomology}

If $A$ is a complex linear algebraic group acting on a complex algebraic variety $Y$, we denote by $\calD^A_{\mathrm{const}}(Y)$ the $A$-equivariant derived category of constructible complexes on $Y$ with complex coefficients, see \cite{BL}. Let $\uC_Y$, respectively $\uD_Y$, be the constant, respectively dualizing, sheaf on $Y$. These are objects of $\calD^A_{\mathrm{const}}(Y)$. We also denote by $\bD_Y \colon \calD^A_{\mathrm{const}}(Y) \simto \calD^A_{\mathrm{const}}(Y)^{\mathrm{op}}$ the Grothendieck--Verdier duality functor.

If $M$ is in $\calD^A_{\mathrm{const}}(Y)$, the $i$-th equivariant cohomology of $Y$ with coefficients in $M$ is by definition
\[
\coH^i_A(Y, M) := \Ext^i_{\calD^A_{\mathrm{const}}(Y)}(\uC_Y,M).
\]
In particular, the equivariant cohomology and Borel--Moore homology of $Y$ are defined by
\[
\coH_A^i(Y) := \coH_A^i(Y,\uC_Y),
\quad \coH_i^A(Y) := \coH_A^{-i}(Y,\uD_Y).
\]
We will also use the notation
\begin{align*}
\coH_A^\bullet(Y) \ := \ \bigoplus_{i \in \mathbb{Z}} \, \coH^i_A(Y), \quad & \coHc_{A}^\bullet(Y) \ := \ \prod_{i \in \mathbb{Z}} \, \coH^i_A(Y), \\
\coH^A_\bullet(Y) \ := \ \bigoplus_{i \in \mathbb{Z}} \, \coH_i^A(Y), \quad & \coHc^{A}_\bullet(Y) \ := \ \prod_{i \in \mathbb{Z}} \, \coH_i^A(Y).
\end{align*}
(By construction of the equivariant derived category, see~\cite[\S 2.2]{BL}, these definitions coincide -- up to grading shift -- with the definitions used e.g.~in~\cite{LuCus1, EG, BZ} using some ``approximations'' of $EA$.)
Note that with our conventions, one can have $\coH_i^A(Y) \neq 0$ for $i<0$. We will use the general convention that we denote by the same symbol an homogeneous morphism between vector spaces of the form $\coH^A_\bullet(\cdot)$ or $\coH_A^\bullet(\cdot)$ and the induced morphism between the associated vector spaces $\coHc^A_\bullet(\cdot)$ or $\coHc_A^\bullet(\cdot)$.

The vector spaces $\coH_A^\bullet(Y)$ and $\coH^A_\bullet(Y)$ have natural gradings, and most morphisms between such spaces that will occur in this paper will be homogeneous. We will sometimes write a morphism e.g.~as $\coH_A^\bullet(Y) \to \coH_A^{\bullet+d}(Y')$ to indicate that it shifts degrees by $d$.

There exists a natural (right) action of the algebra $\coH_A^\bullet(Y)$ on $\coH^A_\bullet(Y)$ induced by composition of morphisms in $\calD^A_{\mathrm{const}}(Y)$; it extends to an action of the algebra $\coHc_A^\bullet(Y)$ on $\coHc^A_\bullet(Y)$.

We will also denote by $\Kth^A(Y)$ the $A$-equivariant $\Kth$-homology of $Y$, i.e.~the Grothendieck group of the category of $A$-equivariant coherent sheaves on $Y$.

We will frequently use the following classical constructions. If $Z$ is another algebraic variety endowed with an action of $A$, and if $f \colon Z \to Y$ is a proper $A$-equivariant morphism, then there exist natural ``proper direct image'' morphisms
\[
\pdiK_f \colon \Kth^A(Z) \to \Kth^A(Y), \qquad \text{resp.} \qquad \pdiH_f \colon \coH^A_\bullet(Z) \to \coH^A_\bullet(Y),
\]
see \cite[\S 5.2.13]{CG}, resp.~ \cite[\S 2.6.8]{CG}.\footnote{Only \emph{non-equivariant} Borel--Moore homology is considered in~\cite{CG}. However, the constructions for equivariant homology are deduced from these, since the equivariant homology of $Y$ can be described in terms of ordinary homology of various spaces of the form $U \times^A Y$ where $U$ is an ``approximation'' of $EA$, see e.g.~\cite[\S 2.8]{EG1}. \label{fn:equiv-homology}} Each of these maps satisfies a projection formula; in particular for $c \in \coH_A^\bullet(Y)$ and $d \in \coH^A_\bullet(Z)$ we have
\begin{equation}
\label{eqn:proj-formula-H}
\pdiH_f(d \cdot f^*(c)) = \pdiH_f(d) \cdot c,
\end{equation}
where $f^* \colon \coH_A^\bullet(Y) \to \coH_A^\bullet(Z)$ is the natural pullback morphism.

On the other hand, if $Y$ is smooth, $Y' \subset Y$ is an $A$-stable smooth closed subvariety, and $Z \subset Y$ is a not necessarily smooth $A$-stable closed subvariety, then we have ``restriction with supports'' morphisms
\[
\resK \colon \Kth^A(Z) \to \Kth^A(Z \cap Y'), \qquad \text{resp.} \qquad \resH \colon \coH^A_\bullet(Z) \to \coH^A_{\bullet-2\dim(Y)+2\dim(Y')}(Z \cap Y')
\]
associated with the inclusion $Y' \hookrightarrow Y$,
see \cite[p.~246]{CG}, resp.~\cite[\S 2.6.21]{CG}. (The definition of the second morphism is recalled in \S\ref{ss:restriction-with-supports}.) Note that the morphism $\resH$ satisfies the formula
\begin{equation}
\label{eqn:res-cohomology}
\resH(c \cdot d) = \resH(c) \cdot i^*(d)
\end{equation}
for $c \in \coH^A_\bullet(Z)$ and $d \in \coH^\bullet_A(Z)$, where $i \colon Z \cap Y' \hookrightarrow Z$ is the embedding and $i^*$ is the pullback in cohomology as in~\eqref{eqn:proj-formula-H}. (In the non-equivariant setting, this follows from~\cite[Equation~(2.6.41)]{CG} and the definition of $\resH$ in~\cite[\S 2.6.21]{CG}; the equivariant case follows using the remark in Footnote~\ref{fn:equiv-homology}.)

Finally, if $E \to Y$ is an $A$-equivariant vector bundle, then we have the Thom isomorphism
\[
\coH^A_\bullet(E) \cong \coH^A_{\bullet-2\rk(E)}(Y).
\]

\subsection{Fourier--Sato transform}
\label{ss:Fourier-transform}

Let again $A$ be  a complex linear algebraic group, and let $Y$ be an $A$-variety. If $r \colon E \to Y$ is an $A$-equivariant (complex) vector bundle, we equip it with an $A \times \Gm$-action where  $t \in \Gm$ acts by multiplication by $t^{-2}$ along the fibers of $r$. We denote by $E^\diamond$ the $A \times \Gm$-equivariant dual vector bundle (so that $t \in \Gm$ acts by multiplication by $t^2$ along the fibers of the projection to $Y$), and by $E^*$ the dual $A$-equivariant vector bundle, which we equip with a $\Gm$-action where  $t \in \Gm$ acts by multiplication by $t^{-2}$ along the fibers. We denote by ${\check r}\colon E^* \to X$ the projection. 

The Fourier--Sato transform defines an equivalence of categories
\begin{equation}
\label{eqn:fourier1}
\mathfrak{F}_E\colon \calD^{A \times \Gm}_{\mathrm{const}}(E) \ \xrightarrow{\sim} \ \calD^{A \times \Gm}_{\mathrm{const}}(E^{\diamond}).
\end{equation}
This equivalence is constructed as follows (see \cite[\S 3.7]{KS}; see also \cite[\S 2.7]{AHJR} for a reminder of the main properties of this construction). Let $Q:=\{(x,y) \in E \times_Y E^\diamond \mid \mathrm{Re}(\langle x,y \rangle) \leq 0\}$, and let $q\colon Q \to E$, ${\check q} \colon Q \to E^\diamond$ be the projections. Then we have
\[
\mathfrak{F}_E:={\check q}_! q^*.
\]
(This equivalence is denoted $(\cdot)^\wedge$ in \cite{KS}; it differs by a cohomological shift from the equivalence $\mathbb{T}_E$ of \cite{AHJR}.)

Inverse image under the automorphism of $A \times \Gm$ which sends $(g,t)$ to $(g,t^{-1})$ establishes an equivalence of categories
\begin{equation}
\label{eqn:fourier2}
\calD^{A \times \Gm}_{\mathrm{const}}(E^\diamond) \ \xrightarrow{\sim} \ \calD^{A \times \Gm}_{\mathrm{const}}(E^*),
\end{equation}
see~\cite[Chap.~6]{BL}. We will denote by
\[
\calF_E \colon \calD^{A \times \Gm}_{\mathrm{const}}(E) \xrightarrow{\sim} \calD^{A \times \Gm}_{\mathrm{const}}(E^*)
\]
the composition of \eqref{eqn:fourier1} and \eqref{eqn:fourier2}.

Let $F \subset E$ be an $A$-stable subbundle, and denote by $F^\bot \subset E^*$ the orthogonal to $F$. Then one can consider the constant sheaf $\uC_F$ as an object of $\calD^{A \times \Gm}_{\mathrm{const}}(E)$. (Here and below, we omit direct images under closed inclusions when no confusion is likely.) Similarly, we have the object $\uC_{F^\bot}$ of $\calD^{A \times \Gm}_{\mathrm{const}}(E^*)$. The following result is well known; we reproduce the proof for future reference.

\begin{lem}
\label{lem:fourier-F}
There exists a canonical isomorphism
\[
\calF_E(\uC_F) \cong \uC_{F^\bot}[-2\mathrm{rk}(F)].
\]
\end{lem}

\begin{proof}
It is equivalent to prove a similar isomorphism for $\mathfrak{F}_E$. For simplicity we denote $F^\bot$ by the same symbol when it is considered as a subbundle of $E^\diamond$.

By definition of $\mathfrak{F}_E$ we have a canonical isomorphism
\[
\mathfrak{F}_E(\uC_F) \cong {\check q}_{F!} \uC_{Q_F},
\]
where $Q_F:=q^{-1}(F) \subset Q$ and ${\check q}_F$ is the composition of ${\check q}$ with the inclusion $Q_F \hookrightarrow Q$. There is a natural closed embedding $i_F \colon F \times_Y F^\bot \hookrightarrow Q_F$; we denote by $U_F$ the complement and by $j_F \colon U_F \hookrightarrow Q_F$ the inclusion. The natural exact triangle $j_{F!} \uC_{U_F} \to \uC_{Q_F} \to i_{F*} \uC_{F \times_X F^\bot} \xrightarrow{+1}$ provides an exact triangle
\[
q_{F!} j_{F!} \uC_{U_F} \to q_{F!} \uC_{Q_F} \to q_{F_!} i_{F!} \uC_{F \times_X F^\bot} \xrightarrow{+1}.
\]
Using the fact that $\coH_c^\bullet(\mathbb{R}_{\geq 0}; \C)=0$,
one can easily check that $q_{F!} j_{F!} \uC_{U_F}=0$, so that the second map in this triangle is an isomorphism. Finally, $q_F \circ i_F \colon F \times_Y F^\bot \to E^\diamond$ identifies with the composition of the projection $F \times_X F^\bot \to F^\bot$ with the embedding $F^\bot \hookrightarrow E^\diamond$. We deduce a canonical isomorphism
\[
q_{F!} \uC_{Q_F} \cong \uC_{F^\bot}[-2 \mathrm{rk}(F)],
\]
which finishes the proof.
\end{proof}

We will mainly use these constructions in the following situation. Let $V$ be an $A$-module (which we will consider as an $A$-equivariant vector bundle over the variety $\pt:=\mathrm{Spec}(\C)$), and let $E:=V \times Y$, an $A$-equivariant vector bundle over $Y$. We denote by $p \colon E \to V$, ${\check p} \colon E^* \to V^*$ the projections. As above, let $F \subset E$ be an $A$-stable subbundle.

\begin{cor}
\label{prop:Fourier-F}
There exists a canonical isomorphism
\begin{equation}
\label{eqn:isom-Fourier-F}
\calF_V(p_! \uC_F) \cong {\check p}_! \uC_{F^\bot}[-2\mathrm{rk}(F)].
\end{equation}
\end{cor}

\begin{proof}
By \cite[Proposition 3.7.13]{KS} (see also \cite[\S A.4]{AHJR}) we have a canonical isomorphism of functors
\[
\calF_V \circ p_! \cong {\check p}_! \circ \calF_E.
\]
In particular we deduce an isomorphism $\calF_V(p_! \uC_F) \cong {\check p}_! \calF_E(\uC_F)$. Then the result follows from Lemma \ref{lem:fourier-F}.
\end{proof}

\subsection{Equivariant homology as an $\Ext$-algebra}
\label{ss:equiv}

From now on we let $G$ be a complex connected reductive algebraic group, $X$ be a smooth and proper complex algebraic variety, and $V$ be a finite dimensional $G$-module. Let $E:=V \times X$, considered as a $G \times \Gm$-equivariant vector bundle as in \S\ref{ss:Fourier-transform}, and let $F_1,F_2$ be $G$-stable subbundles of the vector bundle $E$ over $X$. As in \S\ref{ss:Fourier-transform}, we denote by $p \colon E \to V$ the projection, and by $F_1^\bot,F_2^\bot \subset E^*$ the orthogonals to $F_1$ and $F_2$. Then there exists a canonical isomorphism
\begin{equation*}
\can_{F_1,F_2} \colon \coH^{G \times \Gm}_\bullet(F_1 \times_V F_2) \simto \Ext^{2\dim(F_2)-\bullet}_{\calD^{G \times \Gm}_{\mathrm{const}}(V)}(p_! \uC_{F_1}, p_! \uC_{F_2}).
\end{equation*}

Let us explain (for future reference) how this isomorphism can be constructed, following \cite{CG,LuCus2}. Consider the cartesian diagram
\[
\xymatrix@C=1.5cm{
E \times_V E \ar@{^{(}->}[r]^-{j} \ar[d]_-{\mu} & E \times E \ar[d]^-{p \times p} \\
V \ar@{^{(}->}[r]^-{\Delta} & V \times V
}
\]
where $\Delta$ is the diagonal embedding. Then in \cite[Equation (8.6.4)]{CG} (see also \cite[\S 1.15 and \S2.4]{LuCus2}) the authors construct a canonical and bifunctorial isomorphism
\[
\mu_* j^! (\mathbb{D}_E(A_1) \boxtimes A_2) \cong R\sheafHom_{\C}(p_! A_1, p_! A_2)
\]
for $A_1,A_2$ in $\calD^{G \times \Gm}_{\mathrm{const}}(E)$. Applying equivariant cohomology, we obtain an isomorphism
\begin{equation}
\label{eqn:morphisms-cohomology}
\Ext^{\bullet}_{\calD^{G \times \Gm}_{\mathrm{const}}(V)}(p_! A_1, p_! A_2) \cong \coH^\bullet_{G \times \Gm} \bigl(E \times_V E, j^!(\mathbb{D}_E(A_1) \boxtimes A_2 ) \bigr).
\end{equation}
Setting $A_1=\uC_{F_1}$, $A_2=\uC_{F_2}$ 
we deduce an isomorphism
\[
\Ext^{\bullet}_{\calD^{G \times \Gm}_{\mathrm{const}}(V)}(p_! \uC_{F_1}, p_! \uC_{F_2}) \cong \coH^\bullet_{G \times \Gm} \bigl(E \times_V E, j^!(\uD_{F_1} \boxtimes \uC_{F_2} ) \bigr).
\]
Let $a \colon F_1 \times F_2 \hookrightarrow E \times E$ be the inclusion, and consider the cartesian diagram
\[
\xymatrix@C=1.5cm{
F_1 \times_V F_2 \ar@{^{(}->}[r]^-{b} \ar@{^{(}->}[d]_-{k} & E \times_V E \ar@{^{(}->}[d]^-{j} \\
F_1 \times F_2 \ar@{^{(}->}[r]^-{a} & E \times E.
}
\]
Then using the base change isomorphism we obtain
\begin{multline*}
\coH^\bullet_{G \times \Gm} \bigl(E \times_V E, j^!(\uD_{F_1} \boxtimes \uC_{F_2} ) \bigr) \cong \coH^\bullet_{G \times \Gm} \bigl(E \times_V E, j^! a_*(\uD_{F_1} \boxtimes \uC_{F_2} ) \bigr) \\ 
\cong \coH^\bullet_{G \times \Gm} \bigl(E \times_V E, b_* k^!(\uD_{F_1} \boxtimes \uC_{F_2} ) \bigr) \cong \coH^\bullet_{G \times \Gm} \bigl(F_1 \times_V F_2, k^!(\uD_{F_1} \boxtimes \uC_{F_2} ) \bigr).
\end{multline*}
Now we use the canonical isomorphisms $\uC_{F_2} \cong \uD_{F_2}[-2\dim(F_2)]$ (since $F_2$ is smooth) and $k^! (\uD_{F_1} \boxtimes \uD_{F_2}) \cong k^! (\uD_{F_1 \times F_2}) \cong \uD_{F_1 \times_V F_2}$ to obtain the isomorphism $\can_{F_1,F_2}$.

\subsection{The $\Fourier$ isomorphism}
\label{ss:Fourier-isomorphism}

We continue with the setting of \S\ref{ss:equiv}, and denote by ${\check p} \colon E^* \to V^*$ the projection. Then we have canonical isomorphisms
\begin{align*}
\can_{F_1,F_2} \colon \coH^{G \times \Gm}_\bullet(F_1 \times_V F_2) \ & \simto \ \Ext^{2\dim(F_2)-\bullet}_{\calD^{G \times \Gm}_{\mathrm{const}}(V)}(p_! \uC_{F_1}, p_! \uC_{F_2}); \\
\can_{F_1^\bot,F_2^\bot} \colon \coH^{G \times \Gm}_\bullet(F_1^\bot \times_{V^*} F_2^\bot)  \ & \simto \ \Ext^{2\dim(F_2^\bot)-\bullet}_{\calD^{G \times \Gm}_{\mathrm{const}}(V^*)}({\check p}_! \uC_{F_1^\bot}, {\check p}_! \uC_{F_2^\bot}).
\end{align*}
On the other hand, through the canonical isomorphisms $\calF_V(p_* \uC_{F_i}) \cong {\check p}_* \uC_{F_i^\bot}[-2\mathrm{rk}(F_i)]$ for $i=1,2$ (see \eqref{eqn:isom-Fourier-F}), the functor $\calF_V$ induces an isomorphism
\[
\Ext^\bullet_{\calD^{G \times \Gm}_{\mathrm{const}}(V)}(p_! \uC_{F_1},p_! \uC_{F_2}) \xrightarrow{\sim} \Ext^{\bullet-2\mathrm{rk}(F_2) + 2 \mathrm{rk}(F_1)}_{\calD^{G \times \Gm}_{\mathrm{const}}(V^*)}({\check p}_! \uC_{F_1^\bot},{\check p}_! \uC_{F_2^\bot}).
\]
We denote by
\[
\Fourier_{F_1,F_2} \colon \coH^{G \times \Gm}_\bullet(F_1 \times_V F_2) \simto \coH^{G \times \Gm}_{\bullet+2\dim(F_2^\bot)-2\dim(F_1)}(F_1^\bot \times_{V^*} F_2^\bot)
\]
the resulting isomorphism. This isomorphism, considered in particular in \cite{EM}, was the starting point of our work on linear Koszul duality.

\subsection{Linear Koszul duality}
\label{ss:lkd}

Let us recall the definition and main properties of linear Koszul duality, following \cite{MR, MR2, MR3}. In this paper we will only consider the geometric situation relevant for convolution algebras, as considered in \cite[\S 4]{MR3}. However we will allow using two different vector bundles $F_1$ and $F_2$; the setting of \cite[\S 4]{MR3} corresponds to the choice $F_1=F_2$. 

We continue with the setting of \S\ref{ss:equiv}, and denote by $\Delta V \subset V \times V$ the diagonal copy of $V$. We will consider the derived category
\[
\calD^c_{G \times \Gm} \bigl( (\Delta V \times X \times X) \, \rcap_{E \times E} \, (F_1 \times F_2) \bigr)
\]
as defined in \cite[\S 3.1]{MR3}. By definition this is a subcategory of the derived category of $G\times\Gm$-equivariant quasi-coherent dg-modules over a certain sheaf of $\calO_{X \times X}$-dg-algebras on $X \times X$, which we will denote by $\calA_{F_1, F_2}$. Note that the derived intersection
\[
(\Delta V \times X \times X) \, \rcap_{E \times E} \, (F_1 \times F_2)
\]
is quasi-isomorphic to the derived fiber product $F_1 \, {\stackrel{_R}{\times}}_V \, F_2$ in the sense of \cite[\S 3.7]{BR}.

Similarly we have a derived category
\[
\calD^c_{G \times \Gm} \bigl( (\Delta V^* \times X \times X) \, \rcap_{E^* \times E^*} \, (F_1^\bot \times F_2^\bot) \bigr).
\]

We denote by $\omega_X$ the canonical line bundle on $X$. Then by \cite[Theorem 3.1]{MR3} there exists a natural equivalence of triangulated categories
\begin{multline*}
\frakK_{F_1,F_2} \colon \calD^c_{G \times \Gm} \bigl( (\Delta V \times X \times X) \, \rcap_{E \times E} \, (F_1 \times F_2) \bigr) \\
\xrightarrow{\sim} \calD^c_{G \times \Gm} \bigl( (\Delta V^* \times X \times X) \, \rcap_{E^* \times E^*} \, (F_1^\bot \times F_2^\bot) \bigr)^\op.
\end{multline*}
More precisely, \cite[Theorem 3.1]{MR3} provides an equivalence of categories
\begin{multline*}
\kappa_{F_1,F_2}\colon \calD^c_{G \times \Gm} \bigl( (\Delta V \times X \times X) \rcap_{E \times E} (F_1 \times F_2) \bigr) \\
\xrightarrow{\sim} \calD^c_{G \times \Gm} \bigl( (\overline{\Delta} V^\diamond \times X \times X) \rcap_{E^\diamond \times E^\diamond} (F_1^\bot \times F_2^\bot) \bigr)^\op
\end{multline*}
where $\overline{\Delta} V^\diamond \subset V^\diamond \times V^\diamond$ is the antidiagonal copy of $V^\diamond$. (The construction of \cite{MR3} depends on the choice of an object $\calE$ in $\calD^b \Coh^{G \times \Gm}(X \times X)$ whose image in $\calD^b \Coh(X \times X)$ is a dualizing object; here we take $\calE=\calO_X \boxtimes \omega_X[\dim(X)]$.) Then $\frakK_{F_1,F_2}$ is the composition of $\kappa_{F_1,F_2}$ with the natural equivalence
\[
\calD^c_{G \times \Gm} \bigl( (\overline{\Delta} V^\diamond \times X \times X) \, \rcap_{E^\diamond \times E^\diamond} \, (F_1^\bot \times F_2^\bot) \bigr) \xrightarrow{\sim}
\calD^c_{G \times \Gm} \bigl( (\Delta V^\diamond \times X \times X) \, \rcap_{E^\diamond \times E^\diamond} \, (F_1^\bot \times F_2^\bot) \bigr)
\]
(see \cite[\S 4.3]{MR3}) and the natural equivalence
\[
\calD^c_{G \times \Gm} \bigl( (\Delta V^\diamond \times X \times X) \, \rcap_{E^\diamond \times E^\diamond} \, (F_1^\bot \times F_2^\bot) \bigr) \xrightarrow{\sim}
\calD^c_{G \times \Gm} \bigl( (\Delta V^* \times X \times X) \, \rcap_{E^* \times E^*} \, (F_1^\bot \times F_2^\bot) \bigr)
\]
induced by the automorphism of $\Gm$ sending $t$ to $t^{-1}$.

Note that 
we have $\calH^0(\calA_{F_1, F_2})=(\pi_{F_1,F_2})_* \calO_{F_1 \times_V F_2}$, where $\pi_{F_1,F_2} \colon F_1 \times_V F_2 \to X \times X$ is the projection (which is an affine morphism). Hence, using~\cite[Lemma 5.1]{MR3} and classical facts on affine morphisms, one can canonically identify the Grothendieck group of the category $\calD^c_{G \times \Gm} \bigl( (\Delta V \times X \times X) \, \rcap_{E \times E} \, (F_1 \times F_2) \bigr)$ with $\Kth^{G \times \Gm}(F_1 \times_V F_2)$.
We have a similar isomorphism for $F_1^\bot$ and $F_2^\bot$; hence the equivalence $\frakK_{F_1,F_2}$ induces an isomorphism
\[
\Koszul_{F_1,F_2} \colon \Kth^{G \times \Gm}(F_1 \times_V F_2) \xrightarrow{\sim} \Kth^{G \times \Gm}(F_1^\bot \times_{V^*} F_2^\bot).
\] 

\subsection{Duality and parity conjugation in $\Kth$-homology}
\label{ss:duality}

To obtain a precise relation between the maps $\Fourier_{F_1,F_2}$ of~\S\ref{ss:Fourier-isomorphism} and $\Koszul_{F_1,F_2}$ of~\S\ref{ss:lkd} we will need two auxiliary maps in $\Kth$-homology.

Our first map has a geometric flavour, and is induced by Grothendieck--Serre duality. More precisely,
consider the ``duality'' equivalence
\[
\mathrm{D}_{F_1^\bot, F_2^\bot}^{G \times \Gm} \colon
\calD^b \Coh^{G \times \Gm}(F_1^{\bot} \times F_2^{\bot}) \to \calD^b \Coh^{G \times \Gm}(F_1^{\bot} \times F_2^{\bot})^\op 
\]
associated with the dualizing complex $\calO_{F_1^{\bot}} \boxtimes \omega_{F_2^\bot}[\dim(F_2^{\bot})]$,
which sends $\calG$ to
\[
R\sheafHom_{\calO_{F_1^{\bot} \times F_2^{\bot}}}(\calG, \, \calO_{F_1^{\bot}} \boxtimes \omega_{F_2^\bot})[\dim(F_2^{\bot})]
\]
(see e.g.~\cite[\S 2.1]{MR3} and references therein).
(Here, $\omega_{F_2^{\bot}}$ is the canonical line bundle on $F_2^\bot$, endowed with its natural $G \times \Gm$-equivariant structure.) This equivalence induces a (contravariant) auto-equivalence of the subcategory $\calD^b \Coh^{G \times \Gm}_{F_1^{\bot} \times_{V^*} F_2^{\bot}}(F_1^{\bot} \times F_2^{\bot})$, which we denote similarly. We denote by
\begin{equation*}
\dual_{F_1^{\bot}, F_2^{\bot}} \colon \Kth^{G \times \Gm}(F_1^{\bot} \times_{V^*} F_2^{\bot}) \xrightarrow{\sim} \Kth^{G \times \Gm}(F_1^{\bot} \times_{V^*} F_2^{\bot})
\end{equation*}
the induced automorphism at the level of Grothendieck groups.

Our second map is a ``correction factor'', with no interesting geometric interpretation. Namely, the direct image functor under the projection $\pi_{F_1^\bot, F_2^\bot} \colon F_1^\bot \times_{V^*} F_2^\bot \to X \times X$ (an affine morphism) induces an equivalence between $\Coh^{G \times \Gm}(F_1^\bot \times_{V^*} F_2^\bot)$ and the category of locally finitely generated $G \times \Gm$-equivariant modules over the $\calO_{X \times X}$-algebra $(\pi_{F_1^\bot,F_2^\bot})_* \calO_{F_1^\bot \times_{V^*} F_2^\bot}$. Since $\Gm$ acts trivially on $X \times X$, one can consider $(\pi_{F_1^\bot,F_2^\bot})_* \calO_{F_1^\bot \times_{V^*} F_2^\bot}$ as a graded $G$-equivariant $\calO_{X \times X}$-algebra, and this grading is concentrated in even degrees. Hence if $\calF$ is any module over this algebra, then we have $\calF=\calF^{\mathrm{even}} \oplus \calF^{\mathrm{odd}}$ where $\calF^{\mathrm{even}}$, resp.~$\calF^{\mathrm{odd}}$, is concentrated in even, resp.~odd, degrees. We denote by
\[
\bfi_{F_1^\bot, F_2^\bot} \colon \Kth^{G \times \Gm}(F_1^{\bot} \times_{V^*} F_2^{\bot}) \xrightarrow{\sim} \Kth^{G \times \Gm}(F_1^{\bot} \times_{V^*} F_2^{\bot})
\]
the automorphism which sends the class of a module $\calF=\calF^{\mathrm{even}} \oplus \calF^{\mathrm{odd}}$ as above to $[\calF^{\mathrm{even}}] - [\calF^{\mathrm{odd}}]$.

\subsection{Reminder on the equivariant Riemann--Roch theorem}


Let us recall the definition and the main properties of the ``equivariant Riemann--Roch morphism'' for a complex algebraic variety, following~\cite{EG}. (See also~\cite{BZ} for a more direct treatment, without much details.) Let $A$ be a complex linear algebraic group, acting on a complex algebraic variety $Y$. Then
we have a ``Riemann--Roch'' morphism
\[
\tau^A_Y \colon \Kth^A(Y) \to \coHc^A_\bullet(Y).
\]
More precisely, we define this morphism as the composition
\begin{equation}
\label{eqn:RRmorphism}
\Kth^A(Y) \longrightarrow \prod_{i \geq 0} \mathbb{Q} \otimes_{\mathbb{Z}} \mathsf{CH}^i_A(Y) \longrightarrow \prod_{i \in \mathbb{Z}} \, \coH_i^A(Y) = \coHc^A_\bullet(Y),
\end{equation}
where $\mathsf{CH}^i_A(Y)$ is the $i$-th equivariant Chow group, see~\cite[\S 1.2]{EG}, the first arrow is the morphism constructed in~\cite[Section~3]{EG}, and the second morphism is induced by the ``equivariant cycle map'' of~\cite[\S 2.8]{EG1}.

\begin{remark}
It follows from~\cite[Theorem~4.1]{EG} that the first morphism in~\eqref{eqn:RRmorphism} induces an isomorphism between a certain completion of $\mathbb{Q} \otimes_{\mathbb{Z}} \Kth^A(Y)$ and $\prod_{i \geq 0} \mathbb{Q} \otimes_{\mathbb{Z}} \mathsf{CH}^i_A(Y)$. Hence, if the equivariant cycle map is an isomorphism, a similar claim holds for our morphism $\tau^A_Y$.
\end{remark}


Below we will use the following properties of the map $\tau^A_Y$, which follow from the main results of~\cite{EG}.

\begin{thm}[Equivariant Riemann--Roch theorem]
\label{thm:equivariantRR}
If $f \colon Y \to Y'$ is an $A$-equivariant proper morphism, then we have
\[
\tau^A_{Y'} \circ \pdiK_f = \pdiH_f \circ \tau^A_Y.
\]
\end{thm}

\begin{proof}
By~\cite[Theorem~3.1(b)]{EG}, the first arrow in~\eqref{eqn:RRmorphism} is compatible with proper direct image morphisms (in the obvious sense). And by~\cite[p.~372]{fulton} the second arrow is also compatible with proper direct image morphisms, completing the proof. (More precisely, only the non-equivariant setting is considered in~\cite{fulton}, but the equivariant case follows, using the same arguments as in Footnote~\ref{fn:equiv-homology}.)
\end{proof}

If $F$ is an $A$-equivariant vector bundle over $Y$, then one can define its (cohomological) equivariant Chern classes in $\coH_A^\bullet(Y)$, and define a (cohomological) equivariant Todd class $\mathrm{td}^A(F) \in \coHc_A^\bullet(Y)$, see~\cite[Section~3]{EG} or~\cite[\S 3]{BZ} for similar constructions. This element is invertible in the algebra $\coHc_A^\bullet(Y)$.
If $Y$ is smooth, we denote by $\mathrm{Td}_Y^A$ the equivariant Todd class of the tangent bundle of $Y$.

The following
result can be stated and proved under much weaker assumptions, but only this particular case will be needed.

\begin{prop}
\label{prop:RR-res}
Let $Y$ be a smooth $A$-variety, and let $f \colon Z \hookrightarrow Y$ be the embedding of a smooth subvariety with normal bundle $N$. Then we have
\[
\resH_f \circ \tau^A_{Y}(x) = \bigl( \tau^A_Z \circ \resK_f(x) \bigr) \cdot \mathrm{td}^A(N)
\]
for any $x \in \Kth^A(Y)$, where $\resH_f \colon \coH^A_\bullet(Y) \to \coH^A_{\bullet-2\dim(Y)+2\dim(Z)}(Z)$ and $\resK_f \colon \Kth^A(Y) \to \Kth^A(Z)$ are the ``restriction with supports'' morphisms.
\end{prop}

\begin{proof}
A similar formula for the first arrow in~\eqref{eqn:RRmorphism} follows from~\cite[Theorem~3.1(d)]{EG}. To deduce our result we need to check that the equivariant cycle map commutes with restriction with supports and with multiplication by a Todd class. In the non-equivariant situation, the first claim follows from~\cite[Example~19.2.1]{fulton} and the second one from~\cite[Proposition~19.1.2]{fulton}. The equivariant case follows, using the same arguments
as in Footnote~\ref{fn:equiv-homology}.
\end{proof}

\begin{remark}
\label{rk:Todd}
Note that, in the setting of Proposition~\ref{prop:RR-res}, we have $f^* \mathrm{Td}^A_Y = \mathrm{Td}^A_Z \cdot \mathrm{td}^A(N)$, where $f^*$ is as in~\eqref{eqn:proj-formula-H}. (In fact, this formula easily follows from the compatibility of Chern classes with pullback and extensions of vector bundles.)
\end{remark}

Finally we will need the following fact, which follows from~\cite[Theorem~3.1(d)]{EG} applied to the projection $Y \to \mathrm{pt}$ (see also~\cite[Theorem~5.1]{BZ}).

\begin{prop}
\label{prop:tau-O-smooth}
If $Y$ is smooth, then
\[
\tau^A_Y(\calO_Y) = [Y] \cdot \mathrm{Td}^A_Y,
\]
where $[Y]$ is the equivariant fundamental class of $Y$ (i.e.~the image of the fundamental class in the Chow group from~\cite[\S 2.2]{EG1} under the cycle map).
\end{prop}

\subsection{Riemann--Roch maps}
\label{ss:RRmaps}

Following \cite[\S 5.11]{CG}, we consider the ``bivariant Riemann--Roch maps''
\begin{align*}
\underline{\mathrm{RR}}_{F_1, F_2} \colon & \ \Kth^{G \times \Gm}(F_1 \times_V F_2) \to \coHc^{G \times \Gm}_\bullet(F_1 \times_V F_2), \\
\overline{\mathrm{RR}}_{F^{\bot}_1, F_2^{\bot}} \colon & \ \Kth^{G \times \Gm}(F_1^{\bot} \times_{V^*} F_2^{\bot}) \to \coHc^{G \times \Gm}_\bullet(F_1^{\bot} \times_{V^*} F_2^{\bot})
\end{align*}
defined by
\begin{align*}
\underline{\mathrm{RR}}_{F_1, F_2} (c) & =  \tau^{G \times \Gm}_{F_1 \times_V F_2}(c) \cdot \bigl(1 \boxtimes (\mathrm{Td}_{F_2}^{G \times \Gm})^{-1} \bigr),\\
\overline{\mathrm{RR}}_{F^{\bot}_1, F_2^{\bot}}(d) & = \tau^{G \times \Gm}_{F_1^\bot \times_{V^*} F_2^\bot}(d) \cdot \bigl( (\mathrm{Td}_{F_1^\bot}^{G \times \Gm})^{-1} \cdot \mathrm{Td}^{G \times \Gm}_X \boxtimes (\mathrm{Td}_X^{G \times \Gm})^{-1} \bigr).
\end{align*}
In the expression for $\underline{\mathrm{RR}}_{F_1,F_2}$, $1 \boxtimes (\mathrm{Td}_{F_2}^{G \times \Gm})^{-1}$ is considered as an element of $\coHc_{G \times \Gm}^\bullet(F_1 \times_{V} F_2)$ through the composition 
\[
\coHc_{(G \times \Gm)^2}^\bullet(F_1 \times F_2) \to \coHc_{G \times \Gm}^\bullet(F_1 \times F_2) \to \coHc_{G \times \Gm}^\bullet(F_1 \times_{V} F_2)
\]
where the first morphism is the restriction morphism associated with the diagonal embedding of $G \times \Gm$, and the second morphism is the pullback in equivariant cohomology. In the expression for $\overline{\mathrm{RR}}_{F_1^\bot, F_2^\bot}$, first we consider $\mathrm{Td}_X^{G \times \Gm}$ as an element of $\coHc^\bullet_{G \times \Gm}(E^*)$ using the Thom isomorphism $\coH^\bullet_{G \times \Gm}(E^*) \xrightarrow{\sim} \coH^\bullet_{G \times \Gm}(X)$; then the same conventions as above allow to consider $(\mathrm{Td}_{F_1^\bot}^{G \times \Gm})^{-1} \cdot \mathrm{Td}^{G \times \Gm}_X \boxtimes (\mathrm{Td}_X^{G \times \Gm})^{-1}$ as an element in $\coHc_{G \times \Gm}^\bullet(F_1^\bot \times_{V^*} F_2^\bot)$.

\subsection{Statement}
\label{ss:statement}

The main result of this paper is the following. 

\begin{thm}
\label{thm:LKDFourier}

Assume that the proper direct image morphism 
\begin{equation}
\label{eqn:morphism-thm}
\coH^{G \times \Gm}_\bullet(F_1^{\bot} \times_{V^*} F_2^{\bot}) \to \coH^{G \times \Gm}_\bullet(F_1^{\bot} \times_{V^*} E^*)
\end{equation}
induced by the inclusion $F_2^\bot \hookrightarrow E^*$ is injective. Then the following diagram commutes:
\[
\xymatrix@C=6cm{
\Kth^{G \times \Gm}(F_1 \times_V F_2) \ar[r]^-{\bfi_{F_1^{\bot}, F_2^{\bot}} \circ \dual_{F_1^{\bot}, F_2^{\bot}} \circ \Koszul_{F_1, F_2}} \ar[d]_-{\underline{\mathrm{RR}}_{F_1, F_2}} & \Kth^{G \times \Gm}(F_1^{\bot} \times_{V^*} F_2^{\bot}) \ar[d]^-{\overline{\mathrm{RR}}_{F_1^{\bot}, F_2^{\bot}}} \\
\coHc^{G \times \Gm}_\bullet(F_1 \times_V F_2) \ar[r]^-{\Fourier_{F_1,F_2}} & \coHc^{G \times \Gm}_\bullet(F_1^{\bot} \times_{V^*} F_2^{\bot}).
}
\]

\end{thm}

The proof of Theorem \ref{thm:LKDFourier} is given in \S\ref{ss:proof-thm}. It is based on compatibility (or functoriality) results for all the maps considered in the diagram, which are stated in Sections \ref{sec:compatibility-Fourier} and \ref{sec:compatibility-others}; some of these results might be of independent interest. Let us point out that our assumption is probably not needed for the result to hold.

\begin{remark}
\label{rmk:injectivity-faithful}
{\em (Injectivity assumption.)}
The fiber product $F_1^{\bot} \times_{V^*} E^*$ is isomorphic to $F_1^\bot \times X$, hence is a vector bundle over $X^2$. In particular, by the Thom isomorphism we have
\begin{equation}
\label{eqn:Thom}
\coH^{G \times \Gm}_\bullet(F_1^{\bot} \times_{V^*} E^*) \cong \coH^{G \times \Gm}_{\bullet-2\rk(F_1^\bot)}(X \times X).
\end{equation}
Moreover, by \cite[Lemma 5.4.35]{CG} the following diagram commutes:
\begin{equation}
\label{eqn:diagram-assumption}
\vcenter{
\xymatrix@C=1cm{
\coH^{G \times \Gm}_\bullet(F_1^{\bot} \times_{V^*} F_2^{\bot}) \ar[r] \ar[d] & \Hom_{\coH^\bullet_{G \times \Gm}(\mathrm{pt})} \bigl( \coH^{G \times \Gm}_\bullet(F_2^\bot), \coH^{G \times \Gm}_{\bullet-2\dim(F_2^\bot)}(F_1^\bot) \bigr) \ar[d]^-{\wr} \\
\coH^{G \times \Gm}_{\bullet-2\rk(F_1^\bot)}(X \times X) \ar[r] & \Hom_{\coH^\bullet_{G \times \Gm}(\mathrm{pt})} \bigl( \coH^{G \times \Gm}_{\bullet-2\rk(F_2^\bot)}(X), \coH^{G \times \Gm}_{\bullet-2\dim(F_2^\bot)-2\rk(F_1^\bot)}(X) \bigr).
}
}
\end{equation}
Here the horizontal arrows are induced by convolution, the left vertical arrow is the composition of~\eqref{eqn:morphism-thm} and the isomorphism~\eqref{eqn:Thom}, and the right vertical arrow is induced by the respective Thom isomorphisms.
Assume now that $\coH_c^{\mathrm{odd}}(X)=0$ (e.g.~that $X$ is paved by affine spaces). Then one can easily check that the lower horizontal arrow in diagram \eqref{eqn:diagram-assumption} is an isomorphism. Hence in this case our assumption is equivalent to injectivity of the upper horizontal arrow. If moreover $F_1=F_2=F$, then $\coH^{G \times \Gm}_\bullet(F^{\bot} \times_{V^*} F^{\bot})$ is an algebra and $\coH^{G \times \Gm}_\bullet(F^\bot)$ is a module over this algebra. In this case our assumption amounts to the condition that the action on this module is faithful.
\end{remark}

\subsection{An injectivity criterion for \eqref{eqn:morphism-thm}}

The following result gives an easy criterion which ensures that the assumption of Theorem \ref{thm:LKDFourier} is satisfied.

\begin{prop}
\label{prop:criterion-thm}
Assume that $\coH_c^{\mathrm{odd}}(F_1^\bot \times_{V^*} F_2^\bot)=0$. Then the proper direct image morphism
\[
\coH^{G \times \Gm}_\bullet(F_1^{\bot} \times_{V^*} F_2^{\bot}) \to \coH^{G \times \Gm}_\bullet(F_1^{\bot} \times_{V^*} E^*)
\]
induced by the inclusion $F_2^\bot \hookrightarrow E^*$ is injective.
\end{prop}

\begin{proof}
Let $T$ be a maximal torus of $G$. Then we have a commutative diagram
\[
\xymatrix@C=2cm{
\coH^{G \times \Gm}_\bullet(F_1^\bot \times_{V^*} F_2^\bot) \ar[r] \ar[d] & \coH^{G \times \Gm}_\bullet(F_1^\bot \times_{V^*} E^*) \ar[d] \\
\coH^{T \times \Gm}_\bullet(F_1^\bot \times_{V^*} F_2^\bot) \ar[r] & \coH^{T \times \Gm}_\bullet(F_1^\bot \times_{V^*} E^*) \\
}
\]
where horizontal arrows are proper direct image morphisms, and vertical arrows are forgetful maps. The left vertical arrow is injective: indeed, by our assumption and \cite[Proposition 7.2]{LuCus1}, there exist (non-canonical) isomorphisms
\begin{align}
\coH^{G \times \Gm}_\bullet(F_1^\bot \times_{V^*} F_2^\bot) & \cong \coH_{G \times \Gm}^{-\bullet}(\pt) \otimes_{\C} \coH_{\bullet} (F_1^\bot \times_{V^*} F_2^\bot), \\
\label{eqn:isom-T-equ}
\coH^{T \times \Gm}_\bullet(F_1^\bot \times_{V^*} F_2^\bot) & \cong \coH_{T \times \Gm}^{-\bullet}(\pt) \otimes_{\C} \coH_{\bullet} (F_1^\bot \times_{V^*} F_2^\bot)
\end{align}
such that our forgetful morphism is induced by the natural morphism $\coH_{G \times \Gm}^\bullet(\pt) \to \coH_{T \times \Gm}^\bullet(\pt)$, which is well known to be injective. Hence, to prove that the upper horizontal arrow is injective it is sufficient to prove that the lower horizontal arrow is injective.

If $\mathsf{Q}$ denotes the fraction field of $\coH:=\coH_{T \times \Gm}^\bullet(\pt)$, then using again isomorphism \eqref{eqn:isom-T-equ}, the natural morphism
\[
\coH^{T \times \Gm}_\bullet(F_1^\bot \times_{V^*} F_2^\bot) \to \mathsf{Q} \otimes_\coH \coH^{T \times \Gm}_\bullet(F_1^\bot \times_{V^*} F_2^\bot)
\]
is injective. We deduce that to prove the proposition it suffices to prove that the induced morphism
\[
\mathsf{Q} \otimes_\coH \coH^{T \times \Gm}_\bullet(F_1^{\bot} \times_{V^*} F_2^{\bot}) \to \mathsf{Q} \otimes_\coH \coH^{T \times \Gm}_\bullet(F_1^{\bot} \times_{V^*} E^*)
\]
is injective. Let $Y:=(X \times X)^T$ denote the $T$-invariants in $X \times X$. Then we have
\[
Y=(F_1^{\bot} \times_{V^*} F_2^{\bot})^{T \times \Gm} = (F_1^{\bot} \times_{V^*} E^*)^{T \times \Gm}.
\]
Consider the commutative diagram
\[
\xymatrix{
\coH^{T \times \Gm}_\bullet(F_1^\bot \times_{V^*} F_2^\bot) \ar[rr]^-{\alpha} & & \coH^{T \times \Gm}_\bullet(F_1^\bot \times_{V^*} E^*) \\
& \coH^{T \times \Gm}_\bullet(Y) \ar[lu]^-{\beta} \ar[ru]_-{\gamma}  &
}
\]
where all morphisms are proper direct image morphisms in homology. Then by the localization theorem (see \cite[Proposition 4.4]{LuCus2} or \cite[Theorem B.2]{EM}) both $\beta$ and $\gamma$ become isomorphisms after applying $\mathsf{Q} \otimes_\coH (\cdot)$. Hence the same is true for $\alpha$; in particular $\mathrm{id}_\mathsf{Q} \otimes_\coH \alpha$ is injective, which finishes the proof.
\end{proof}

\begin{remark}
\label{rk:odd-vanishing}
Using a non-equivariant variant of isomorphism $\Fourier_{F_1,F_2}$,
one can check that the condition $\coH_c^{\mathrm{odd}}(F_1^\bot \times_{V^*} F_2^\bot)=0$ is equivalent to the condition $\coH_c^{\mathrm{odd}}(F_1 \times_{V} F_2)=0$.
\end{remark}

\section{The case of convolution algebras}
\label{sec:convolution-algebras}

In this subsection we study more closely the case $F_1=F_2$. In this case, as we will explain, all the objects appearing in the diagram of Theorem~\ref{thm:LKDFourier} are equipped with convolution products, and all the maps are compatible with these products. In a particular case, these algebras are related to affine Hecke algebras, and our diagram explains the relation between the categorifications of Iwahori--Matsumoto involutions obtained in~\cite{EM} and~\cite{MR3}, via maps introduced in~\cite{LuAff}.

None of the results of this section are used in the proof of Theorem~\ref{thm:LKDFourier}.

\subsection{Convolution}
\label{ss:convolution}


We set $F:=F_1=F_2$. As explained in~\cite[\S 5.2.20]{CG} or~\cite[\S 4.1]{MR3}, the group $\Kth^{G \times \Gm}(F \times_V F)$ can be endowed with a natural (associative and unital) convolution product $\star$. In fact, for $c,d \in \Kth^{G \times \Gm}(F \times_V F)$, with our notations this product satisfies\footnote{Note that our convention for the definition of the convolution product is opposite to the one adopted in~\cite{MR3}.\label{fn:convention-product}}
\[
c \star d = \pdiK_{p_{1,3}} \circ \resK(c \boxtimes d)
\]
where $c \boxtimes d \in \Kth^{G \times \Gm}\bigl( (F \times_V F) \times (F \times_V F) \bigr)$ is the exterior product of $c$ and $d$,
\[
\resK \colon \Kth^{G \times \Gm}\bigl( (F \times_V F) \times (F \times_V F) \bigr) \to \Kth^{G \times \Gm}(F \times_V F \times_V F)
\]
is the restriction with supports morphism associated with the inclusion $F^3 \hookrightarrow F^4$ sending $(x,y,z)$ to $(x,y,y,z)$, and $p_{1,3} \colon F \times_V F \times_V F \to F \times_V F$ is the (proper) projection on the first and third factors. (See~\cite[\S 4.2]{MR3} for a similar description at the categorical level.) The unit in this algebra is the structure sheaf $\calO_{\Delta F}$ of the diagonal $\Delta F \subset F \times_V F$. The same constructions provide a left, resp.~right, action of the algebra $\Kth^{G \times \Gm}(F \times_V F)$ on the group $\Kth^{G \times \Gm}(F)$ defined by
\[
c \star d = \pdiK_{p_{1}} \circ \resK_{\mathrm{l}}(c \boxtimes d), \qquad \text{resp.} \qquad d \star c = \pdiK_{p_{2}} \circ \resK_{\mathrm{r}}(d \boxtimes c)
\]
for $c \in \Kth^{G \times \Gm}(F \times_V F)$ and $d \in \Kth^{G \times \Gm}(F)$. Here $p_1,p_2 \colon F \times_V F \to F$ are the projections on the first and second factor respectively, the exterior products are defined in the obvious way, and
\begin{multline*}
\resK_{\mathrm{l}} \colon \Kth^{G \times \Gm}\bigl( (F \times_V F) \times F \bigr) \to \Kth^{G \times \Gm}(F \times_V F), \\
\text{resp.} \quad \resK_{\mathrm{r}} \colon \Kth^{G \times \Gm}\bigl( F \times (F \times_V F) \bigr) \to \Kth^{G \times \Gm}(F \times_V F),
\end{multline*}
is the restriction with supports morphism associated with the inclusion $F^2 \hookrightarrow F^3$ sending $(x,y)$ to $(x,y,y)$, resp.~to $(x,x,y)$.

Of course we have similar constructions for the subbundle $F^\bot \subset E^*$, and we will use the same notation in this context.


\begin{lem}
The morphisms $\Koszul_{F,F}$, $\dual_{F^{\bot}, F^{\bot}}$ and $\bfi_{F^\bot, F^\bot}$ are (unital) algebra isomorphisms.
\end{lem}

\begin{proof}
The case of $\Koszul_{F,F}$ follows from~\cite[Propositions~4.3 \&~4.5]{MR3}.\footnote{In~\cite{MR3} we use the dualizing complex $\omega_X \boxtimes \calO_X [\dim(X)]$ instead of $\calO_X \boxtimes \omega_X[\dim(X)]$. But the results cited remain true (with an identical proof) with our present conventions. \label{fn:MR3}} The case of $\dual_{F^{\bot}, F^{\bot}}$ is not difficult, and left to the reader (see~\cite[Lemma~9.5]{LUSBas} for a similar statement, with slightly different conventions in the definition of Grothendieck--Serre duality). Finally, the case of $\bfi_{F^\bot, F^\bot}$ is obvious.
\end{proof}

This convolution construction has a natural analogue in equivariant Borel--Moore homology, see e.g.~\cite[\S 2.7]{CG} or~\cite[\S 2]{LuCus2}. In fact, the convolution product on
$\coH_\bullet^{G \times \Gm}(F \times_V F)$, which we will also denote $\star$, satisfies
\[
c \star d = \pdiH_{p_{1,3}} \circ \resH(c \boxtimes d),
\]
where $\resH$ is defined as for $\resK$ above (replacing $\Kth$-homology by Borel--Moore homology). The unit for this convolution product is the equivariant fundamental class $[\Delta F]$ of the diagonal $\Delta F \subset F \times_V F$.
We also have a left and a right module structure on $\coH_\bullet^{G \times \Gm}(F)$, defined via the formulas
\[
c \star d = \pdiH_{p_{1}} \circ \resH_{\mathrm{l}}(c \boxtimes d), \qquad \text{resp.} \qquad d \star c = \pdiH_{p_{2}} \circ \resH_{\mathrm{r}}(d \boxtimes c)
\]
for $c \in \coH^{G \times \Gm}_\bullet(F \times_V F)$ and $d \in \coH_\bullet^{G \times \Gm}(F)$. Finally we have similar structures for the subbundle $F^\bot \subset E^*$.

\begin{lem}
The morphism $\Fourier_{F,F}$ is a (unital) algebra isomorphism.
\end{lem}

\begin{proof}
One can also show that the isomorphism $\can_{F,F}$ is a (unital) algebra isomorphism, where the right-hand side is endowed with the Yoneda product; see~\cite[Theorem~8.6.7]{CG}, \cite[Lemma~2.5]{LuCus2} or~\cite[Theorem~4.5]{Ka} for similar statements. Then the claim follows from the fact that $\Fourier_{F,F}$ is induced by a functor.
\end{proof}

\subsection{Compatibility for the Riemann--Roch maps}
\label{ss:RR-convolution}

\begin{lem}
\label{lem:RR-convolution}
Assume\footnote{This assumption is probably unnecessary. However, to avoid it one would need a more general variant of Proposition~\ref{prop:RR-res} (as in~\cite[Theorem~5.8.14]{CG}, for instance) for which we could not find any reference or easy proof.} that
$\coH_c^{\mathrm{odd}}(F \times_{V} F)=0$. Then the morphisms $\uRR_{F,F}$ and $\oRR_{F^\bot, F^\bot}$ are unital algebra morphisms.
\end{lem}

\begin{proof}
We only treat the case of $\uRR_{F,F}$; the case of $\oRR_{F^\bot, F^\bot}$ is similar. (Note that,
by Remark~\ref{rk:odd-vanishing}, our ``odd vanishing'' assumption implies that $\coH_c^{\mathrm{odd}}(F^\bot \times_{V^*} F^\bot)=0$ also.) The fact that our morphism sends the unit to the unit follows from Theorem~\ref{thm:equivariantRR} and Proposition~\ref{prop:tau-O-smooth}, using the projection formula~\eqref{eqn:proj-formula-H}. It remains to prove the compatibility with products.

To prove the lemma we use ``projective completions,'' namely we set $\overline{V}:=\mathbb{P}(V \oplus \C)$ and let $\overline{F}$ be the projective bundle associated with the vector bundle $F \times \C$ over $X$. Then we have a projection $\overline{F} \to \overline{V}$, and open embeddings $F \hookrightarrow \overline{F}$, $V \hookrightarrow \overline{V}$. Note that $F \times_V F = F \times_{\overline{V}} \overline{F}$, so that $F \times_V F$ is a closed subvariety in $F \times \overline{F}$. Similarly, one can identify $F \times_V F$ with a closed subvariety in $\overline{F} \times F$, so that we have proper direct image morphisms
\begin{gather*}
\imath_1 \colon \coH_\bullet^{G \times \Gm}(F \times_V F) \to \coH_\bullet^{G \times \Gm}(F \times \overline{F}), \qquad \imath_2 \colon \coH_\bullet^{G \times \Gm}(F \times_V F) \to \coH_\bullet^{G \times \Gm}(\overline{F} \times F), \\
\imath_3 \colon \coH_\bullet^{G \times \Gm}(F \times_V F) \to \coH_\bullet^{G \times \Gm}(F \times F).
\end{gather*}
Using the same arguments as in the proof of Proposition~\ref{prop:criterion-thm}, one can check that the morphism $\imath_3$ is injective under our assumption. There exists a natural convolution product
\[
\star \colon \coH_\bullet^{G \times \Gm}(F \times \overline{F}) \times \coH_\bullet^{G \times \Gm}(\overline{F} \times F) \to \coH_\bullet^{G \times \Gm}(F \times F)
\]
defined by 
\[
c \star d = \pdiH_{p'_{1,3}} \circ \resH'(c \boxtimes d),
\]
where $p'_{1,3} \colon F \times \overline{F} \times F \to F \times F$ is the (proper) projection on the first and third factors, and $\resH \colon \coH_\bullet^{G \times \Gm}(F \times \overline{F} \times \overline{F} \times F) \to  \coH_\bullet^{G \times \Gm}(F \times \overline{F} \times F)$ is the restriction with supports morphism associated with the inclusion sending $(x,y,z)$ to $(x,y,y,z)$. Moreover one can check (using in particular Lemma~\ref{lem:restriction-pushforward}) that for $c,d \in \coH_\bullet^{G \times \Gm}(F \times_V F)$ we have
\[
\imath_3(c \star d) = \imath_1(c) \star \imath_2(d).
\]

We have a similar construction of a convolution product in equivariant $\Kth$-homology, for which we will use similar notations.
Hence, using the injectivity of $\imath_3$, Theorem~\ref{thm:equivariantRR} and the projection formula~\eqref{eqn:proj-formula-H}, to prove the lemma it is enough to prove that
\begin{equation}
\label{eqn:convolution-RR}
\uRR_3(c \star d) = \uRR_1(c) \star \uRR_2(d)
\end{equation}
for $c \in \Kth^{G \times \Gm}(F \times \overline{F})$ and $d \in \Kth^{G \times \Gm}(\overline{F} \times F)$, where
\[
\uRR_1 \colon \Kth^{G \times \Gm}(F \times \overline{F}) \to \coHc_\bullet^{G \times \Gm}(F \times \overline{F})
\]
is defined by
\[
\uRR_1(d) = \tau^{G \times \Gm}_{F \times \overline{F}}(d) \cdot \big(1 \boxtimes (\mathrm{Td}_{\overline{F}}^{G \times \Gm})^{-1} \bigr),
\]
and $\uRR_2$ and $\uRR_3$ are defined similarly.

Now we have
\begin{align*}
\uRR_3(c \star d) & = \tau^{G \times \Gm}_{F \times F}(\pdiK_{p'_{1,3}} \circ \resK'(c \boxtimes d)) \cdot (1 \boxtimes (\mathrm{Td}_{F}^{G \times \Gm})^{-1}) \\
& = \pdiH_{p'_{1,3}} \bigl( \tau^{G \times \Gm}_{F \times \overline{F} \times F}(\resK'(c \boxtimes d)) \bigr) \cdot (1 \boxtimes (\mathrm{Td}_{F}^{G \times \Gm})^{-1}) \\
& = \pdiH_{p'_{1,3}} \bigl( \tau^{G \times \Gm}_{F \times \overline{F} \times F}(\resK'(c \boxtimes d)) \cdot (1 \boxtimes 1 \boxtimes (\mathrm{Td}_{F}^{G \times \Gm})^{-1}) \bigr) \\
& = \pdiH_{p'_{1,3}}  \bigl( \resH' \circ  \tau^{G \times \Gm}_{F \times \overline{F} \times \overline{F} \times F}(c \boxtimes d) \cdot \mathrm{td}^{G \times \Gm}(N) \cdot (1 \boxtimes 1 \boxtimes (\mathrm{Td}_{F}^{G \times \Gm})^{-1}) \bigr),
\end{align*}
where $N$ is the normal bundle to the embedding $F \times \overline{F} \times F \hookrightarrow F \times \overline{F}^2 \times F$.
(Here the second equality follows from Theorem~\ref{thm:equivariantRR}, the third one from the projection formula~\eqref{eqn:proj-formula-H}, and the last equality from Proposition~\ref{prop:RR-res}.) On the other hand we have
\[
\uRR_1(c) \star \uRR_2(d) = \pdiH_{p'_{1,3}}  \circ \resH' \bigl( \tau^{G \times \Gm}_{F \times \overline{F} \times \overline{F} \times F}(c \boxtimes d) \cdot (1 \boxtimes (\mathrm{Td}_{\overline{F}}^{G \times \Gm})^{-1}) \boxtimes 1 \boxtimes (\mathrm{Td}_{F}^{G \times \Gm})^{-1})\bigr).
\]
Now the normal bundle $N$ is canonically isomorphic to the restriction to $F \times \overline{F} \times F$ of the pullback of the tangent bundle of $\overline{F}$ under the projection $F \times \overline{F}^2 \times F \to \overline{F}$ on the second factor. Using~\eqref{eqn:res-cohomology} and comparing the formulas for $\uRR_3(c \star d)$ and for $\uRR_1(c) \star \uRR_2(d)$ obtained above, we deduce~\eqref{eqn:convolution-RR}.
\end{proof}

\subsection{Compatibility for the actions on the natural modules}

In~\S\ref{ss:convolution} we have defined (left and right) actions of the algebra $\Kth^{G \times \Gm}(F \times_V F)$, resp.~$\Kth^{G \times \Gm}(F^\bot \times_{V^*} F^\bot)$, resp.~$\coH^{G \times \Gm}(F \times_V F)$, resp.~$\coH^{G \times \Gm}(F^\bot \times_{V^*} F^\bot)$, on the module $\Kth^{G \times \Gm}(F)$, resp.~$\Kth^{G \times \Gm}(F^\bot)$, resp.~$\coH^{G \times \Gm}(F)$, resp.~$\coH^{G \times \Gm}(F^\bot)$. We now define ``bivariant Riemann--Roch maps''
\[
\uRR_F \colon \Kth^{G \times \Gm}(F) \to \coHc^{G \times \Gm}_\bullet(F), \qquad \oRR_{F^\bot} \colon \Kth^{G \times \Gm}(F^\bot) \to \coHc^{G \times \Gm}_\bullet(F^\bot)
\]
by the formulas
\[
\uRR_F = \tau^{G \times \Gm}_F, \qquad \oRR_{F^\bot}(c) = \tau^{G \times \Gm}_{F^\bot}(c) \cdot (\mathrm{Td}_X^{G \times \Gm})^{-1}
\]
(where use the same conventions as in~\S\ref{ss:RRmaps}).
The following technical lemma will be used to compute explicitly some Riemann--Roch maps in~\S\ref{ss:geom-gHaff}.

\begin{lem}
\label{lem:RR-action}
Assume that $\coH_c^{\mathrm{odd}}(F \times_{V} F)=0$. Then
the morphisms $\uRR$ and $\oRR$ are compatible with the module structures, in the sense that for $c \in \Kth^{G \times \Gm}(F \times_V F)$ and $d \in \Kth^{G \times \Gm}(F)$, resp.~for $c \in \Kth^{G \times \Gm}(F^\bot \times_{V^*} F^\bot)$ and $d \in \Kth^{G \times \Gm}(F^\bot)$, we have
\[
\uRR_F(c \star d) = \uRR_{F,F}(c) \star \uRR_F(d), \qquad \text{resp.} \qquad \oRR_{F^\bot}(d \star c) = \oRR_{F^\bot}(d) \star \oRR_{F^\bot, F^\bot}(c).
\]
\end{lem}

\begin{proof}
We only prove the first equality; the second one can be proved by similar arguments. First, we claim that
\begin{equation}
\label{eqn:RR-modules}
\tau^{G \times \Gm}_{F \times_V F} \circ \resK_{\mathrm{l}}(c \boxtimes d) = \bigl( \resH_{\mathrm{l}} \circ \tau^{G \times \Gm}_{F \times_V F \times F} (c \boxtimes d) \bigr) \cdot \mathrm{td}^A(N)^{-1},
\end{equation}
where $N$ is the normal bundle to the inclusion $F \times F \hookrightarrow F \times F \times F$ considered in the definition of $\resK_{\mathrm{l}}$. Indeed, as in the proof of Lemma~\ref{lem:RR-convolution}, our assumption ensures that the proper direct image morphism
\[
\imath \colon \coH_\bullet^{G \times \Gm}(F \times_V F) \to \coH_\bullet^{G \times \Gm}(F \times F)
\]
is injective. Hence it is enough to prove that the image under $\imath$ of both sides in~\eqref{eqn:RR-modules} are equal. Now by the projection formula~\eqref{eqn:proj-formula-H}, Theorem~\ref{thm:equivariantRR} and Lemma~\ref{lem:restriction-pushforward} we have
\begin{multline*}
\imath \Bigl( \bigl( \resH_{\mathrm{l}} \circ \tau^{G \times \Gm}_{F \times_V F} (c \boxtimes d) \bigr) \cdot \mathrm{td}^A(N)^{-1} \Bigr) = \imath \bigl( \resH_{\mathrm{l}} \circ \tau^{G \times \Gm}_{F \times_V F \times F} (c \boxtimes d) \bigr) \cdot \mathrm{td}^A(N)^{-1} \\
= \bigl( \resH_{\mathrm{l}}' \circ \tau^{G \times \Gm}_{F \times F \times F} (\imath(c) \boxtimes d) \bigr) \cdot \mathrm{td}^A(N)^{-1},
\end{multline*}
where
\[
\resH_{\mathrm{l}}' \colon \coH_\bullet^{G \times \Gm}( F \times F \times F) \to \coH_\bullet^{G \times \Gm}(F \times F)
\]
is the restriction with supports morphism associated with the embedding $F^2 \hookrightarrow F^3$ considered in the definition of $\resH_{\mathrm{l}}$.

On the other hand, by Theorem~\ref{thm:equivariantRR} and the obvious $\Kth$-theoretic analogue of Lemma~\ref{lem:restriction-pushforward} we have
\[
\imath \bigl( \tau^{G \times \Gm}_{F \times_V F} \circ \resK_{\mathrm{l}}(c \boxtimes d) \bigr) = \tau^{G \times \Gm}_{F \times F} \circ \resK'_{\mathrm{l}}(\imath(c) \boxtimes d),
\]
where $\resK'_{\mathrm{l}}$ is defined as for $\resH_{\mathrm{l}}'$. Hence the desired equality follows from
Proposition~\ref{prop:RR-res}.

Now we have
\begin{align*}
\uRR_F(c \star d) &= \tau_F^{G \times \Gm}(\pdiK_{p_{1}} \circ \resK_{\mathrm{l}}(c \boxtimes d)) \\
&= \pdiH_{p_1} \circ \tau_{F \times_V F}^{G \times \Gm} \circ \resK_{\mathrm{l}}(c \boxtimes d) \\
&= \pdiH_{p_1} \Bigl( \bigl( \resH_{\mathrm{l}} \circ \tau^{G \times \Gm}_{F \times_V F \times F} (c \boxtimes d) \bigr) \cdot \mathrm{td}^A(N)^{-1} \Bigr) \\
&= \pdiH_{p_1} \Bigl( \resH_{\mathrm{l}} \bigl( (\tau^{G \times \Gm}_{F \times_V F}(c) \boxtimes \tau^{G \times \Gm}_F (d) ) \cdot (1 \boxtimes (\mathrm{Td}^A_F)^{-1} \boxtimes 1) \bigr) \Bigr) \\
&= \uRR_{F,F}(c) \star \uRR_F(d).
\end{align*}
(Here the second equality follows from Theorem~\ref{thm:equivariantRR}, the third one from~\eqref{eqn:RR-modules}, and the fourth one from~\eqref{eqn:res-cohomology}.) This concludes the proof.
\end{proof}

\subsection{Affine Hecke algebras and their graded versions}
\label{ss:Haff-gHaff}

From now on in this section we restrict to the case of the affine Hecke algebra and its graded version. Our notation mainly follows~\cite{LuAff}. Namely,
we fix a semisimple and simply connected complex algebraic group $G$, with fixed maximal torus $T$ and Borel subgroup $B$ with $T \subset B$. We denote by $W$ the Weyl group of $(G,T)$, and by $S \subset W$ the set of Coxeter generators determined by the choice of $B$. We also denote by $\mathbb{X}$ the lattice of characters of $T$, and by $R \subset \bbX$ the root system of $(G,T)$. We denote by $R^+ \subset R$ the system of positive roots consisting of the roots \emph{opposite} to the roots of $B$. Then the \emph{affine Hecke algebra} $\Haff$ (with equal parameters) attached to these data is the $\Z[v,v^{-1}]$-algebra generated by elements $T_s$ for $s \in S$ and $\theta_x$ for $x \in \bbX$, subject to the following relations (where $m_{s,t}$ is the order of $st$ in $W$):
\begin{enumerate}
\item
$(T_s+1)(T_s-v^2)=0$ for $s \in S$;
\item
$T_s T_t \cdots = T_t T_s \cdots$ for $s,t \in S$ (with $m_{s,t}$ factors on each side);
\item
\label{it:Haff-rel-2}
$\theta_x \theta_y = \theta_{x+y}$ for $x,y \in \bbX$;
\item
\label{it:Haff-rel-3}
$\theta_0=1$;
\item
\label{it:Haff-rel-4}
$T_s \cdot \theta_x - \theta_{sx} \cdot T_s = (v^2-1) \frac{\theta_x - \theta_{sx}}{1 - \theta_{-\alpha}}$ for $s \in S$, where $\alpha \in R$ is the corresponding simple root. 
\end{enumerate}

\begin{remark}
\label{rk:subalg-theta}
\begin{enumerate}
\item
Relations~\eqref{it:Haff-rel-2} and~\eqref{it:Haff-rel-3} imply that the subalgebra generated by the generators $\theta_x$ for $x \in \bbX$ is isomorphic to the group algebra $\Z[v,v^{-1}][\bbX]$; then the quotient in the right-hand side in~\eqref{it:Haff-rel-4} denotes the quotient in this integral ring.
\item
The present notation differs slightly from the notation in~\cite{MR3}. In fact the element denoted $T_s$ here coincides with the element denoted $t_\alpha$ in~\cite[\S 5.2]{MR3} (for $\alpha$ the corresponding simple root).
\end{enumerate}
\end{remark}

The following reformulation of relation~\eqref{it:Haff-rel-4} (see~\cite[Proposition~3.9]{LuAff}) will be useful:
\begin{equation}
\label{eqn:comm-relation-Haff}
(T_s+1) \cdot \theta_x - \theta_{sx} \cdot (T_s+1) = (\theta_x-\theta_{sx}) \cdot \mathscr{G}(\alpha) \qquad \text{with} \qquad \mathscr{G}(\alpha)=\frac{v^2 \theta_\alpha - 1}{\theta_\alpha-1}.
\end{equation}

The subalgebra of $\Haff$ generated by the elements $T_s$ ($s \in S$) can be identified with the Hecke algebra $\calH_W$ of the Coxeter group $(W,S)$. We will consider the left module $\mathrm{sgn}_{\mathrm{l}}$ of this subalgebra which is (canonically) free of rank one over $\Z[v,v^{-1}]$, and where $T_s$ acts by $-1$. The same recipe also defines a right module $\mathrm{sgn}_{\mathrm{r}}$ over $\calH_W$. Then we can define the ``antispherical'' left, resp.~right, module over $\Haff$ as
\[
\calM^{\mathrm{asph}}_{\mathrm{l}} := \Haff \otimes_{\calH_W} \mathrm{sgn}_{\mathrm{l}}, \qquad \calM^{\mathrm{asph}}_{\mathrm{r}} := \mathrm{sgn}_{\mathrm{r}} \otimes_{\calH_W} \Haff.
\]
For both modules, we will simply denote by $1$ the ``base point'' $1 \otimes 1$.

We will also consider the associated \emph{graded affine Hecke algebra} $\gHaff$ (again, with equal parameters). This algebra is the $\C[r]$-algebra generated by $\calO(\frakt)=\mathrm{S}(\frakt^*)$ (where $\frakt$ is the Lie algebra of $T$) and elements $t_w$ for $w \in W$, subject to the following relations:
\begin{enumerate}
\item
$t_1=1$;
\item 
$t_v t_w = t_{vw}$ for $v,w \in W$;
\item
\label{it:gHaff-rel-2}
$t_s \cdot \phi - s(\phi) t_s = (\phi - s(\phi)) \cdot (g(\alpha) - 1)$ for $s \in S$, where $\alpha \in R$ is the corresponding simple root.
\end{enumerate}
Here following~\cite{LuAff} we have used the notation
\[
g(\alpha) = \frac{\dot{\alpha}+ 2 r}{\dot{\alpha}},
\]
where $\dot{\alpha} \in \frakt^*$ is the differential of the root $\alpha$. In this case also, one can reformulate relation~\eqref{it:gHaff-rel-2} in the following form, see~\cite[4.6(c)]{LuAff}:
\begin{equation}
\label{eqn:comm-relation-gHaff}
(t_s+1) \cdot \phi - s(\phi) \cdot (t_s+1) = (\phi - s(\phi)) \cdot g(\alpha).
\end{equation}

The subalgebra of $\gHaff$ generated by the elements $t_w$ (for $w \in W$) identifies with the group algebra $\overline{\calH}_W=\C[r][W]$. As above one can define a ``sign'' left, resp.~right, module over this algebra (where $t_s$ acts by $-1$ for $s \in S$), which we will denote by $\overline{\mathrm{sgn}}_{\mathrm{l}}$, resp.~$\overline{\mathrm{sgn}}_{\mathrm{r}}$, and corresponding ``antispherical'' modules
\[
\overline{\calM}^{\mathrm{asph}}_{\mathrm{l}} := \gHaff \otimes_{\overline{\calH}_W} \overline{\mathrm{sgn}}_{\mathrm{l}}, \qquad \overline{\calM}^{\mathrm{asph}}_{\mathrm{r}} := \overline{\mathrm{sgn}}_{\mathrm{r}} \otimes_{\overline{\calH}_W} \gHaff.
\]

Let $\mathfrak{m} \subset \calO(\frakt)[r]=\calO(\frakt \times \mathbb{A}^1)$ denote the maximal ideal associated with the point $(0,0) \in \frakt \times \mathbb{A}^1$, and let $\widehat{\calO(\frakt)[r]}$ be the $\mathfrak{m}$-adic completion of $\calO(\frakt)[r]$. Then $\hgHaff:=\widehat{\calO(\frakt)[r]} \otimes_{\calO(\frakt)[r]} \gHaff$ has a natural algebra structure extending the structure on $\gHaff$. With this notation introduced,
the algebras $\Haff$ and $\gHaff$ are related by the Lusztig morphism
\[
\mathscr{L}_{\mathrm{r}} \colon \Haff \to \hgHaff
\]
defined in~\cite[\S 9]{LuAff}.\footnote{The setting considered in~\cite[\S 9]{LuAff} is much more general than the case considered in the present paper. With Lusztig's notation, we only consider the case $v_0=1$ (which is covered by~\cite[\S 9.7]{LuAff}), $r_0=0$, $t_0=1$, $\overline{\Sigma}=\{0\}$.
This case suffices (except in the case when $v$ is specialized to a non trivial root of unity) for the study of the representation theory of $\Haff$
via the (more accessible) study of the representation theory of $\gHaff$; see~\cite{LuAff} for details.\label{fn:setting-lusztig}}
Let us recall the definition of this morphism.
First, we denote by $\bbY:=X_*(T)$ the lattice of cocharacters of $T$, and consider the map
\[
e \colon \left\{
\begin{array}{ccc}
\frakt = \bbY \otimes_\Z \C & \to & T = \bbY \otimes_\Z \C^\times \\
\lambda^\vee \otimes a & \mapsto & \lambda^\vee \otimes \exp(a)
\end{array}
\right. .
\]
This map induces a map
\[
\Z[v,v^{-1}][\bbX] \to \widehat{\calO(\frakt)[r]}
\]
sending $x \in \bbX$ to (the power series expansion of) $x \circ e$ and $v$ to $\exp(r)$, which can be used to define $\mathscr{L}_{\mathrm{r}}$ on the subalgebra of $\Haff$ generated by the elements $\theta_x$ ($x \in \bbX$), see Remark~\ref{rk:subalg-theta}. Then the description of $\mathscr{L}_{\mathrm{r}}$ is completed by the formula
\[
\mathscr{L}_{\mathrm{r}}(T_s+1) = (t_s+1) \cdot g(\alpha)^{-1} \cdot \widetilde{\mathscr{G}}(\alpha), \qquad \text{where} \qquad \widetilde{\mathscr{G}}(\alpha) = \mathscr{L}_{\mathrm{r}}(\mathscr{G}(\alpha)).
\]
In more concrete terms, 
we have (see~\cite[Proof of Lemma~9.5]{LuAff}):
\[
g(\alpha)^{-1} \cdot \widetilde{\mathscr{G}}(\alpha) =\frac{\exp(\dot{\alpha}+2r)-1}{\dot{\alpha}+2r} \cdot \frac{\dot{\alpha}}{\exp(\dot{\alpha})-1}.
\]

From the defining relations of $\Haff$ (resp.~$\gHaff$) one can see that there exists an anti-involution of $\Haff$ (resp.~$\gHaff$) as a $\Z[v,v^{-1}]$-algebra (resp.~$\C[r]$-algebra), which fixes all generators $T_s$ for $s \in S$ and $\theta_x$ for $x \in \bbX$ (resp.~the generators $t_s$ for $s \in S$ and the elements of $\calO(\frakt)$). Conjugating the morphism $\mathscr{L}_{\mathrm{r}}$ with these anti-involutions we obtain a second Lusztig morphism
\[
\mathscr{L}_{\mathrm{l}} \colon \Haff \to \hgHaff
\]
which satisfies
\[
\mathscr{L}_{\mathrm{l}}(\theta_x) = \mathscr{L}_{\mathrm{r}}(\theta_x), \qquad \mathscr{L}_{\mathrm{l}}(v)=\mathscr{L}_{\mathrm{r}}(v), \qquad \mathscr{L}_{\mathrm{l}}(T_s+1) = g(\alpha)^{-1} \cdot \widetilde{\mathscr{G}}(\alpha) \cdot (t_s+1).
\]

\subsection{Geometric realization of $\Haff$ and its antispherical module(s)}
\label{ss:geom-Haff}

Let $\mathcal{B}:=G/B$ be the flag variety of $G$. Then we can consider the constructions of~\S\ref{ss:convolution} for the data $X=\mathcal{B}$, $V=\frakg^*$, and with $F$ being the subbundle
\[
\wcalN:=\{(\xi,gB) \in \frakg^* \times \calB \mid \xi_{|g \cdot \frakb}=0\},
\]
where $\frakb$ is the Lie algebras of $B$.
(This variety is isomorphic to the \emph{Springer resolution} of the nilpotent cone of $G$.) We will also consider
\[
\wfrakg:=\{(\xi,gB) \in \frakg^* \times \calB \mid \xi_{|g \cdot [\frakb,\frakb]}=0\}.
\]
(This variety is isomorphic to the \emph{Grothendieck simultaneous resolution.}) Note that the Killing form defines a $G$-equivariant isomorphism $(\frakg^*)^* \cong \frakg^*$, hence a $G \times \Gm$-equivariant isomorphism $E \cong E^*$. Via this isomorphism, $F^\bot$ identifies with $\wfrakg$.

The \emph{Steinberg variety} is the fiber product
\[
Z:=\wcalN \times_{\frakg^*} \wcalN.
\]
If $\alpha$ is a simple root, we denote by $P_\alpha \subset G$ the corresponding minimal standard parabolic subgroup, and by $\mathcal{P}_\alpha := G/P_\alpha$ the associated partial flag variety. Then as in~\cite{riche}\footnote{Due to a typo, the subscript ``$\mathcal{P}_\alpha$'' is missing in the fiber product in the description of $S_\alpha'$ in~\cite[\S 6.1]{riche}.} we set
\[
S_\alpha' := \{(X,g_1 B, g_2 B) \in \frakg^* \times (\mathcal{B} \times_{\mathcal{P}_\alpha} \mathcal{B}) \mid X_{g_1 \cdot \frakb + g_2 \cdot \frakb}=0 \}.
\]
In other words, $S_\alpha'$ is the inverse image of $\mathcal{B} \times_{\mathcal{P}_\alpha} \mathcal{B}$ under the projection $Z \to \mathcal{B} \times \mathcal{B}$. This scheme is reduced but not irreducible: its two irreducible components are the diagonal $\Delta \wcalN$ and
\[
Y_\alpha :=  \{(X,g_1 B, g_2 B) \in \frakg^* \times (\mathcal{B} \times_{\mathcal{P}_\alpha} \mathcal{B}) \mid X_{g_1 \cdot \frakp_\alpha}=0\},
\]
where $\frakp_\alpha$ is the Lie algebra of $P_\alpha$.

With these definitions,
we obtain algebras $\Kth^{G \times \Gm}(Z)$ and $\coH^{G \times \Gm}_\bullet(Z)$. It follows from work of Kazhdan--Lusztig~\cite{KL}, Ginzburg~\cite{CG} and Lusztig~\cite{LUSBas} that there exists an algebra isomorphism\footnote{Due to our change of convention in the definition of the convolution product (see Footnote~\ref{fn:convention-product}), the isomorphism~\eqref{eqn:Haff-K(Z)} is the composition of the isomorphism considered in~\cite[\S 5.2]{MR3} with the anti-involution considered at the end of~\S\ref{ss:Haff-gHaff}.}
\begin{equation}
\label{eqn:Haff-K(Z)}
\Haff \simto \Kth^{G \times \Gm}(Z)
\end{equation}
which satisfies
\[
v \mapsto [\calO_{\Delta \wcalN} \langle 1 \rangle], \qquad \theta_x \mapsto [\calO_{\Delta \wcalN}(x)], \qquad T_s \mapsto -[\calO_{Y_\alpha}(-\rho, \rho-\alpha)] - [\calO_{\Delta \wcalN}] = -[\calO_{S_\alpha'}].
\]
(In the middle term, $\calO_{\Delta \wcalN}(x)$ is (the direct image of) the line bundle on 
$\Delta \wcalN$
obtained by pullback of the line bundle on $\calB$ naturally associated with $x$. In the third term, $\alpha$ is the simple root associated with $s$, 
and $\rho$ is the half-sum of the positive roots;
the equality follows from~\cite[Lemma~6.1.1]{riche}.)


We also have isomorphisms of $\Z[v,v^{-1}]$-modules
\begin{equation}
\label{eqn:Masph-Kth}
\calM^{\mathrm{asph}}_{\mathrm{l}} \simto \Kth^{G \times \Gm}(\wcalN), \qquad \text{resp.} \qquad \calM^{\mathrm{asph}}_{\mathrm{l}} \simto \Kth^{G \times \Gm}(\wcalN),
\end{equation}
where $v^n \theta_x \cdot 1$, resp.~$1 \cdot v^n \theta_x$, corresponds to $[\calO_{\wcalN}(x) \langle n \rangle]$
(for $x \in \bbX$).


\begin{lem}
\label{lem:apsh-module-geom}
The
isomorphisms~\eqref{eqn:Masph-Kth} 
are
isomorphisms of left and right $\Haff$-modules respectively.
\end{lem}

\begin{proof}
It is enough to prove that for $\alpha$ a simple root we have
\[
[\calO_{S_\alpha'}] \star [\calO_{\wcalN}] = [\calO_{\wcalN}] = [\calO_{\wcalN}] \star [\calO_{S_\alpha'}].
\]
By symmetry the two equalities are equivalent, so we restrict to the first one.
By definition we have $[\calO_{S_\alpha'}] \star [\calO_{\wcalN}] = [Rp_{1*} (\calO_{S_\alpha'})]$. If $S_\alpha \subset \wfrakg \times \wfrakg$ is the subvariety defined in~\cite[\S 1.4]{riche}, in the derived category of (equivariant) coherent sheaves on $\wfrakg \times \wfrakg$, by~\cite[Lemma~4.1]{riche} we have
\[
\calO_{\wcalN \times \wfrakg} \, \lotimes_{\calO_{\wfrakg \times \wfrakg}} \, \calO_{S_\alpha} \cong \calO_{S_\alpha'}.
\]
Then, by the (non flat) base change theorem (e.g.~in the form of~\cite[Proposition~3.7.1]{BR}), to prove our equality it is enough to prove that
\[
Rq_{1*} \calO_{S_\alpha} \cong \calO_{\wfrakg},
\]
where $q_1 \colon \wfrakg \times \wfrakg \to \wfrakg$ is the projection on the first factor. This is proved in~\cite[Lemma~2.7.2]{BR}.\footnote{The subvariety $S_\alpha$ is denoted $Z_s$ in~\cite{BR}, where $s$ is the corresponding simple reflection. Also, in~\cite[\S 2]{BR} the base field is assumed to be of positive characteristic; but the proof of the cited lemma works over any algebraically closed field of coefficients.}
\end{proof}

One also has a similar geometric realization using $\wfrakg$ instead of $\wcalN$. In fact, if we set
\[
\calZ:=\wfrakg \times_{\frakg^*} \wfrakg,
\]
as explained in~\cite[Lemma~5.2]{MR3}, restriction with supports associated with the inclusion $\wcalN \times \wfrakg \hookrightarrow \wfrakg \times \wfrakg$ induces an algebra isomorphism $\Kth^{G \times \Gm}(\calZ) \simto \Kth^{G \times \Gm}(Z)$. Therefore, we have an algebra isomorphism
\begin{equation}
\label{eqn:Haff-K(calZ)}
\Haff \simto \Kth^{G \times \Gm}(\calZ)
\end{equation}
which satisfies
\[
v \mapsto [\calO_{\Delta \wfrakg} \langle 1 \rangle], \qquad \theta_x \mapsto [\calO_{\Delta \wfrakg}(x)], \qquad T_s \mapsto -[\calO_{S_\alpha}].
\]
(Here we use conventions similar to those for $\wcalN$, and $S_\alpha$ is defined in~\cite[\S 1.4]{riche}.) As in Lemma~\ref{lem:apsh-module-geom}, we also have isomorphisms of left, resp.~right, $\Haff$-modules
\[
\calM^{\mathrm{asph}}_{\mathrm{l}} \simto \Kth^{G \times \Gm}(\wfrakg), \qquad \mathrm{resp.} \qquad \calM^{\mathrm{asph}}_{\mathrm{r}} \simto \Kth^{G \times \Gm}(\wfrakg).
\]

\subsection{Geometric realization of $\gHaff$ and its antispherical module(s)}
\label{ss:geom-gHaff}

Following~\cite{LuCus1, LuCus2}, replacing $\Kth$-theory by Borel--Moore homology in the constructions of~\S\ref{ss:geom-Haff} one obtains a geometric realization of $\gHaff$; in fact, applying~\cite[Theorem~8.11]{LuCus2} in our situation (i.e.~for the Levi subgroup $T$, its nilpotent orbit $\{0\}$, and the cuspidal local system $\underline{\C}_{\{0\}}$ on $\{0\}$), we obtain an algebra isomorphism
\begin{equation}
\label{eqn:gHaff-H(calZ)}
\gHaff \simto \coH_\bullet^{G \times \Gm}(\calZ)
\end{equation}
such that the subalgebra $\calO(\frakt)[r]$ is obtained as the image (under proper direct image) of
\begin{equation}
\label{eqn:gHaff-geom-diag}
\coH^{G \times \Gm}_\bullet(\Delta \wfrakg) \cong \coH^{G \times \Gm}_\bullet(\calB) \cong \coH^{B \times \Gm}_\bullet(\mathrm{pt}) \cong \coH^{T \times \Gm}_\bullet(\mathrm{pt}) \cong \calO(\frakt)[r].
\end{equation}
(More concretely, if $x \in \bbX$, then $\dot{x} \in \frakt^*$ corresponds to $[\Delta \wfrakg] \cdot c_1^{G \times \Gm}(\calO_{\wfrakg}(x))$, where $c_1^{G \times \Gm}(-)$ is the first equivariant Chern class.) The image of $\C[W]$ is obtained via the ``Springer isomorphism''
\[
\C[W] \simto \Hom_{D^{G \times \Gm}_{\mathrm{const}}(\frakg^*)}(p_! \underline{\C}_{\wfrakg}, p_! \underline{\C}_{\wfrakg}) \hookrightarrow \coH^{G \times \Gm}_\bullet(\calZ).
\]
(Here the inclusion is induced by the isomorphism $\can_{\wfrakg,\wfrakg}$ of~\S\ref{ss:equiv}.) As in~\eqref{eqn:gHaff-geom-diag} we also have a
natural isomorphism of $\C[r]$-modules
\begin{equation}
\label{eqn:gMasph-geom}
\overline{\calM}^{\mathrm{asph}}_{\mathrm{l}} \simto \coH^{G \times \Gm}_\bullet(\wfrakg), \qquad \text{resp.} \qquad \overline{\calM}^{\mathrm{asph}}_{\mathrm{r}} \simto \coH^{G \times \Gm}_\bullet(\wfrakg),
\end{equation}
where $\dot{x} \cdot 1$, resp.~$1 \cdot \dot{x}$, corresponds to $[\wfrakg] \cdot c_1^{G \times \Gm}(\calO_{\wfrakg}(x))$.

\begin{lem}
The
isomorphisms~\eqref{eqn:gMasph-geom} 
are
isomorphisms of left and right $\gHaff$-modules respectively.
\end{lem}

\begin{proof}
As in Lemma~\ref{lem:apsh-module-geom}, by symmetry it is enough to prove the equivariance in the first case.
Using similar constructions
as for $\coH_{\bullet}^{G \times \Gm}(\wfrakg)$, one can construct an action by convolution of $\coH^{G \times \Gm}_\bullet(\calZ)$ on $\coH^{G \times \Gm}_\bullet(\calB)$, where $\calB$ is seen as the zero section of $\wfrakg$; see~\cite[Corollary~2.7.41]{CG} in the non-equivariant setting. Moreover, the Thom isomorphism $\coH_{\bullet}^{G \times \Gm}(\wfrakg) \simto \coH^{G \times \Gm}_\bullet(\calB)$ is equivariant for this action. Therefore, it is enough to prove that the natural isomorphism $\calO(\frakt)[r] \simto \coH^{G \times \Gm}_\bullet(\calB)$ induces an isomorphism of left $\gHaff$-modules $\overline{\calM}^{\mathrm{asph}}_{\mathrm{l}} \simto \coH_{\bullet}^{G \times \Gm}(\calB)$. And for this it is enough to prove that $t_s \cdot [\calB] = -[\calB]$ for $s \in S$. Now the forgetful morphism $\coH^{G \times \Gm}_{2\dim(\calB)}(\calB) \to \coH_{2\dim(\calB)}(\calB)$ is an isomorphism, and so is the morphism $\coH^{G \times \Gm}_{2\dim(\calZ)}(\calZ) \to \coH_{2\dim(\calZ)}(\calZ)$. Hence we have reduced our question to a claim about non-equivariant Borel--Moore homology, which can be solved using Springer theory.

By~\cite[Proposition~8.6.16]{CG}, if $i_0 \colon \{0\} \hookrightarrow \wfrakg$ denotes the inclusion, there exists a canonical isomorphism $\coH_{\bullet}(\calB) \simto \coH^{2\dim(\frakg)-\bullet}(i_0^! p_! \underline{\C}_{\wfrakg})$, which identifies the action of $\coH_\bullet(\calZ)$ with the natural action of $\Hom^\bullet_{D^b_{\mathrm{const}}(\frakg^*)}(p_! \underline{\C}_{\wfrakg}, p_! \underline{\C}_{\wfrakg})$ via the non-equivariant analogue of the isomorphism $\can_{\wfrakg, \wfrakg}$. Hence what we have to show is that the $1$-dimensional $W$-module
\[
\coH_{2\dim(\calB)}(\calB) \cong \coH^{2\dim(\frakg)-2\dim(\calB)}(i_0^! p_! \underline{\C}_{\wfrakg})
\]
is the sign representation. This fact is well known, see e.g.~\cite[Lemmas 4.5 \& 4.6]{AHJR}.
\end{proof}

As in~\S\ref{ss:geom-Haff}, we have a similar story when $\wfrakg$ is replaced by $\wcalN$. In fact, constructions similar to those in~\cite[Lemma~5.2]{MR3} show that restriction with supports induces an algebra isomorphism $\coH^{G \times \Gm}_\bullet(\calZ) \simto \coH^{G \times \Gm}_\bullet(Z)$. (This property can also be extracted from~\cite{LuCus1, LuCus2}; it is used implictly in~\cite{EM}.) Therefore, we obtain isomorphisms of algebras and modules over these algebras
\begin{equation}
\label{eqn:gHaff-H(Z)}
\gHaff \simto \coH_\bullet^{G \times \Gm}(Z), \qquad \overline{\calM}^{\mathrm{asph}}_{\mathrm{l}} \simto \coH_{\bullet}^{G \times \Gm}(\wcalN), \qquad \overline{\calM}^{\mathrm{asph}}_{\mathrm{r}} \simto \coH_\bullet^{G \times \Gm}(\wcalN).
\end{equation}

\begin{prop}
\label{prop:RR-Haff}
\begin{enumerate}
\item
\label{it:RR-Haff-l}
Under the isomorphisms~\eqref{eqn:Haff-K(Z)} and~\eqref{eqn:gHaff-H(Z)}, the morphism $\uRR_{\wcalN,\wcalN}$ identifies with $\mathscr{L}_{\mathrm{l}}$.
\item
\label{it:RR-Haff-r}
Under the isomorphisms~\eqref{eqn:Haff-K(calZ)} and~\eqref{eqn:gHaff-H(calZ)}, the morphism $\oRR_{\wfrakg,\wfrakg}$ identifies with the morphism $c \mapsto e_{\calB} \cdot \mathscr{L}_{\mathrm{r}}(c) \cdot e_{\calB}^{-1}$, where
\[
e_{\calB} := \prod_{\alpha \in R^+} \frac{\dot{\alpha}}{1 - \exp(-\dot{\alpha})}.
\]
\end{enumerate}
\end{prop}

\begin{proof}
First, we note that $Z$ and $\calZ$ are paved by affine spaces, so that the ``parity vanishing'' assumptions in some of our statements above are satisfied in these cases.

\eqref{it:RR-Haff-l}
Both of our maps are algebra morphisms (see Lemma~\ref{lem:RR-convolution}), so it is enough to check that they coincide on the generators of $\Haff$. The case of $v$ is obvious (see~\cite[\S 3.3]{EG}), and the case of $\theta_x$ follows from Proposition~\ref{prop:tau-O-smooth}. It remains to consider the case of $T_s$; in fact it will be simpler (but equivalent) to prove that 
\begin{equation}
\label{eqn:RR-Haff-s}
\uRR_{\wcalN,\wcalN}(1+T_s)=\mathscr{L}_{\mathrm{l}}(1+T_s) = g(\alpha)^{-1} \cdot \widetilde{\mathscr{G}}(\alpha) \cdot (t_s+1).
\end{equation}

By Remark~\ref{rmk:injectivity-faithful} and Proposition~\ref{prop:criterion-thm}, $\coH_\bullet^{G \times \Gm}(\wcalN)$ is faithful as a module over $\coH_\bullet^{G \times \Gm}(Z)$. Therefore, the same is true for the completions, and to prove~\eqref{eqn:RR-Haff-s} it is enough to prove that both sides act similarly on $\coHc_\bullet^{G \times \Gm}(\wcalN)$. However, by Lemma~\ref{lem:apsh-module-geom} and~\eqref{eqn:comm-relation-Haff}, for $x \in \bbX$ we have $(1+T_s) \cdot (\theta_x \cdot 1) = \bigl( (\theta_x - \theta_{sx}) \cdot \mathscr{G}(\alpha) \bigr) \cdot 1$. By Lemma~\ref{lem:RR-action}, this implies that in $\overline{\calM}^{\mathrm{asph}}_{\mathrm{l}}$ we have
\[
\uRR_{\wcalN,\wcalN}(1+T_s) \cdot (\exp(\dot{x}) \cdot 1) = \bigl( (\exp(\dot{x})-\exp(s\dot{x})) \cdot \widetilde{\mathscr{G}}(\alpha) \bigr) \cdot 1.
\]
Using~\eqref{eqn:comm-relation-gHaff}, this coincides with the action of $g(\alpha)^{-1} \cdot \widetilde{\mathscr{G}}(\alpha) \cdot (t_s+1)$. Since the elements of the form $r^n \exp(\dot{x}) \cdot 1$ form a topological basis of $\coHc_\bullet^{G \times \Gm}(\wcalN)$, we deduce the equality in~\eqref{eqn:RR-Haff-s}.

\eqref{it:RR-Haff-r}
The proof is similar to the proof of~\eqref{it:RR-Haff-l}, using the \emph{right} action on $\coHc_\bullet^{G \times \Gm}(\wfrakg)$, and using the fact that
\[
\mathrm{Td}_{\calB}^{G \times \Gm} = \prod_{\alpha \in R^+} \frac{\dot{\alpha}}{1 - \exp(-\dot{\alpha})} \qquad \text{in} \qquad \coHc_{G \times \Gm}^\bullet(\wfrakg) = \widehat{\calO(\frakt)[r]}
\]
(as follows from~\cite[\S 3.3]{EG}, since the tangent bundle on $\calB$ has a filtration with associated graded the sum of the line bundles $\calO_{\calB}(\alpha)$ for $\alpha \in R^+$).
\end{proof}

\begin{remark}
In~\cite[\S 0.3]{LuAff}, Lusztig explains that his morphism $\mathscr{L}_{\mathrm{r}}$ ``is of the same nature as the Chern character from $\Kth$-theory to homology.'' Proposition~\ref{prop:RR-Haff} is a concrete justification of this claim.
\end{remark}

\subsection{Commutative diagram for affine Hecke algebras}

Finally we can consider the diagram of Theorem~\ref{thm:LKDFourier} in the geometric setting of \S\S\ref{ss:geom-Haff}--\ref{ss:geom-gHaff}:
\begin{equation}
\label{eqn:diagram-Haff}
\vcenter{
\xymatrix@C=3cm{
\Kth^{G \times \Gm}(Z) \ar[r]^-{\Koszul_{\wcalN, \wcalN}} \ar[d]_-{\underline{\mathrm{RR}}_{\wcalN, \wcalN}} & \Kth^{G \times \Gm}(\calZ) \ar[r]^-{\bfi_{\wfrakg, \wfrakg} \circ \dual_{\wfrakg, \wfrakg}} & \Kth^{G \times \Gm}(\calZ) \ar[d]^-{\overline{\mathrm{RR}}_{\wfrakg, \wfrakg}} \\
\coHc^{G \times \Gm}_\bullet(Z) \ar[rr]^-{\Fourier_{\wcalN,\wcalN}} & & \coHc^{G \times \Gm}_\bullet(\calZ).
}
}
\end{equation}
Note that Proposition \ref{prop:criterion-thm} ensures that the assumption of Theorem~\ref{thm:LKDFourier} is satisfied in this case, since $\calZ$ is paved by affine spaces, and that the results of~\S\ref{ss:convolution}--\ref{ss:RR-convolution} ensure that all the maps in this diagram are unital algebra morphisms. Using Proposition~\ref{prop:RR-Haff} and the results of~\cite{EM} and~\cite{MR3} we can describe explicitly all the maps in this diagram, and hence illustrate the content of Theorem~\ref{thm:LKDFourier} in this particular situation.

The morphism $\Koszul_{\wcalN,\wcalN}$ was studied in~\cite[\S 5.3]{MR3}. In particular,~\cite[Theorem~5.4]{MR3} describes this automorphism algebraically, and shows that it is closely related to the \emph{Iwahori--Matsumoto} involution of $\calH_\aff$.
Using the identifications~\eqref{eqn:Haff-K(Z)} and~\eqref{eqn:Haff-K(calZ)}, we have
\[
\Koszul_{\wcalN,\wcalN}(T_s)= \theta_\rho(-v^2 T_s^{-1}) \theta_{-\rho}, \qquad \Koszul_{\wcalN,\wcalN}(\theta_x)=\theta_{-x}, \qquad \Koszul_{\wcalN,\wcalN}(v)=-v
\]
for $s \in S$ a simple root and $x \in \bbX$.\footnote{As noted in Footnote~\ref{fn:MR3}, the conventions in the definition of $\frakK_{\wcalN,\wcalN}$ used in the present paper differ slightly from the conventions used in~\cite{MR3}. Our identification of $\Kth^{G \times \Gm}(\calZ)$ is also slightly different, see~\cite[Comments at the end of~\S5.2]{MR3}. This explains the differences with the formulas in~\cite{MR3}.}

Concerning the map $\dual_{\wfrakg,\wfrakg}$, one can check that, with the 
identification~\eqref{eqn:Haff-K(calZ)}, it satisfies
\[
\dual_{\wfrakg,\wfrakg}(T_s)=T_s^{-1}, \qquad \dual_{\wfrakg,\wfrakg}(\theta_x)=\theta_{-x}, \qquad \dual_{\wfrakg,\wfrakg}(v)=v^{-1}.
\]
(See~\cite[Lemma~9.7]{LUSBas} for a similar computation, with different conventions.) 
Finally, the morphism 
$\bfi_{\wfrakg,\wfrakg}$ 
is the same as the involution $\iota$ of~\cite[\S 5.3]{MR3}; it satisfies
\[
\bfi_{\wfrakg,\wfrakg}(T_s)=T_s, \qquad \bfi_{\wfrakg,\wfrakg}(\theta_x)=\theta_{x}, \qquad \bfi_{\wfrakg,\wfrakg}(v)=-v.
\]


On the Borel--Moore homology side,
the map $\Fourier_{\wcalN,\wcalN}$ was studied in~\cite{EM}. In 
that paper it was shown to be closely related to the Iwahori--Matsumoto involution of $\gHaff$; more precisely it satisfies
\[
\Fourier_{\wcalN,\wcalN} (t_w) = (-1)^{\ell(w)} t_w, \qquad \Fourier_{\wcalN,\wcalN}(\phi)=\phi, \qquad \Fourier_{\wcalN,\wcalN}(r)=-r
\]
for $w \in W$ and $\phi \in \calO(\frakt)$.

Using these formulas one can check the commutativity of~\eqref{eqn:diagram-Haff} by hand. For instance, for the element $1+T_s$, the commutativity of the diagram amounts to the following equality in $\gHaff$:
\begin{multline*}
\frac{\exp(\dot{\alpha}-2r)-1}{\dot{\alpha}-2r}\frac{\dot{\alpha}}{\exp(\dot{\alpha})-1}(-t_s+1) = \\
1 - \exp(-\dot{\rho}-2r) e_\calB \Bigl( (t_s+1) \frac{\exp(\dot{\alpha}+2r)-1}{\dot{\alpha}+2r}\frac{\dot{\alpha}}{\exp(\dot{\alpha})-1} - 1 \Bigr) e_\calB^{-1} \exp(\dot{\rho}).
\end{multline*}

\section{Compatibility of the $\Fourier$ isomorphism with inclusions}
\label{sec:compatibility-Fourier}

In this section and the next one we will consider compatibility properties of our morphisms in two geometric situations. We use the same setting and notation as in \S\S\ref{ss:equiv}--\ref{ss:statement}.

\subsection{Further notation}
\label{ss:settings}

First we will consider a situation which we will refer to as Setting (A): here we are given an additional subbundle $F_2' \subset E$ containing $F_2$ and such that $F_2$, $F_2'$ and $E$ can be locally simultaneously trivialized. Then we have ``restriction with supports'' morphisms associated with the embedding $F_2 \hookrightarrow F_2'$, both in $\Kth$-homology and in Borel--Moore homology, which we denote as follows:
\begin{align*}
\resK^{F_1, F_2'}_{F_1,F_2} \colon \Kth^{G \times \Gm} (F_1 \times_V F_2') \ & \to \ \Kth^{G \times \Gm}(F_1 \times_V F_2); \\
\resH^{F_1, F_2'}_{F_1,F_2} \colon \coH^{G \times \Gm}_\bullet (F_1 \times_V F_2') \ & \to \ \coH_{\bullet-2\mathrm{rk}(F_2')+2\mathrm{rk}(F_2)}^{G \times \Gm}(F_1 \times_V F_2).
\end{align*}
We also have proper direct image morphisms associated with the embedding $(F_2')^\bot \hookrightarrow F_2^\bot$, again both in $\Kth$-homology and in Borel--Moore homology, which we denote as follows:
\begin{align*}
\pdiK^{F_1^\bot, (F_2')^\bot}_{F_1^\bot, F_2^\bot} \colon \Kth^{G \times \Gm}(F_1^\bot \times_{V^*} (F_2')^\bot) \ &\to \ \Kth^{G \times \Gm}(F_1^\bot \times_{V^*} F_2^\bot); \\
\pdiH^{F_1^\bot, (F_2')^\bot}_{F_1^\bot, F_2^\bot} \colon \coH^{G \times \Gm}_\bullet(F_1^\bot \times_{V^*} (F_2')^\bot) \ &\to \ \coH_\bullet^{G \times \Gm}(F_1^\bot \times_{V^*} F_2^\bot).
\end{align*}


Secondly,
we will consider a situation which we will refer to as Setting (B): here we are given an additional subbundle $F_1' \subset E$ containing $F_1$ and such that $F_1$, $F_1'$ and $E$ can be locally simultaneously trivialized. Then we have proper direct image morphisms associated with the embedding $F_1 \hookrightarrow F_1'$, both in $\Kth$-homology and in Borel--Moore homology, which we denote as follows:
\begin{align*}
\pdiK^{F_1,F_2}_{F_1',F_2} \colon \Kth^{G \times \Gm}(F_1 \times_V F_2) \ & \to \ \Kth^{G \times \Gm}(F_1' \times_V F_2); \\
\pdiH^{F_1,F_2}_{F_1',F_2} \colon \coH_\bullet^{G \times \Gm}(F_1 \times_V F_2) \ & \to \ \coH_\bullet^{G \times \Gm}(F_1' \times_V F_2).
\end{align*}
We also have ``restriction with supports'' morphisms associated with the embedding $(F_1')^\bot \hookrightarrow F_1^\bot$, again both in $\Kth$-homology and in Borel--Moore homology, which we denote as follows:
\begin{align*}
\resK^{F_1^\bot, F_2^\bot}_{(F_1')^\bot, F_2^\bot} \colon \Kth^{G \times \Gm}(F_1^\bot \times_{V^*} F_2^\bot) \ & \to \ \Kth^{G \times \Gm}((F_1')^\bot \times_{V^*} F_2^\bot); \\
\resH^{F_1^\bot, F_2^\bot}_{(F_1')^\bot, F_2^\bot} \colon \coH_\bullet^{G \times \Gm}(F_1^\bot \times_{V^*} F_2^\bot) \ &\to \ \coH_{\bullet-2\mathrm{rk}(F_1^\bot)+2\mathrm{rk}((F_1')^\bot)}^{G \times \Gm}((F_1')^\bot \times_{V^*} F_2^\bot).
\end{align*}


\subsection{Convolution algebras and inclusion of subbundles}
\label{ss:subbundle}

Consider Setting (A) of \S \ref{ss:settings}. Then we have natural morphisms induced by adjunction
\[
\mathrm{adj}_{F_2,F_2'}^* \colon \uC_{F_2'} \to \uC_{F_2} \quad \text{ and } \quad \mathrm{adj}_{(F_2')^\bot,F_2^\bot}^! \colon \uC_{(F_2')^\bot} \to \uC_{F_2^\bot}[2\mathrm{rk}(F_2^\bot)-2\mathrm{rk}((F_2')^\bot)].
\]
The proof of the following result being rather technical (and the details not needed), it is postponed to the appendix (see \S\S\ref{ss:proof-inclusion-Ginzburg-1}--\ref{ss:proof-inclusion-Ginzburg-2}).

\begin{prop}
\label{prop:inclusion-Ginzburg-A}
\begin{enumerate}
\item
The following diagram commutes:
\[
\xymatrix@C=2.5cm{
\coH^{G \times \Gm}_\bullet(F_1 \times_V F_2') \ar[r]^-{\can_{F_1, F_2'}}_-{\sim} \ar[d]_-{\resH^{F_1,F_2'}_{F_1,F_2}} & \Ext^{2\dim(F_2')-\bullet}_{\calD^{G \times \Gm}_{\mathrm{const}}(V)}(p_! \uC_{F_1}, p_! \uC_{F_2'}) \ar[d]^-{(p_! \mathrm{adj}_{F_2,F_2'}^*) \circ (\cdot)} \\
\coH^{G \times \Gm}_{\bullet-2\mathrm{rk}(F_2')+2\mathrm{rk}(F_2)}(F_1 \times_V F_2) \ar[r]^-{\can_{F_1,F_2}}_-{\sim} & \Ext^{2\dim(F_2')-\bullet}_{\calD^{G \times \Gm}_{\mathrm{const}}(V)}(p_! \uC_{F_1}, p_! \uC_{F_2}).
}
\]
\item
The following diagram commutes:
\[
\xymatrix@C=2.5cm{
\coH^{G \times \Gm}_\bullet(F_1^\bot \times_{V^*} (F_2')^\bot) \ar[r]^-{\can_{F_1^\bot, (F_2')^\bot}}_-{\sim} \ar[d]_-{\pdiH^{F_1^\bot, (F_2')^\bot}_{F_1^\bot, F_2^\bot}} & \Ext^{2\dim((F_2')^\bot)-\bullet}_{\calD^{G \times \Gm}_{\mathrm{const}}(V^*)}({\check p}_! \uC_{F_1^\bot}, {\check p}_! \uC_{(F_2')^\bot}) \ar[d]^-{({\check p}_! \mathrm{adj}_{(F_2')^\bot,F_2^\bot}^!) \circ (\cdot)} \\
\coH^{G \times \Gm}_\bullet(F_1^\bot \times_{V^*} F_2^\bot) \ar[r]^-{\can_{F_1^\bot, F_2^\bot}}_-{\sim} & \Ext^{2\dim(F_2^\bot)-\bullet}_{\calD^{G \times \Gm}_{\mathrm{const}}(V^*)}({\check p}_! \uC_{F_1^\bot}, {\check p}_! \uC_{F_2^\bot}).
}
\]
\end{enumerate}

\end{prop}

Consider now Setting (B) of \S\ref{ss:settings}. We have natural morphisms induced by adjunction
\[
\mathrm{adj}_{F_1,F_1'}^* \colon \uC_{F_1'} \to \uC_{F_1} \quad \text{ and } \quad \mathrm{adj}_{(F_1')^\bot,F_1^\bot}^! \colon \uC_{(F_1')^\bot} \to \uC_{F_1^\bot}[2\mathrm{rk}(F_1^\bot)-2\mathrm{rk}((F_1')^\bot)].
\]
The proof of the following proposition is similar to that of Proposition \ref{prop:inclusion-Ginzburg-A}, and is therefore omitted.

\begin{prop}
\label{prop:inclusion-Ginzburg-B}
\begin{enumerate}
\item
The following diagram commutes:
\[
\xymatrix@C=2.5cm{
\coH^{G \times \Gm}_\bullet(F_1 \times_V F_2) \ar[r]^-{\can_{F_1, F_2}}_-{\sim} \ar[d]_-{\pdiH^{F_1, F_2}_{F_1', F_2}} & \Ext^{2\dim(F_2)-\bullet}_{\calD^{G \times \Gm}_{\mathrm{const}}(V)}(p_! \uC_{F_1}, p_! \uC_{F_2}) \ar[d]^-{(\cdot) \circ (p_! \mathrm{adj}_{F_1,F_1'}^*)} \\
\coH^{G \times \Gm}_\bullet(F_1' \times_V F_2) \ar[r]^-{\can_{F_1',F_2}}_-{\sim} & \Ext^{2\dim(F_2)-\bullet}_{\calD^{G \times \Gm}_{\mathrm{const}}(V)}(p_! \uC_{F_1'}, p_! \uC_{F_2}).
}
\]
\item
The following diagram commutes:
{\small
\[
\xymatrix@C=2cm{
\coH^{G \times \Gm}_\bullet(F_1^\bot \times_{V^*} F_2^\bot) \ar[r]^-{\can_{F_1^\bot, F_2^\bot}}_-{\sim} \ar[d]_-{\resH^{F_1^\bot, F_2^\bot}_{(F_1')^\bot, F_2^\bot}} & \Ext^{2\dim(F_2^\bot)-\bullet}_{\calD^{G \times \Gm}_{\mathrm{const}}(V^*)}({\check p}_! \uC_{F_1^\bot}, {\check p}_! \uC_{F_2^\bot}) \ar[d]^-{(\cdot) \circ ({\check p}_! \mathrm{adj}_{(F_1')^\bot,F_1^\bot}^!)} \\
\coH^{G \times \Gm}_{\bullet-2\mathrm{rk}(F_1^\bot)+2\mathrm{rk}((F_1')^\bot)}((F_1')^\bot \times_{V^*} F_2^\bot) \ar[r]^-{\can_{(F_1')^\bot, F_2^\bot}}_-{\sim} & \Ext^{2\dim(F_2^\bot)+2\mathrm{rk}(F_1^\bot)-2\mathrm{rk}((F_1')^\bot)-\bullet}_{\calD^{G \times \Gm}_{\mathrm{const}}(V^*)}({\check p}_! \uC_{(F_1')^\bot}, {\check p}_! \uC_{F_2^\bot}).
}
\]
}
\end{enumerate}

\end{prop}

\subsection{Fourier transform and inclusion of subbundles}

In the next lemma $G$ can be replaced by any linear algebraic group, $X$ by any smooth $G$-variety, and $E$ by any $G$-equivariant vector bundle over $X$.
We consider subbundles $F \subset F' \subset E$ which can be locally simultaneously trivialized. (In practice, $E$ and $X$ will be as above, and we will take $F=F_i$, $F'=F'_i$ for $i \in \{1,2\}$.)
Adjunction induces morphisms
\[
\mathrm{adj}^*_{F,F'} \colon \uC_{F'} \to \uC_F \quad \text{and} \quad \mathrm{adj}^!_{(F')^\bot,F^\bot} \colon \uC_{(F')^\bot} \to \uC_{F^\bot}[2 \mathrm{rk}(F^\bot)-2\mathrm{rk}((F')^\bot)].
\]

\begin{lem}
\label{lem:fourier-adj}
The following diagram is commutative:
\[
\xymatrix@C=2.5cm{
\calF_E(\uC_{F'}) \ar[d]^-{\wr} \ar[r]^-{\calF_E ( \mathrm{adj}^*_{F,F'} )} & \calF_E(\uC_F) \ar[d]_-{\wr} \\
\uC_{(F')^\bot}[-2\mathrm{rk}(F')] \ar[r]^-{\mathrm{adj}^!_{(F')^\bot,F^\bot}} & \uC_{F^\bot}[-2\mathrm{rk}(F)],
}
\]
where vertical isomorphisms are provided by Lemma~{\rm \ref{lem:fourier-F}}.
\end{lem}

\begin{proof}
It is equivalent to prove a similar isomorphism for $\mathfrak{F}_E$; for simplicity we still denote by $F^\bot, (F')^\bot$ the orthogonals viewed in $E^\diamond$, and by ${\check r} \colon E^\diamond \to X$ the projection. By the construction in the proof of Lemma~\ref{lem:fourier-F} we have natural isomorphisms 
\[
\mathfrak{F}_E(\uC_{F'}) \cong {\check q}_! \uC_{Q_{F'}} \qquad \text{and} \qquad \mathfrak{F}_E(\uC_F) \cong {\check q}_! \uC_{Q_F},
\]
where $Q_{F'}:=q^{-1}(F')$, $Q_F:=q^{-1}(F)$.
It follows from the definitions that the morphism $\calF_E ( \mathrm{adj}^*_{F,F'} )$ is the image under ${\check q}_!$ of the morphism $\uC_{Q_{F'}} \to \uC_{Q_F}$ induced by adjunction (for the inclusion $Q_F \hookrightarrow Q_{F'}$). Hence what we have to show is that the morphism $\varphi$ in the following diagram coincides with $\mathrm{adj}^!_{(F')^\bot,F^\bot}$,
where the upper arrow is induced by adjunction as above, and the vertical isomorphisms are as in the proof of Lemma \ref{lem:fourier-F}:
\[
\xymatrix@C=1.5cm{
{\check q}_! \uC_{Q_{F'}} \ar[r] \ar[d]_-{\wr} & {\check q}_! \uC_{Q_F} \ar[d]^-{\wr} \\
{\check q}_! \uC_{F' \times_X (F')^\bot} \ar[d]_-{\wr} & {\check q}_! \uC_{F \times_X F^\bot} \ar[d]^-{\wr} \\
\uC_{(F')^\bot}[-2\rk(F')] \ar[r]^-{\varphi} & \uC_{F^\bot}[-2\rk(F)].
}
\]
Now we have canonical isomorphisms
\[
{\check r}_! \bigl( \uC_{(F')^\bot}[-2\rk(F')] \bigr) \cong \uC_X[-2\rk(E)], \quad {\check r}_! \bigl( \uC_{F^\bot}[-2\rk(F)] \bigr) \cong \uC_X[-2\rk(E)],
\]
and
one can check that the functor ${\check r}_!$ induces an isomorphism
\begin{multline*}
\Hom_{\calD^{G\times \Gm}_{\mathrm{const}} (E^\diamond)} \bigl( \uC_{(F')^\bot}[-2\rk(F')] , \uC_{F^\bot}[-2\rk(F)] \bigr) \simto \\
\Hom_{\calD^{G\times \Gm}_{\mathrm{const}}(X)} \bigl( \uC_X[-2\rk(E)], \uC_X[-2\rk(E)] \bigr)
\end{multline*}
sending $\mathrm{adj}^!_{(F')^\bot,F^\bot}[-2\rk(F')]$ to the identity morphism of $\uC_X[-2\rk(E)]$. Hence it is enough to prove that ${\check r}_! \varphi[2\rk(E)]$ is the identity of $\uC_X$ (through the canonical isomorphisms above).
The latter statement is about sheaves (and not complexes), so that we can forget about equivariance and check the claim locally over $X$. (This is allowed by combining \cite[Proposition 2.5.3]{BL} and \cite[Proposition 4.2.7]{letellier}.) By local triviality, one can then assume that $X=\mathrm{pt}$ (i.e.~that $E$ is a vector space and that $F,F' \subset E$ are subspaces).

In this case the claim boils down to the fact that the dotted arrow in the following diagram is the identity:
\[
\xymatrix@C=1.5cm{
 \coH_c^{2\dim(E)}(F' \times (F')^\bot) \ar[d]_-{\wr} & \coH_c^{2 \dim(E)}(Q) \ar[r]^-{\sim} \ar[l]_-{\sim} & \coH_c^{2\dim(E)}(F \times F^\bot) \ar[d]^-{\wr} \\
\C \ar@{.>}[rr] & & \C.
}
\]
To prove this fact we regard $E \times E^*$ as a real vector space, endowed with the non-degenerate quadratic form given by $q(x,\xi):=\mathrm{Re}(\langle \xi, x \rangle)$. The orthogonal group $H$ of this form stabilizes $Q$, hence acts on $\coH_c^{2 \dim(E)}(Q)$, and this action factors through the group of components $H/H^\circ$. Now $F \times F^\bot$ and $F' \times (F')^\bot$ are conjugate under the action of $H^\circ$, with finishes the proof.
\end{proof}

In the following proposition we get back to the assumption that $E=V \times X$, and we let $p \colon E \to V$ be the projection.
The following result is an immediate consequence of Lemma~\ref{lem:fourier-adj} and the isomorphism of functors $\calF_V \circ p_! \cong {\check p}_! \circ \calF_E$, see the proof of Corollary \ref{prop:Fourier-F}.

\begin{prop}
\label{prop:inclusion-Fourier}
The following diagram is commutative:
\[
\xymatrix@C=3cm{
\calF_V(p_! \uC_{F'}) \ar[d]_-{\eqref{eqn:isom-Fourier-F}}^-{\wr} \ar[r]^-{\calF_V ( p_!(\mathrm{adj}^*_{F,F'}) )} & \calF_V(p_! \uC_F) \ar[d]^-{\eqref{eqn:isom-Fourier-F}}_-{\wr} \\
{\check p}_! \uC_{(F')^\bot}[-2\mathrm{rk}(F')] \ar[r]^-{{\check p}_!(\mathrm{adj}^!_{(F')^\bot,F^\bot})} & {\check p}_! \uC_{F^\bot}[-2\mathrm{rk}(F)].
}
\]
\end{prop}

\subsection{The $\Fourier$ isomorphism and inclusion of subbundles}

We come back to Setting (A) of \S\ref{ss:settings}.

\begin{prop}
\label{prop:compatibility-fourier-A}
We have an equality
\[
\Fourier_{F_1,F_2} \circ \resH^{F_1, F_2'}_{F_1,F_2} = \pdiH^{F_1^\bot, (F_2')^\bot}_{F_1^\bot, F_2^\bot} \circ \Fourier_{F_1,F_2'}
\]
of morphisms $\coH^{G \times \Gm}_\bullet(F_1 \times_V F_2') \to \coH^{G \times \Gm}_{\bullet+2\dim((F_2')^\bot)-2\dim(F_1)} ( F_1^\bot \times_{V^*} F_2^\bot)$.
\end{prop}

\begin{proof}
By functoriality the following diagram commutes, where horizontal maps are induced by the functor $\calF_V$:
\[
\xymatrix@C=2cm{
\Ext^{2\dim(F_2')-\bullet}_{\calD^{G \times \Gm}_{\mathrm{const}}(V)}(p_! \uC_{F_1}, p_! \uC_{F_2'}) \ar[d]_-{(p_! \mathrm{adj}_{F_2,F_2'}^*) \circ (\cdot)} \ar[r]_-{\sim} & \Ext^{2\dim(F_2')-\bullet}_{\calD^{G \times \Gm}_{\mathrm{const}}(V^*)}(\calF_V(p_! \uC_{F_1}), \calF_V(p_! \uC_{F_2'})) \ar[d]^-{\calF_V(p_! \mathrm{adj}_{F_2,F_2'}^*) \circ (\cdot)} \\
\Ext^{2\dim(F_2')-\bullet}_{\calD^{G \times \Gm}_{\mathrm{const}}(V)}(p_! \uC_{F_1}, p_! \uC_{F_2}) \ar[r]_-{\sim} & \Ext^{2\dim(F_2')-\bullet}_{\calD^{G \times \Gm}_{\mathrm{const}}(V^*)}(\calF_V(p_! \uC_{F_1}), \calF_V(p_! \uC_{F_2})).
}
\]
Now by Proposition \ref{prop:inclusion-Fourier} the following diagram commutes, where vertical maps are induced by the isomorphisms $\calF_V(p_! \uC_{F}) \cong {\check p}_! \uC_{F^\bot}[-2\rk(F)]$ for $F=F_1,F_2$ or $F_2'$ (see \eqref{eqn:isom-Fourier-F}):
\[
\xymatrix@C=1cm{
\Ext^{2\dim(F_2')-\bullet}_{\calD^{G \times \Gm}_{\mathrm{const}}(V^*)}(\calF_V(p_! \uC_{F_1}), \calF_V(p_! \uC_{F_2'})) \ar[d]_-{\calF_V(p_! \mathrm{adj}_{F_2,F_2'}^*) \circ (\cdot)} \ar[r]_-{\sim} & \Ext^{2\dim(F_1)-\bullet}_{\calD^{G \times \Gm}_{\mathrm{const}}(V^*)}({\check p}_! \uC_{F_1^\bot}, {\check p}_! \uC_{(F_2')^\bot}) \ar[d]^-{({\check p}_! \mathrm{adj}_{(F_2')^\bot,F_2^\bot}^!) \circ (\cdot)} \\
\Ext^{2\dim(F_2')-\bullet}_{\calD^{G \times \Gm}_{\mathrm{const}}(V^*)}(\calF_V(p_! \uC_{F_1}), \calF_V(p_! \uC_{F_2})) \ar[r]_-{\sim} & \Ext^{2\dim(F_1)+2\dim(F_2')-2\dim(F_2)-\bullet}_{\calD^{G \times \Gm}_{\mathrm{const}}(V^*)}({\check p}_! \uC_{F_1^\bot}, {\check p}_! \uC_{F_2^\bot}).
}
\]
Pasting these diagrams with the ones of Proposition~\ref{prop:inclusion-Ginzburg-A} we obtain the desired equality.
\end{proof}

Now we consider Setting (B) of \S\ref{ss:settings}. The proof of the following proposition is similar to that of Proposition \ref{prop:compatibility-fourier-A} (replacing Proposition \ref{prop:inclusion-Ginzburg-A} by Proposition \ref{prop:inclusion-Ginzburg-B}), and is therefore omitted.

\begin{prop}
\label{prop:compatibility-fourier-B}
We have an equality
\[
\Fourier_{F_1',F_2} \circ \pdiH^{F_1,F_2}_{F_1',F_2} = \resH^{F_1^\bot, F_2^\bot}_{(F_1')^\bot, F_2^\bot} \circ \Fourier_{F_1,F_2}
\]
of morphisms $\coH^{G \times \Gm}_\bullet(F_1 \times_V F_2) \to \coH^{G \times \Gm}_{\bullet+2\dim(F_2^\bot)-2\dim(F_1')} \bigl( (F_1')^\bot \times_{V^*} F_2^\bot \bigr)$.
\end{prop}

\section{Compatibility of the remaining constructions with inclusions}
\label{sec:compatibility-others}

\subsection{Compatibilities for linear Koszul duality}

Consider Setting (A) of \S\ref{ss:settings}. Then we have equivalences of triangulated categories $\frakK_{F_1,F_2}$ and $\frakK_{F_1,F_2'}$ constructed as in \S\ref{ss:lkd}.
We also have natural morphisms of dg-schemes
\begin{multline*}
f \colon (\Delta V \times X \times X) \rcap_{E \times E} (F_1 \times F_2) \to (\Delta V \times X \times X) \rcap_{E \times E} (F_1 \times F_2'), \\
g \colon (\Delta V^* \times X \times X) \rcap_{E^* \times E^*} (F_1^\bot \times (F_2')^\bot) \to (\Delta V^* \times X \times X) \rcap_{E^* \times E^*} (F_1^\bot \times F_2^\bot)
\end{multline*}
associated with the inclusions $F_2 \hookrightarrow F_2'$ and $(F_2')^\bot \hookrightarrow F_2^\bot$ respectively,
and associated functors
\[
Lf^* \colon \calD^c_{G \times \Gm} \bigl( (\Delta V \times X \times X) \rcap_{E \times E} (F_1 \times F_2') \bigr) \to \calD^c_{G \times \Gm} \bigl( (\Delta V \times X \times X) \rcap_{E \times E} (F_1 \times F_2) \bigr),
\]
\begin{multline*}
Rg_* \colon \calD^c_{G \times \Gm} \bigl( (\Delta V^* \times X \times X) \rcap_{E^* \times E^*} (F_1^\bot \times (F_2')^\bot) \bigr) \to \\
\calD^c_{G \times \Gm} \bigl( (\Delta V^* \times X \times X) \rcap_{E^* \times E^*} (F_1^\bot \times F_2^\bot) \bigr)
\end{multline*}
(see \cite[\S\S 3.2--3.3]{MR3} for details).
By \cite[Proposition 3.5]{MR3} there exists an isomorphism of functors
\[
\frakK_{F_1,F_2} \circ Lf^* \cong Rg_* \circ \frakK_{F_1,F_2'}.
\]

It easily follows from definitions that the following diagram commutes:
\[
\xymatrix@C=2cm{
\calD^c_{G \times \Gm} \bigl( (\Delta V \times X \times X) \rcap_{E \times E} (F_1 \times F_2') \bigr) \ar[r]^{Lf^*} \ar[d] & \calD^c_{G \times \Gm} \bigl( (\Delta V \times X \times X) \rcap_{E \times E} (F_1 \times F_2) \bigr) \ar[d] \\
\calD^b \Coh^{G \times \Gm}(F_1 \times F_2') \ar[r] & \calD^b \Coh^{G \times \Gm}(F_1 \times F_2).
}
\]
(Here the lower horizontal arrow is the usual pullback functor associated with the embedding $F_1 \times F_2 \hookrightarrow F_1 \times F_2'$. The right vertical arrow is induced by the ``restriction of scalars'' functor associated with the embedding $\calA_{F_1, F_2}^0 \hookrightarrow \calA_{F_1, F_2}$, where the dg-algebra $\calA_{F_1, F_2}$ is defined in \S\ref{ss:lkd}; note that $\calA_{F_1, F_2}^0$ is the direct image of the structure sheaf under the affine morphism $F_1 \times F_2 \to X \times X$. The left vertical arrow is defined similarly.)
We deduce that the morphism induced by $Lf^*$ in $\Kth$-homology is $\resK^{F_1, F_2'}_{F_1,F_2}$.
Similarly, the morphism induced by $Rg_*$ in $\Kth$-homology is $\pdiK^{F_1^\bot, (F_2')^\bot}_{F_1^\bot, F_2^\bot}$ (see the proof of~\cite[Lemma~3.3]{MR3}).
We deduce the following result.

\begin{prop}
\label{prop:compatibility-LKD-A}
We have an equality
\[
\Koszul_{F_1,F_2} \circ \resK^{F_1, F_2'}_{F_1,F_2} = \pdiK^{F_1^\bot, (F_2')^\bot}_{F_1^\bot, F_2^\bot} \circ \Koszul_{F_1,F_2'}
\]
of morphisms $\Kth^{G \times \Gm}(F_1 \times_V F_2') \to \Kth^{G \times \Gm}( F_1^\bot \times_{V^*} F_2^\bot)$.
\end{prop}

Now, consider Setting (B) of \S\ref{ss:settings}. The same considerations as above allow to prove the following result.

\begin{prop}
\label{prop:compatibility-LKD-B}
We have an equality
\[
\Koszul_{F_1',F_2} \circ \pdiK^{F_1, F_2}_{F_1', F_2} = \resK^{F_1^\bot, F_2^\bot}_{(F_1')^\bot, F_2^\bot} \circ \Koszul_{F_1,F_2}
\]
of morphisms $\Kth^{G \times \Gm}(F_1 \times_V F_2) \to \Kth^{G \times \Gm}\bigl( (F_1')^\bot \times_{V^*} F_2^\bot \bigr)$.
\end{prop}

\subsection{Compatibilities for the other maps in $\Kth$-homology}

Consider Setting (A) of \S\ref{ss:settings}.

\begin{prop}
\label{prop:compatibility-duality-A}
We have equalities
\begin{gather*}
\dual_{F_1^\bot, F_2^\bot} \circ \pdiK^{F_1^\bot, (F_2')^\bot}_{F_1^\bot, F_2^\bot} = \pdiK^{F_1^\bot, (F_2')^\bot}_{F_1^\bot, F_2^\bot} \circ \dual_{F_1^\bot,(F_2')^\bot}, \\
\bfi_{F_1^\bot, F_2^\bot} \circ \pdiK^{F_1^\bot, (F_2')^\bot}_{F_1^\bot, F_2^\bot} = \pdiK^{F_1^\bot, (F_2')^\bot}_{F_1^\bot, F_2^\bot} \circ \bfi_{F_1^\bot,(F_2')^\bot}
\end{gather*}
of morphisms $\Kth^{G \times \Gm}(F_1^\bot \times_{V^*} (F_2')^\bot) \to \Kth^{G \times \Gm}(F_1^\bot \times_{V^*} F_2^\bot)$.
%
\end{prop}

\begin{proof}
The second equality is easy, and left to the reader. Let us consider the first one.
We denote the inclusion morphism by
\[
h_A \colon F_1^\bot \times (F_2')^\bot \hookrightarrow F_1^\bot \times F_2^\bot, 
\]
and consider the duality functor
\[
\mathrm{D}^{G \times \Gm}_{F_1^\bot,F_2^\bot} \colon \calD^b \Coh^{G \times \Gm}_{F_1^\bot \times_{V^*} F_2^\bot}(F_1^\bot \times F_2^\bot) \to \calD^b \Coh^{G \times \Gm}_{F_1^\bot \times_{V^*} F_2^\bot}(F_1^\bot \times F_2^\bot)^\op
\]
defined as in \S \ref{ss:duality}, and similarly for $\mathrm{D}^{G \times \Gm}_{F_1^\bot,(F_2')^\bot}$.
Then the result follows from the natural isomorphism
\[
R(h_A)_* \circ \mathrm{D}^{G \times \Gm}_{F_1^\bot,(F_2')^\bot} \cong \mathrm{D}^{G \times \Gm}_{F_1^\bot,F_2^\bot} \circ R(h_A)_*
\]
provided by the duality theorem \cite[Theorem VII.3.3]{H}. More precisely we need an equivariant version of the duality theorem, which can be derived from the non-equivariant version by the arguments of~\cite[\S 2.1]{MR3}.
\end{proof}

Consider now Setting (B) of \S\ref{ss:settings}.

\begin{prop}
\label{prop:compatibility-duality-B}
We have equalities
\begin{gather*}
\dual_{(F_1')^\bot,F_2^\bot} \circ \resK^{F_1^\bot, F_2^\bot}_{(F_1')^\bot, F_2^\bot} = \resK^{F_1^\bot, F_2^\bot}_{(F_1')^\bot, F_2^\bot} \circ \dual_{F_1^\bot,F_2^\bot}, \\
\bfi_{(F_1')^\bot,F_2^\bot} \circ \resK^{F_1^\bot, F_2^\bot}_{(F_1')^\bot, F_2^\bot} = \resK^{F_1^\bot, F_2^\bot}_{(F_1')^\bot, F_2^\bot} \circ \bfi_{F_1^\bot,F_2^\bot}
\end{gather*}
of morphisms $\Kth^{G \times \Gm}(F_1^\bot \times_{V^*} F_2^\bot) \to \Kth^{G \times \Gm}((F_1')^\bot \times_{V^*} F_2^\bot)$.
%
%
\end{prop}

\begin{proof}
The second equality is easy, and left to the reader. Let us consider the first one.
We denote the inclusion morphism by
\[
h_B \colon (F_1')^\bot \times F_2^\bot \hookrightarrow F_1^\bot \times F_2^\bot,
\]
and consider the duality functors $\mathrm{D}^{G \times \Gm}_{F_1^\bot,F_2^\bot}$ and $\mathrm{D}^{G \times \Gm}_{(F_1')^\bot,F_2^\bot}$
defined as in \S \ref{ss:duality}. 
The claim follows from an isomorphism of functors
\[
L(h_B)^* \circ \mathrm{D}^{G \times \Gm}_{F_1^\bot,F_2^\bot} \cong \mathrm{D}^{G \times \Gm}_{(F_1')^\bot,F_2^\bot} \circ L(h_B)^*
\]
which can be proved by arguments similar to those of \cite[Proposition II.5.8]{H}, taking into account our assumption that $X$ is a smooth variety (so that $F_1^\bot \times F_2^\bot$ and $(F_1')^\bot \times F_2^\bot$ are also smooth), which implies that every object of the bounded derived category of coherent sheaves is isomorphic to a bounded complex of locally free sheaves.
\end{proof}

\subsection{Compatibilities for $\uRR$}
\label{ss:compatibility-RR}

First, consider Setting (A) of \S\ref{ss:settings}.

\begin{prop}
\label{prop:compatibility-uRR-A}
Assume that the proper direct image morphism
\[
\coH^{G \times \Gm}_\bullet(F_1 \times_V F_2) \to \coH^{G \times \Gm}_\bullet(F_1 \times F_2)
\]
is injective.
Then we have an equality
\[
\uRR_{F_1,F_2} \circ \resK^{F_1, F_2'}_{F_1, F_2} = \resH^{F_1, F_2'}_{F_1, F_2} \circ \uRR_{F_1,F_2'}
\]
of morphisms $\Kth^{G \times \Gm}(F_1 \times_V F_2') \to \coHc^{G \times \Gm}_\bullet(F_1 \times_V F_2)$.
%
\end{prop}

\begin{proof}
Consider the following cube:
\[
\xymatrix{
\Kth^{G \times \Gm}(F_1 \times_V F_2') \ar[dd]_-{\resK^{F_1, F_2'}_{F_1, F_2}} \ar[rd]^-{\pdiK} \ar[rr]^-{\uRR_{F_1,F_2'}} & & \coHc^{G \times \Gm}_\bullet(F_1 \times_V F_2') \ar'[d][dd]^-{\resH^{F_1, F_2'}_{F_1, F_2}} \ar[rd]^-{\pdiH} & \\
& \Kth^{G \times \Gm}(F_1 \times F_2') \ar[dd]_<<<<<<{\resK} \ar[rr]^<<<<<<<<<<{(1)} & & \coHc^{G \times \Gm}_\bullet(F_1 \times F_2') \ar[dd]^-{\resH} \\
\Kth^{G \times \Gm}(F_1 \times_V F_2) \ar[rd]^-{\pdiK} \ar'[r][rr]^-{\uRR_{F_1,F_2}} && \coHc^{G \times \Gm}_\bullet(F_1 \times_V F_2) \ar[rd]^-{\pdiH} & \\
& \Kth^{G \times \Gm}(F_1 \times F_2) \ar[rr]^-{(2)} && \coHc^{G \times \Gm}_\bullet(F_1 \times F_2).
}
\]
Here the labels $\resK$ and $\resH$, resp. $\pdiK$ and $\pdiH$, indicate restriction with supports (always with respect to the morphism induced by $F_2 \hookrightarrow F_2'$), resp.~proper direct image, the arrow labelled by $(1)$ is given by $\tau^{G \times \Gm}_{F_1 \times F_2'} \cdot \bigl(1 \boxtimes (\mathrm{Td}_{F_2'}^{G \times \Gm})^{-1} \bigr)$, and the arrow labelled by $(2)$ by $\tau^{G \times \Gm}_{F_1 \times F_2} \cdot \bigl(1 \boxtimes (\mathrm{Td}_{F_2}^{G \times \Gm})^{-1} \bigr)$. The upper and lower faces of this cube commute by Theorem~\ref{thm:equivariantRR}
and the projection formula~\eqref{eqn:proj-formula-H}.
The left face commutes by definition, and the right one by Lemma \ref{lem:restriction-pushforward}. The front face commutes by Proposition~\ref{prop:RR-res}, Remark~\ref{rk:Todd} and formula~\eqref{eqn:res-cohomology}.
Using our assumption, we deduce the commutativity of the back face, which finishes the proof.
\end{proof}

Now, consider Setting (B) of \S\ref{ss:settings}. The following proposition follows from Theorem~\ref{thm:equivariantRR} and the projection formula~\eqref{eqn:proj-formula-H}.

\begin{prop}
\label{prop:compatibility-uRR-B}
We have an equality
\[
\uRR_{F_1',F_2} \circ \pdiK^{F_1, F_2}_{F_1', F_2} = \pdiH^{F_1, F_2}_{F_1', F_2} \circ \uRR_{F_1,F_2}
\]
of morphisms $\Kth^{G \times \Gm}(F_1 \times_V F_2) \to \coHc^{G \times \Gm}_\bullet(F_1' \times_V F_2)$.
%
\end{prop}

\subsection{Compatibilities for $\overline{\mathrm{RR}}$}

The proofs in this subsection are analogous to those of the corresponding statements in \S\ref{ss:compatibility-RR}; they are therefore omitted.

First, consider Setting (A) of \S\ref{ss:settings}.

\begin{prop}
\label{prop:compatibility-oRR-A}
We have an equality
\[
\oRR_{F_1^\bot,F_2^\bot} \circ \pdiK^{F_1^\bot, (F_2')^\bot}_{F_1^\bot, F_2^\bot} = \pdiH^{F_1^\bot, (F_2')^\bot}_{F_1^\bot, F_2^\bot} \circ \oRR_{F_1^\bot,(F_2')^\bot}
\]
of morphisms $\Kth^{G \times \Gm}(F_1^\bot \times_{V^*} (F_2')^\bot) \to \coHc^{G \times \Gm}_\bullet(F_1^\bot \times_{V^*} F_2^\bot)$.
%
\end{prop}

Now, consider Setting (B) of \S\ref{ss:settings}. 

\begin{prop}
\label{prop:compatibility-oRR-B}
Assume that the proper direct image morphism
\[
\coH^{G \times \Gm}_\bullet((F_1')^\bot \times_{V^*} F_2^\bot) \to \coH^{G \times \Gm}_\bullet((F_1')^\bot \times F_2^\bot)
\]
is injective.
Then we have an equality
\[
\oRR_{(F_1')^\bot,F_2^\bot} \circ \resK^{F_1^\bot, F_2^\bot}_{(F_1')^\bot, F_2^\bot} = \resH^{F_1^\bot, F_2^\bot}_{(F_1')^\bot, F_2^\bot} \circ \oRR_{F_1^\bot,F_2^\bot}
\]
of morphisms $\Kth^{G \times \Gm}(F_1^\bot \times_{V^*} F_2^\bot) \to \coHc^{G \times \Gm}_\bullet((F_1')^\bot \times_{V^*} F_2^\bot)$.
%
\end{prop}

\section{Proof of Theorem \ref{thm:LKDFourier}}
\label{sec:proof}

\subsection{A particular case}
\label{ss:particular-case}

In this subsection we study the case when $F_1=E$ and $F_2=X$ (considered as the zero-section of $E$) so that $F_1^\bot=X$, $F_2^\bot=E^*$. In this case, the assumption of Theorem \ref{thm:LKDFourier} is trivially satisfied.

\begin{lem}
\label{lem:fourier-id}
Under the identifications $E \times_V X = X \times X=X \times_{V^*} E^*$,
the isomorphism
\[
\Fourier_{E,X} \colon \coH^{G \times \Gm}_\bullet(E \times_V X) \xrightarrow{\sim} \coH_\bullet^{G \times \Gm}(X \times_{V^*} E^*)
\]
coincides with the automorphism of $\coH^{G \times \Gm}_\bullet(X \times X)$ induced by the involution of $G \times \Gm$ sending $(g,t)$ to $(g,t^{-1})$.
\end{lem}

\begin{proof}
The lemma is equivalent to the statement that the isomorphism $\coH^{G \times \Gm}_\bullet(E \times_V X) \xrightarrow{\sim} \coH_\bullet^{G \times \Gm}(X \times_{V^\diamond} E^\diamond)$ induced by the equivalence $\mathfrak{F}_V$ of~\S\ref{ss:Fourier-transform} is the identity morphism of $\coH^{G \times \Gm}_\bullet(X \times X)$.

Using the canonical isomorphism of \S\ref{ss:equiv} in the case $V=\{0\}$, $F_1=F_2=X$ we obtain an isomorphism
\[
\alpha \colon \coH^{G \times \Gm}_\bullet(X \times X) \xrightarrow{\sim} \Ext^{2\dim(X)-\bullet}_{\calD_{\mathrm{const}}^{G \times \Gm}(\pt)} \bigl( (p_0)_! \uC_X, (p_0)_! \uC_X \bigr),
\]
where $p_0 \colon X \to \pt$ is the projection. Then the composition 
\[
\coH^{G \times \Gm}_\bullet(X \times X) = \coH^{G \times \Gm}_\bullet(E \times_V X) \cong \Ext^{2\dim(X)-\bullet}_{\calD_{\mathrm{const}}^{G \times \Gm}(V)}(p_! \uC_E,p_! \uC_X)
\]
sends each $c \in \coH^{G \times \Gm}_i(X \times X)$ to the morphism
\[
p_! \uC_E = \uC_V \boxtimes (p_0)_! \uC_X \xrightarrow{\varphi \boxtimes \alpha(c)} \uC_{\{0\}} \boxtimes (p_0)_! \uC_X [2\dim(X) - i] = p_! \uC_X [2\dim(X) - i]
\]
where $\varphi \colon \uC_V \to \uC_{\{0\}}$ is the $({}^*,{}_*)$-adjunction morphism for the inclusion $\{0\} \hookrightarrow V$, and we use the identification $V=V \times \pt$. Similarly, the composition
\[
\coH^{G \times \Gm}_\bullet(X \times X) = \coH^{G \times \Gm}_\bullet(X \times_{V^\diamond} E^\diamond) \cong \Ext^{2\dim(E^*)-\bullet}_{\calD_{\mathrm{const}}^{G \times \Gm}(V^\diamond)}({\check p}_! \uC_X,{\check p}_! \uC_{E^\diamond})
\]
sends each $c \in \coH^{G \times \Gm}_i(X \times X)$ to the morphism
\[
{\check p}_! \uC_X = \uC_{\{0\}} \boxtimes (p_0)_! \uC_X \xrightarrow{\psi \boxtimes \alpha(c)} \uC_{V^\diamond} \boxtimes (p_0)_! \uC_X [2\dim(E^*) - i]
\]
where $\psi \colon \uC_{\{0\}} \to \uC_{V^\diamond}[2\dim(V^*)]$ is the $({}_!,{}^!)$-adjunction morphism for the inclusion $\{0\} \hookrightarrow V^\diamond$, and we use the identification $V^\diamond=V^\diamond \times \pt$. Now using Lemma \ref{lem:fourier-adj} we obtain that $\mathfrak{F}_V$ sends $\varphi \boxtimes \alpha(c)$ to $\psi \boxtimes \alpha(c)$, and the lemma follows.
\end{proof}

With this result in hand we can prove Theorem \ref{thm:LKDFourier} in our particular case.

\begin{lem}
\label{lem:thm-particular-case}
Theorem {\rm \ref{thm:LKDFourier}} holds in the case $F_1=E$, $F_2=X$.
\end{lem}

\begin{proof}
We have $F_1 \times_V F_2=X \times X$, and also $F_1^\bot \times_{V^*} F_2^\bot=X \times X$. There exists a natural morphism of dg-schemes
\[
(\Delta V \times X \times X) \rcap_{E \times E} (E \times X) \to (X \times X) \rcap_{X \times X} (X \times X)
\]
associated with the morphism of vector bundles $p \times p \colon E \times E \to X \times X$, see \cite[\S 3.2]{MR3}. In our case it is easily checked that this morphism is a quasi-isomorphism, hence it induces an equivalence of triangulated categories
\[
L\Phi^* \colon \calD^c_{G \times \Gm}((X \times X) \rcap_{X \times X} (X \times X)) \xrightarrow{\sim} \calD^c_{G \times \Gm}((\Delta V \times X \times X) \rcap_{E \times E} (E \times X)),
\]
see~\cite[Proposition~1.3.2]{MR}. Moreover by definition the left-hand side coincides with the category $\calD^b \Coh^{G \times \Gm}(X \times X)$, so that $L\Phi^*$ can (and will) be considered as an equivalence from $\calD^b \Coh^{G \times \Gm}(X \times X)$ to $\calD^c_{G \times \Gm}((\Delta V \times X \times X) \rcap_{E \times E} (E \times X))$. It is easily checked that the induced automorphism of $\Kth^{G \times \Gm}(X \times X)$ is the identity.

Similarly, the morphism dual to $p \times p$ induces a quasi-isomorphism
\[
(X \times X) \rcap_{X \times X} (X \times X) \to (\Delta V^* \times X \times X) \rcap_{E^* \times E^*} (X \times E^*),
\]
hence an equivalence of triangulated categories
\[
R\Psi_* \colon \calD^b \Coh^{G \times \Gm}(X \times X) \xrightarrow{\sim} \calD^c_{G \times \Gm}((\Delta V^* \times X \times X) \rcap_{E^* \times E^*} (X \times E^*)),
\]
which induces the identity morphism in $\Kth$-homology.

If $\frakK_{X,X}$ denotes the linear Koszul duality equivalence defined as in~\S\ref{ss:lkd} (in the case $V=\{0\}$, $F_1=F_2=E=X$), by \cite[Proposition 3.4]{MR3} there exists an isomorphism
\[
\frakK_{E,X} \circ L\Phi^* \cong R\Psi_* \circ \frakK_{X,X}.
\]
Using the remarks above and the definition of the equivalence $\frakK_{X,X}$, we deduce that, if $\calG$ is in $\calD^b \Coh^G(X \times X)$ (considered as an object of $\calD^b \Coh^{G \times \Gm}(X \times X)$ with trivial $\Gm$-action), the morphism $\Koszul_{E,X}$ sends the class of $\calG \langle m \rangle$ to the class of
\[
R\sheafHom_{\calO_{X \times X}}(\calG,\calO_X \boxtimes \omega_X)\langle m \rangle[\dim(X)+m].
\]


Using the compatibility of Grothendieck--Serre duality with proper direct images (as in the proof of Proposition~\ref{prop:compatibility-duality-A}) one easily checks that, with similar notation, $\dual_{X,E^*}$ sends the class of $\calG \langle m \rangle$ to the class of
\[
R\sheafHom_{\calO_{X \times X}}(\calG,\calO_X \boxtimes \omega_X)\langle -m \rangle[\dim(X)].
\]
We deduce that $\dual_{X,E^*} \circ \Koszul_{E,X}$ sends the class of $\calG \langle m \rangle$ to the class of $\calG \langle -m \rangle [-m]$, and then that $\bfi_{X,E^*} \circ \dual_{X,E^*} \circ \Koszul_{E,X}$ identifies with the automorphism of $\Kth^{G \times \Gm}(X \times X)$ induced by the involution of $G \times \Gm$ sending $(g,t)$ to $(g,t^{-1})$.

%

The statement in the lemma follows from this description, Lemma~\ref{lem:fourier-id}, and the compatibility of the Riemann--Roch maps with inverse image (in $\Kth$-homology and Borel--Moore homology) under an automorphism of $G \times \Gm$.
\end{proof}

\subsection{Compatibility with inclusion}

Consider first Setting (A) of \S\ref{ss:settings}. 

\begin{prop}
\label{prop:inclusion-A}
\begin{enumerate}
\item
We have an equality
\begin{multline*}
\oRR_{F_1^\bot, F_2^\bot} \circ \bfi_{F_1^\bot, F_2^\bot} \circ \dual_{F_1^\bot, F_2^\bot} \circ \Koszul_{F_1,F_2} \circ \resK^{F_1, F_2'}_{F_1, F_2} = \\
\pdiH^{F_1^\bot, (F_2')^\bot}_{F_1^\bot, F_2^\bot} \circ \oRR_{F_1^\bot, (F_2')^\bot} \circ \bfi_{F_1^\bot, (F_2')^\bot} \circ \dual_{F_1^\bot, (F_2')^\bot} \circ \Koszul_{F_1,F_2'}
\end{multline*}
of morphisms $\Kth^{G \times \Gm}(F_1 \times_V F_2') \to \coHc_\bullet^{G \times \Gm}(F_1^\bot \times_{V^*} F_2^\bot)$.
%
\item
Assume that the proper direct image morphism
\[
\coH^{G \times \Gm}_\bullet(F_1 \times_V F_2) \to \coH^{G \times \Gm}_\bullet(F_1 \times F_2)
\]
is injective.
Then we have an equality
\[
\Fourier_{F_1,F_2} \circ \uRR_{F_1,F_2} \circ \resK^{F_1, F_2'}_{F_1, F_2} = \pdiH^{F_1^\bot, (F_2')^\bot}_{F_1^\bot, F_2^\bot} \circ \Fourier_{F_1,F_2'} \circ \uRR_{F_1,F_2'}
\]
of morphisms $\Kth^{G \times \Gm}(F_1 \times_V F_2') \to \coHc_\bullet^{G \times \Gm}(F_1^\bot \times_{V^*} F_2^\bot)$.
%
%
\end{enumerate}
\end{prop}

\begin{proof}
(1) follows from Propositions \ref{prop:compatibility-LKD-A}, \ref{prop:compatibility-duality-A} and \ref{prop:compatibility-oRR-A}. (2) follows from Propositions \ref{prop:compatibility-uRR-A} and \ref{prop:compatibility-fourier-A}.
\end{proof}

Consider now Setting (B) of \S\ref{ss:settings}.

\begin{prop}
\label{prop:inclusion-B}
\begin{enumerate}
\item
Assume that the proper direct image morphism
\[
\coH^{G \times \Gm}_\bullet((F_1')^\bot \times_{V^*} F_2^\bot) \to \coH^{G \times \Gm}_\bullet((F_1')^\bot \times F_2^\bot)
\]
is injective.
Then we have an equality
\begin{multline*}
\oRR_{(F_1')^\bot, F_2^\bot} \circ \bfi_{(F_1')^\bot, F_2^\bot} \circ \dual_{(F_1')^\bot, F_2^\bot} \circ \Koszul_{F_1',F_2} \circ \pdiK^{F_1, F_2}_{F_1', F_2} \\
= \resH^{F_1^\bot, F_2^\bot}_{(F_1')^\bot, F_2^\bot} \circ \oRR_{F_1^\bot, F_2^\bot} \circ \bfi_{F_1^\bot, F_2^\bot} \circ \dual_{F_1^\bot, F_2^\bot} \circ \Koszul_{F_1,F_2}
\end{multline*}
of morphisms $\Kth^{G \times \Gm}(F_1 \times_V F_2) \to \coHc_\bullet^{G \times \Gm}((F_1')^\bot \times_{V^*} F_2^\bot)$.
%
\item
We have an equality
\[
\Fourier_{F_1',F_2} \circ \underline{\mathrm{RR}}_{F_1',F_2} \circ \pdiK^{F_1, F_2}_{F_1', F_2} = \resH^{F_1^\bot, F_2^\bot}_{(F_1')^\bot, F_2^\bot} \circ \Fourier_{F_1,F_2} \circ \underline{\mathrm{RR}}_{F_1,F_2}
\]
of morphisms $\Kth^{G \times \Gm}(F_1 \times_V F_2) \to \coHc_\bullet^{G \times \Gm}((F_1')^\bot \times_{V^*} F_2^\bot)$.
%
\end{enumerate}
\end{prop}

\begin{proof}
(1) follows from Propositions \ref{prop:compatibility-LKD-B}, \ref{prop:compatibility-duality-B} and \ref{prop:compatibility-oRR-B}. (2) follows from Propositions \ref{prop:compatibility-uRR-B} and \ref{prop:compatibility-fourier-B}.
\end{proof}

\subsection{Proof of Theorem \ref{thm:LKDFourier}}
\label{ss:proof-thm}

By assumption, the proper direct image morphism 
\[
\coH^{G \times \Gm}_\bullet(F_1^{\bot} \times_{V^*} F_2^{\bot}) \to \coH^{G \times \Gm}_\bullet(F_1^{\bot} \times_{V^*} E^*)
\]
is injective. Hence the same is true for the induced morphism
\[
\coHc_\bullet^{G \times \Gm}(F_1^{\bot} \times_{V^*} F_2^{\bot}) \to \coHc_\bullet^{G \times \Gm}(F_1^{\bot} \times_{V^*} E^*).
\]
By Proposition \ref{prop:inclusion-A} applied to the inclusion $X \subset F_2$, we deduce that it suffices to prove the theorem in the case $F_2=X$. (Note that the inclusion $F_1 \times_V X \hookrightarrow F_1 \times X$ is the inclusion of the zero section in the vector bundle $F_1 \times X$ over $X \times X$. Hence the injectivity assumption in Proposition \ref{prop:inclusion-A}(2) holds by Lemma \ref{lem:Euler-class}.)

Now consider the inclusion of vector subbundles $F_1 \subset E$ (again with $F_2=X$). In this case, the morphism
\[
\resH^{F_1^\bot, E^*}_{X, E^*} \colon 
\coH^{G \times \Gm}_\bullet(F_1^\bot \times_{V^*} E^*) \to \coH^{G \times \Gm}_{\bullet-2\rk(F_1^\bot)}(X \times_{V^*} E^*) = \coH^{G \times \Gm}_{\bullet-2\rk(F_1^\bot)}(X \times X)
\]
is the Thom isomorphism for the vector bundle $F_1^\bot \times_{V^*} E^* \cong F_1^\bot \times X$ over $X \times X$; in particular it is injective. Using Proposition \ref{prop:inclusion-B} we deduce that it suffices to prove the theorem in the case $F_1=E$, $F_2=X$. (Note that in our situation the inclusion $E^\bot \times_{V^*} X^\bot \hookrightarrow E^\bot \times X^\bot$ is the inclusion of the zero section in the vector bundle $X \times E^*$ over $X \times X$, so that the injectivity assumption in Proposition \ref{prop:inclusion-B}(1) holds by Lemma \ref{lem:Euler-class}.) In this case the theorem holds by Lemma~\ref{lem:thm-particular-case}, hence our proof is complete.


\appendix

\section{Proofs of some technical results}
\label{sec:appendix}

\subsection{Conventions}

In \S\S\ref{ss:commutative-diagrams}--\ref{ss:adjunction} we work in the $A$-equivariant constructible derived category of some complex algebraic $A$-varieties (for some arbitrary complex linear algebraic group $A$). If $X,Y,Z$ are $A$-varieties and $f\colon X \to Y$, $g\colon Y \to Z$ are $A$-equivariant morphisms, then there exist canonical ``composition'' isomorphisms
\[
g_* f_* \cong (g \circ f)_*, \quad g_! f_! \cong (g \circ f)_!, \quad f^* g^* \cong (g \circ f)^*, \quad f^! g^! \cong (g \circ f)^!,
\]
which we will all indicate by $\Co$. Similarly, given a cartesian square
\[
\xymatrix{
Y' \ar[d]_-{g'} \ar[r]^-{f'} \ar@{}[rd]|{\square} &Z' \ar[d]^-{g} \\
Y \ar[r]^-{f} & Z
}
\]
of $A$-equivariant morphisms, there exist canonical ``base change'' isomorphisms
\[
f^* g_! \cong (g')_! (f')^*, \quad f^! g_* \cong (g')_* (f')^!,
\]
which we will indicate by $\BC$.

\subsection{Some commutative diagrams}
\label{ss:commutative-diagrams}

Consider a commutative diagram of $A$-varieties and $A$-equivariant morphisms
\[
\xymatrix@R=0.3cm{
Y \ar[rr]^-{g} \ar[rd]^-{a} \ar[dd]_-{c} & & Z \ar[rd]^-{d} \ar'[d][dd]_<<<<{f} & \\
& Y' \ar[rr]^<<<<{g'} \ar[ld]_-{b} & & Z' \ar[ld]^-{e} \\
Y'' \ar[rr]^-{g''} & & Z''
}
\]
where all squares are cartesian. The following lemma is a restatement of \cite[Lemma B.7(d)]{AHR}.

\begin{lem}
\label{lem:ahr}
The following diagram of isomorphisms of functors commutes:
\[
\xymatrix@C=2cm{
(g'')^! f_* \ar[r]^{\Co}_-{\sim} \ar[d]_-{\BC}^-{\wr} & (g'')^! e_* d_* \ar[r]^-{\BC}_-{\sim} & b_* (g')^! d_* \ar[d]^-{\BC}_-{\wr} \\
c_* g^! \ar[rr]^{\Co}_-{\sim} & & b_* a_* g^!.
}
\]
\end{lem}

Now, consider $A$-equivariant morphisms
\[
\xymatrix{
W \ar[r]^-{f} & X \ar[r]^-{g} & Y \ar[r]^-{h} & Z.
}
\]
The following lemma is a restatement of \cite[Lemma B.4(a) \& Lemma B.4(d)]{AHR}.

\begin{lem}
\label{lem:ahr2}
The following diagrams of isomorphisms of functors commute:
\[
\xymatrix@C=2cm{
h_* g_* f_* \ar[r]^-{\Co}_-{\sim} \ar[d]_-{\Co}^-{\wr} & h_* (g \circ f)_* \ar[d]^-{\Co}_-{\wr} \\
(h \circ g)_* f_* \ar[r]^-{\Co}_-{\sim} & (h \circ g \circ f)_*,
}
\qquad
\xymatrix@C=2cm{
f^! g^! h^! \ar[r]^-{\Co}_-{\sim} \ar[d]_-{\Co}^-{\wr} & f^! (h \circ g)^! \ar[d]^-{\Co}_-{\wr} \\
(g \circ f)^! h^! \ar[r]^-{\Co}_-{\sim} & (h \circ g \circ f)^!.
}
\]
\end{lem}

\subsection{Base change and adjunction}

Consider a cartesian diagram
\begin{equation}
\label{eqn:cartesian-diagram}
\vcenter{
\xymatrix{
Y' \ar[d]_-{g'} \ar[r]^-{f'} \ar@{}[rd]|{\square} & Z' \ar[d]^-{g} \\
Y \ar[r]^-{f} & Z
}
}
\end{equation}
of $A$-varieties and $A$-equivariant morphisms.
Then there exists a canonical morphism of functors
\begin{equation}
\label{eqn:morphism-adjunction-2}
(f')_! (g')^! \to g^! f_!
\end{equation}
which can be defined equivalently as the composition
\[
(f')_! (g')^! \to (f')_! (g')^! f^! f_! \xrightarrow[\sim]{\Co} (f')_! (f \circ g')^! f_! \xrightarrow[\sim]{\Co} (f')_! (f')^! g^! f_!  \to g^! f_!
\]
or as the composition
\[
(f')_! (g')^! \to g^! g_! (f')_! (g')^! \xrightarrow[\sim]{\Co} g^! (g \circ f')_! (g')^! \xrightarrow[\sim]{\Co} g^! f_! (g')_! (g')^! \to g^! f_!
\]
where the unlabelled arrows are induced by the appropriate adjunction morphisms. (We leave it to the reader to check that these compositions coincide.)

As stated in \cite[Exercise III.9]{KS}, the following diagram is commutative, where vertical arrows are induced by the canonical morphisms $f_! \to f_*$ and $(f')_! \to (f')_*$:
\[
\xymatrix@C=1.5cm{
(f')_! (g')^! \ar[r]^-{\eqref{eqn:morphism-adjunction-2}} \ar[d] & g^! f_! \ar[d] \\
(f')_* (g')^! \ar[r]^-{\BC}_-{\sim} & g^! f_*.
}
\]
We deduce the following.

\begin{lem}
\label{lem:BC-adjunction}
If $f$ (hence also $f'$) is proper, then the base change isomorphism 
$(f')_* (g')^! \cong g^! f_*$ coincides, under the natural identifications $f_! = f_*$ and $(f')_! = (f')_*$, with morphism 
\eqref{eqn:morphism-adjunction-2}.
\end{lem}

\subsection{Some consequences}
\label{ss:adjunction}

Consider again a cartesian diagram
\eqref{eqn:cartesian-diagram},
and assume that $f$ (hence also $f'$) is proper. 

First, one can consider the diagram of morphisms of functors
\begin{equation}
\label{eqn:diagram-adjunction-!-1}
\vcenter{
\xymatrix@C=1.5cm{
(f')_* (g')^! f^! \ar[d]^-{\Co}_-{\wr} \ar[r]_-{\sim}^-{\BC} & g^! f_* f^! \ar[dd] \\
(f')_* (f \circ g')^! \ar[d]^-{\Co}_-{\wr} & \\
(f')_* (f')^! g^! \ar[r] & g^!
}
}
\end{equation}
where the right vertical arrow is induced by the adjunction morphism $f_* f^! = f_! f^! \to \mathrm{id}$ and the lower horizontal arrow is induced by the adjunction morphism $(f')_* (f')^! = (f')_! (f')^! \to \mathrm{id}$.

\begin{lem}
\label{lem:adjunction-!-1}
Diagram \eqref{eqn:diagram-adjunction-!-1} is commutative.
\end{lem}

\begin{proof}
The claim follows from Lemma \ref{lem:BC-adjunction} (using the first description of morphism \eqref{eqn:morphism-adjunction-2}) and the fact that the composition of adjunction morphisms
\[
f^! \to f^! f_! f^! \to f^!
\]
is the identity.
\end{proof}

One can also consider the diagram of morphisms of functors
\begin{equation}
\label{eqn:diagram-adjunction-!-2}
\vcenter{
\xymatrix@C=1.5cm{
g_! (f')_* (g')^! \ar[d]^-{\Co}_-{\wr} \ar[r]_-{\sim}^-{\BC} & g_! g^! f_* \ar[dd] \\
(g \circ f')_! (g')^! \ar[d]^-{\Co}_-{\wr} & \\
f_* (g')_! (g')^! \ar[r] & f_*
}
}
\end{equation}
where unlabelled arrows are induced by adjunction, and in the left-hand side we use the identifications $f_!=f_*$ and $(f')_! = (f')_*$.

\begin{lem}
\label{lem:adjunction-!-2}
Diagram \eqref{eqn:diagram-adjunction-!-2} is commutative.
\end{lem}

\begin{proof}
The claim follows from Lemma \ref{lem:BC-adjunction} (using the second description of morphism \eqref{eqn:morphism-adjunction-2}) and the fact that the composition of adjunction morphisms
\[
g_! \to g_! g^! g_! \to g_!
\]
is the identity.
\end{proof}

%
%

\subsection{Restriction with supports in Borel--Moore homology}
\label{ss:restriction-with-supports}

As in \S\ref{ss:homology-cohomology},
let $A$ be a complex linear algebraic group, let $Y$ be a smooth complex $A$-variety, and let $Y' \subset Y$ be a smooth $A$-stable closed subvariety. Consider another $A$-stable closed subvariety $Z \subset Y$, not necessarily smooth, and set $Z':=Z \cap Y'$. Then we have a cartesian diagram of closed inclusions
\[
\xymatrix{
Z' \ar@{^{(}->}[r]^-{i'} \ar@{^{(}->}[d]_-{g} & Y' \ar@{^{(}->}[d]^-{f} \\
Z \ar@{^{(}->}[r]^-{i} & Y.
}
\]
Set $N:=2\dim(Y)-2\dim(Y')$.
The ``restriction with supports'' morphism
\begin{equation*}
\resH^Z_{Z'} \colon \coH^{A}_{\bullet}(Z) \to \coH^A_{\bullet-N}(Z')
\end{equation*}
associated with the inclusion $Y' \hookrightarrow Y$ is defined as follows. Consider the composition
\[
i^! \to i^! f_* f^* \xrightarrow[\sim]{\BC} g_* (i')^! f^*
\]
where the first morphism is induced by the adjunction morphism $\mathrm{id} \to f_* f^*$. Then applying this composition to $\uD_Y$ and using the isomorphisms
\[
i^! \uD_Y \cong \uD_Z, \quad f^* \uD_Y \cong f^* \uC_Y[2\dim(Y)] \cong \uC_{Y'}[2\dim(Y)] \cong \uD_{Y'}[N], \quad \text{and} \quad (i')^! \uD_{Y'} \cong \uD_{Z'}
\]
we obtain a morphism
\[
\uD_Z \to g_* \uD_{Z'}[N].
\]
Taking (equivariant) cohomology provides our morphism $\resH^Z_{Z'}$.

The same construction, applied to the subvariety $Y' \subset Y$ instead of $Z$, provides another morphism
\[
\resH^Y_{Y'} \colon \coH^{A}_{\bullet}(Y) \to \coH^A_{\bullet-N}(Y')
\]

\begin{lem}
\label{lem:restriction-pushforward}
The following diagram is commutative:
\[
\xymatrix@C=2cm{
\coH^{A}_{\bullet}(Z) \ar[r]^-{\resH^Z_{Z'}} \ar[d]_-{\pdiH_i} & \coH^A_{\bullet-N}(Z') \ar[d]^-{\pdiH_{i'}} \\
\coH^{A}_{\bullet}(Y) \ar[r]^-{\resH^Y_{Y'}} & \coH^A_{\bullet-N}(Y').
}
\]
\end{lem}

\begin{proof}
Consider the following diagram:
\[
\xymatrix{
i_!i^! \ar[r] \ar[dd] & i_! i^! f_* f^* \ar[r]^-{\BC}_-{\sim} \ar[rdd] & i_! g_* (i')^! f^* \ar[d]^-{(\ddag)}_-{\wr} \\
&& f_* (i')_! (i')^! f^* \ar[d] \\
\mathrm{id} \ar[rr] & & f_* f^*.
}
\]
Here the unlabelled arrows are induced by the appropriate adjunction morphisms, and the arrow labelled with $(\ddag)$ is induced by the composition of natural isomorphisms
\[
i_! g_* \cong i_! g_! \xrightarrow[\sim]{\Co} (i \circ g)_!  \xrightarrow[\sim]{\Co} f_! (i')_! \cong f_* (i')_!.
\]
The left part of the diagram is clearly commutative, and the right part is commutative by Lemma \ref{lem:adjunction-!-2}. Hence the diagram as a whole is commutative. Now, when applied to $\uD_Y$ and after taking equivariant cohomology, this diagram induces the diagram of the lemma, hence these remarks finish the proof. (In this argument we also use the left diagram in Lemma \ref{lem:ahr2}, which allows to forget about the ``$\Co$'' isomorphisms in the right-hand side of the diagram once equivariant cohomology is taken.)
\end{proof}

\subsection{Proof of Proposition \ref{prop:inclusion-Ginzburg-A}(1)}
\label{ss:proof-inclusion-Ginzburg-1}

By functoriality of isomorphism \eqref{eqn:morphisms-cohomology} the following diagram commutes, where the right vertical morphism is induced by $\mathrm{adj}^*_{F_2,F_2'}$:
\begin{equation}
\label{eqn:appendix-diagram-1}
\vcenter{
\xymatrix@C=2cm{
\Ext^{\bullet}_{\calD^{G \times \Gm}_{\mathrm{const}}(V)}(p_* \uC_{F_1}, p_* \uC_{F_2'}) \ar[d]_-{(p_* \mathrm{adj}^*_{F_2,F_2'}) \circ (\cdot)} \ar[r]^-{\eqref{eqn:morphisms-cohomology}}_-{\sim} & \coH^\bullet_{G \times \Gm}(E \times_V E, j^!(\uD_{F_1} \boxtimes \uC_{F_2'})) \ar[d] \\
\Ext^{\bullet}_{\calD^{G \times \Gm}_{\mathrm{const}}(V)}(p_* \uC_{F_1}, p_* \uC_{F_2}) \ar[r]^-{\eqref{eqn:morphisms-cohomology}}_-{\sim} & \coH^\bullet_{G \times \Gm}(E \times_V E, j^!(\uD_{F_1} \boxtimes \uC_{F_2})).
}
}
\end{equation}

Now, consider the following diagram, where all squares are cartesian and all morphisms are closed inclusions:
\[
\xymatrix{
F_1 \times_V F_2 \ar[r]^-{c} \ar[d]_-{k} \ar@/^20pt/[rr]^-{b} & F_1 \times_V F_2' \ar[d]^-{k'} \ar[r]^-{b'} & E \times_V E \ar[d]^-{j} \\
F_1 \times F_2 \ar[r]^-{d} \ar@/_20pt/[rr]_-{a} & F_1 \times F_2' \ar[r]^-{a'} & E \times E. \\
}
\]
Then under the natural identifications 
\begin{align*}
\coH^\bullet_{G \times \Gm}(E \times_V E, j^!(\uD_{F_1} \boxtimes \uC_{F_2})) &\cong \coH^\bullet_{G \times \Gm}(E \times_V E, j^! a_* (\uD_{F_1} \boxtimes \uC_{F_2})), \\
\coH^\bullet_{G \times \Gm}(E \times_V E, j^!(\uD_{F_1} \boxtimes \uC_{F_2'})) &\cong \coH^\bullet_{G \times \Gm}(E \times_V E, j^! (a')_* (\uD_{F_1} \boxtimes \uC_{F_2'})),
\end{align*}
the right vertical morphism in \eqref{eqn:appendix-diagram-1} identifies with the morphism
\begin{equation}
\label{eqn:appendix-morphism-1}
\coH^\bullet_{G \times \Gm} \bigl(E \times_V E,  j^!(a')_* (\uD_{F_1} \boxtimes \uC_{F_2'}) \bigr) \to \coH^\bullet_{G \times \Gm} \bigl(E \times_V E, j^! a_* (\uD_{F_1} \boxtimes \uC_{F_2}) \bigr)
\end{equation}
induced by the adjunction morphism $\mathrm{id} \to d_* d^*$ (through the ``composition'' isomorphism $(a')_* d_* \cong a_*$).

Consider the following diagram of morphisms of functors:
\[
\xymatrix@C=1.5cm{
j^! (a')_* \ar[rr]^-{\BC}_-{\sim} \ar[d] \ar@{.>}@/_50pt/[dd]_-{(\star)} & & (b')_* (k')^! \ar[d] \ar@{.>}@/^40pt/[dd]^-{(\dag)} \\
j^! (a')_* d_* d^* \ar[rr]^-{\BC}_-{\sim} \ar[d]_-{\Co}^-{\wr} & & (b')_* (k')^! d_* d^* \ar[d]^-{\BC}_-{\wr} \\
j^! a_* d^* \ar[r]^-{\BC}_-{\sim} & b_* k^! d^* \ar[r]^-{\Co}_-{\sim} & (b')_* c_* k^! d^*.
}
\]
Here the upper vertical arrows are induced by the adjunction morphism $\mathrm{id} \to d_* d^*$, and other arrows are either base change or composition isomorphisms as indicated. The upper square is clearly commutative, and the lower square is commutative by Lemma \ref{lem:ahr}. Hence the whole diagram is commutative, which allows to define the dotted arrows uniquely. The arrow labelled with $(\star)$ is the morphism which defines \eqref{eqn:appendix-morphism-1}, and the arrow labelled with $(\dag)$ is the morphism used in the definition of restriction with supports $\resH^{F_1, F_2'}_{F_1, F_2}$,
see 
\S\ref{ss:restriction-with-supports}.
Applying this diagram to $\uD_{F_1} \boxtimes \uC_{F_2'}$ and taking equivariant cohomology allows to finish the proof of Proposition \ref{prop:inclusion-Ginzburg-A}(1). (In this argument we also use the left diagram in Lemma \ref{lem:ahr2}, which allows e.g.~to forget about the ``$\Co$'' isomorphism on the lower line once equivariant cohomology is taken.)

\subsection{Proof of Proposition \ref{prop:inclusion-Ginzburg-A}(2)}
\label{ss:proof-inclusion-Ginzburg-2}

Consider the following diagram, where all squares are cartesian and all morphisms are closed inclusions:
\[
\xymatrix{
F_1^\bot \times_{V^*} (F_2')^\bot \ar[r]^-{\tilde{c}} \ar[d]_-{\tilde{k}} \ar@/^20pt/[rr]^-{\tilde{b}} & F_1^\bot \times_{V^*} F_2^\bot \ar[d]^-{\tilde{k}'} \ar[r]^-{\tilde{b}'} & E^* \times_{V^*} E^* \ar[d]^-{\tilde{j}} \\
F_1^\bot \times (F_2')^\bot \ar[r]^-{\tilde{d}} \ar@/_20pt/[rr]_-{\tilde{a}} & F_1^\bot \times F_2^\bot \ar[r]^-{\tilde{a}'} & E^* \times E^*. \\
}
\]
Then by functoriality of isomorphism \eqref{eqn:morphisms-cohomology} we have a commutative diagram
\begin{equation}
\label{eqn:appendix-diagram-2}
\vcenter{
{\footnotesize
\xymatrix@C=0.35cm{
\Ext^\bullet_{\calD^{G \times \Gm}_{\mathrm{const}}(V^*)}({\check p}_* \uC_{F_1^\bot},{\check p}_* \uC_{(F_2')^\bot}) \ar[r]^-{\sim} \ar[d]_-{({\check p}_* \mathrm{adj}^!_{(F_2')^\bot,F_2^\bot}) \circ (\cdot)} & \coH_{G \times \Gm}^\bullet\bigl(E^* \times_{V^*} E^*, \tilde{j}^! \tilde{a}_*(\uD_{F_1^\bot} \boxtimes \uC_{(F_2')^\bot} ) \bigr) \ar[d] \\
\Ext^{\bullet+2\rk(F_2^\bot)-2\rk((F_2')^\bot)}_{\calD^{G \times \Gm}_{\mathrm{const}}(V^*)}({\check p}_* \uC_{F_1^\bot},{\check p}_* \uC_{F_2^\bot}) \ar[r]^-{\sim} & \coH_{G \times \Gm}^{\bullet+2\rk(F_2^\bot)-2\rk((F_2')^\bot)}\bigl(E^* \times_{V^*} E^*, \tilde{j}^! (\tilde{a}')_*(\uD_{F_1^\bot} \boxtimes \uC_{F_2^\bot} ) \bigr),
}
}
}
\end{equation}
where horizontal arrows are induced by isomorphism \eqref{eqn:morphisms-cohomology} and the right vertical morphism is induced by the adjunction morphism $\tilde{d}_! \tilde{d}^! \to \mathrm{id}$ (through the isomorphisms $(\tilde{a}')_* \tilde{d}_! \cong (\tilde{a}')_* \tilde{d}_* \cong \tilde{a}_*$ and $\uC_{(F_2')^\bot} \cong \uD_{(F_2')^\bot}[-2\dim((F_2')^\bot)]$, $\uC_{F_2^\bot} \cong \uD_{F_2^\bot}[-2\dim(F_2^\bot)]$).

Consider the following diagram of morphisms of functors:
\[
\xymatrix@C=1.5cm{
\tilde{j}^! \tilde{a}_* \tilde{d}^! \ar[r]^-{\BC}_-{\sim} \ar[dd]_-{\Co}^-{\wr} \ar@{.>}@/_45pt/[ddd]_-{(\#)} & \tilde{b}_* \tilde{k}^! \tilde{d}^! \ar[d]^-{\Co}_-{\wr} \ar@{.>}@/^190pt/[ddd]^-{(\flat)} & \\
& (\tilde{b}')_* \tilde{c}_* \tilde{k}^! \tilde{d}^! \ar[d]^-{\BC}_-{\wr} \ar[r]^-{\Co}_-{\sim} & (\tilde{b}')_* \tilde{c}_* (\tilde{d} \circ \tilde{k})^! \ar[d]^-{\Co}_-{\wr} \\
\tilde{j}^! (\tilde{a}')_* \tilde{d}_* \tilde{d}^! \ar[r]^-{\BC}_-{\sim} \ar[d] & (\tilde{b}')_* (\tilde{k}')^! \tilde{d}_* \tilde{d}^! \ar[d] & (\tilde{b}')_* \tilde{c}_* \tilde{c}^! (\tilde{k}')^! \ar[ld] \\
\tilde{j}^! (\tilde{a}')_* \ar[r]^-{\BC}_-{\sim} & (\tilde{b}')_* (\tilde{k}')^! & 
}
\]
Here all the unlabelled arrows are induced by the appropriate adjunction morphisms (using the identifications $\tilde{c}_*=\tilde{c}_!$ and $\tilde{d}_*=\tilde{d}_!$). The upper square is commutative by Lemma \ref{lem:ahr}, the lower square is obviously commutative, and the right square is commutative by Lemma \ref{lem:adjunction-!-1}. Hence the diagram as a whole is commutative, which allows to define the dotted arrows uniquely. The arrow labelled with $(\#)$ is the one which induces the right arrow in diagram \eqref{eqn:appendix-diagram-2} (when applied to $\uD_{F_1^\bot \times F_2^\bot}$), while the arrow labelled with $(\flat)$ is the one which induces the proper direct image morphism $\pdiH^{F_1^\bot, (F_2')^\bot}_{F_1^\bot, F_2^\bot}$
(again when applied to $\uD_{F_1^\bot \times F_2^\bot}$), see \cite[\S 8.3.19]{CG}. 
The result
follows. (As in \S\ref{ss:proof-inclusion-Ginzburg-1}, in this argument we also use the diagrams of Lemma \ref{lem:ahr2}.)

\subsection{A lemma on Euler classes}

Let $A$ be a complex linear algebraic group acting on a smooth complex algebraic variety $Y$, and let $F \to Y$ be an $A$-equivariant vector bundle of rank $r$. We consider $F$ (hence also its zero-section $Y$) as an $A \times \Gm$-variety with the $\Gm$-action defined as in \S\ref{ss:Fourier-transform}. Note that, as $\Gm$ acts trivially on $Y$, there exists a canonical isomorphism of graded algebras
\begin{equation}
\label{eqn:cohomology-Gm}
\coH_{A \times \Gm}^\bullet(Y) \ \cong \ \coH_A^\bullet(Y) \otimes_{\C} \coH_{\Gm}^\bullet(\pt).
\end{equation}

\begin{lem}
\label{lem:Euler-class}
The proper direct image morphism
\[
\coH^{A \times \Gm}_\bullet(Y) \to \coH^{A \times \Gm}_{\bullet}(F)
\]
associated with the inclusion $Y \hookrightarrow F$ is injective.
\end{lem}

\begin{proof}
It is well known that the composition of our morphism with the Thom isomorphism $ \coH^{A \times \Gm}_{\bullet}(F) \cong \coH^{A \times \Gm}_{\bullet - 2r}(Y)$ identifies with the action of the equivariant Euler class $\mathrm{Eu}(F) \in \coH_{A \times \Gm}^{2r}(Y)$ of $F$, see e.g.~\cite[\S 1.19]{LuCus2}. Since $Y$ is smooth, the equivariant homology $\coH^{A \times \Gm}_\bullet(Y)$ is a free module of rank one over $\coH_{A \times \Gm}^\bullet(Y)$, hence it is enough to prove that $\mathrm{Eu}(F)$ is not a zero-divisor in $\coH_{A \times \Gm}^\bullet(Y)$. However one can check that (due to our choice of $\Gm$-action) this Euler class can be written, using isomorphism \eqref{eqn:cohomology-Gm}, as
\[
\mathrm{Eu}(F) = 1 \otimes (-2u)^r + x
\]
where $1 \in \coH^0_{A}(Y)$ is the unit, $u \in \coH^2_{\Gm}(\pt)$ is the canonical generator and $x \in \bigoplus_{i \geq 2} \coH_{A}^i(Y) \otimes \coH_{\Gm}^{2r-i}(\pt)$. It follows that this element is indeed not a zero-divisor.
\end{proof}

\end{document}